\documentclass[reqno]{amsart}

\makeatletter
\@namedef{subjclassname@2020}{\textup{2020} Mathematics Subject Classification}
\makeatother

\usepackage{amssymb}
\usepackage{graphicx}
\usepackage{mathrsfs}
\usepackage[colorlinks=true, linkcolor=blue, urlcolor=blue, citecolor=blue]{hyperref}
\usepackage{color}
\usepackage{url}
\usepackage{cite}

\usepackage{times}
\usepackage{microtype}
\usepackage{amssymb}
\usepackage{mathtools}
\usepackage{tikz}
\usepackage{wasysym}
\usepackage{centernot}
\usepackage{color}
\usepackage{setspace}
\usepackage{appendix}
\usepackage{booktabs}
\usepackage{multirow}

\newtheorem{theorem}{Theorem}[section]
\newtheorem{lemma}[theorem]{Lemma}
\newtheorem{proposition}[theorem]{Proposition}
\newtheorem{corollary}[theorem]{Corollary}

\theoremstyle{definition}
\newtheorem{definition}[theorem]{Definition}

\theoremstyle{remark}
\newtheorem{remark}[theorem]{Remark}

\numberwithin{equation}{section}

\def\dd{\mathrm{d}}

\def\bP{\mathbf{P}}

\def\bR{{\mathbb R}}

\def\sE{{\mathscr E}}
\def\sF{{\mathscr F}}
\def\sA{{\mathscr A}}
\def\sG{{\mathscr G}}

\def\sL{\mathscr{L}}
\def\cD{\mathcal{D}}

\def\bE{\mathbf{E}}

\def\bN{\mathbb{N}}

\def\sH{\mathscr H}

\begin{document}

\title[BD processes are time-changed Feller's Brownian motions]{Birth-death processes are time-changed Feller's Brownian motions}

\author{ Liping Li}
\address{Fudan University, Shanghai, China.  }
\email{liliping@fudan.edu.cn}
\thanks{The author is a member of LMNS,  Fudan University.  He is partially supported by NSFC (No. 11931004 and 12371144). }




\date{\today}




\begin{abstract}
A Feller's Brownian motion refers to a Feller process on the interval $[0, \infty)$ that is equivalent to the {\color{blue}killed} Brownian motion before reaching $0$. It is fully determined by four parameters $(p_1, p_2, p_3, p_4)$, reflecting its killing, reflecting, sojourn, and jumping behaviors at the boundary $0$. On the other hand, a birth-death process is a continuous-time Markov chain on $\mathbb{N}$ with a given {\color{blue}birth-death $Q$-matrix}, and it is characterized by three parameters $(\gamma, \beta, \nu)$ that describe its killing, reflecting, and jumping behaviors at the boundary $\infty$. The primary objective of this paper is to establish a connection between Feller's Brownian motion and birth-death process. We will demonstrate that any Feller's Brownian motion can be transformed into a specific birth-death process through a unique time change transformation, and conversely, any birth-death process can be derived from Feller's Brownian motion via time change. Specifically, the birth-death process generated by the Feller’s Brownian motion, determined by the parameters $(p_1, p_2, p_3, p_4)$, through time change, has the parameters:
\[
\gamma = p_1, \quad \beta = 2p_2, \quad \nu_n = \mathfrak{p}_n, \; n \in \mathbb{N},
\]
where $\{\mathfrak{p}_n : n \in \mathbb{N}\}$ is a sequence derived by allocating weights to the measure  $p_4$  in a specific manner. Utilizing the pathwise representation of Feller's Brownian motion, our results provide a pathwise construction scheme for birth-death processes,  addressing a gap in the existing literature.
\end{abstract}

\maketitle

\tableofcontents

\section{Introduction}

To study the boundary behavior of Markov processes, historically, two classical probabilistic models have been developed that seem different but are actually similar. The first model is based on Brownian motion on the half-line $(0, \infty)$, allowing all possible behaviors of the process at the boundary $0$ while ensuring the strong Markov property. This line of work began with Feller's research on the theory of transition semigroup \cite{F52} and was later extended by It\^o and McKean, who completed the pathwise construction of all such processes in \cite{IM63}. These processes are referred to as \emph{Feller's Brownian motion} by It\^o and McKean, with a detailed definition provided in \cite[\S5]{IM63} (see also Definition~\ref{DEF31}).  From an analytical perspective, the key to determining the domain of the generator of Feller's Brownian motion is the boundary condition at $0$:
\begin{equation}\label{eq:12}
p_1f(0) - p_2f'(0) + \frac{p_3}{2}f''(0) + \int_{(0,\infty)} \left(f(0) - f(x)\right) p_4(\dd x) = 0,
\end{equation}
where $p_1, p_2, p_3 \geq 0$ are constants, and $p_4$ is a positive measure on $(0, \infty)$ satisfying $\int_{(0,\infty)} (1 \wedge x) p_4(\dd x) < \infty$ {\color{blue}($p_4$ may be a null measure)}. From a probabilistic perspective, these four parameters $(p_1, p_2, p_3, p_4)$ represent the four possible types of boundary behavior of Feller’s Brownian motion at $0$: \emph{killing}, \emph{reflecting}, \emph{sojourn}, and \emph{jumping}. The work of It\^o and McKean \cite{IM63} deeply explores the mechanisms underlying these four types of boundary behavior, providing a clear and intuitive pathwise representation of Feller's Brownian motion; see \S\ref{SEC32} for a brief summarization.  

The other type of model is a fundamental one in the theory of continuous-time Markov chains, known as the \emph{birth-death process}. This process is induced by a standard transition matrix, whose derivative at time $0$ is the {\color{blue}$Q$-matrix} given by \eqref{eq:11},  on the discrete space $\mathbb{N}$. Its defining characteristic is that from any given point, the process can only jump to its immediate left or right neighbors, a feature akin to the continuity of Brownian motion trajectories. Since the birth-death process may reach $\infty$ in finite time, describing its behavior at $\infty$ and after reaching $\infty$ is a fundamental problem in the theory of birth-death processes. This line of inquiry also began with Feller, who in \cite{F59} observed that the boundary behavior of the birth-death process at $\infty$ bears some intrinsic similarity to the boundary behavior of the aforementioned Feller's Brownian motion at $0$. Feller's approach was analytical; although he only examined partial birth-death processes, he derived the boundary condition at $\infty$ for the resolvent function $F$ (i.e., the function in the domain of the generator) as follows:
\begin{equation}\label{eq:15}
\gamma F(\infty) +\frac{\beta}{2} F^+(\infty) + \sum_{k \in \mathbb{N}} (F(\infty) - F(k)) \nu_k = 0,
\end{equation}
where $F^+$ denotes the discrete gradient (see \S\ref{SEC23}), $\gamma, \beta \geq 0$ are constants, and $\nu = (\nu_k)_{k \in \mathbb{N}}$ is a positive measure {\color{blue}(or a null measure)} on $\mathbb{N}$,  {\color{blue}where $\bN=\{0,1,\cdots\}$ denotes the set of all natural numbers, including zero}. Shortly thereafter, building on Feller's work, Yang in 1965 (see \cite[Chapter 7]{WY92}) completed the analytical construction of all birth-death processes using the resolvent method, thereby finalizing the construction theory from an analytical perspective. In our recent article \cite{L23}, we proved that all (non-Doob) birth-death processes satisfy the boundary condition \eqref{eq:15} observed by Feller.
{\color{blue}Alternatively, a probabilistic construction of all birth-death processes was given by Wang in his 1958 doctoral thesis  (see \cite[Chapter 6]{WY92}).
}
It should be noted that this probabilistic construction differs from the pathwise construction of Feller's Brownian motion by It\^o and McKean, as it instead uses a series of Doob processes to approximate the original birth-death process. In fact, understanding the probabilistic intuition behind the parameters $(\gamma, \beta,  \nu)$ through this construction is challenging,  although, by comparing with \eqref{eq:12}, it becomes evident that these parameters should correspond to the killing, reflecting, and jumping behaviors of the birth-death process at $\infty$, respectively.

Apart from two minor differences, the structural consistency of the boundary conditions \eqref{eq:12} and \eqref{eq:15} is quite evident: First, the reflecting term in \eqref{eq:12} is negative, while in \eqref{eq:15} it is positive; this is because $0$ and $\infty$ are at opposite ends of their respective state spaces, thus causing the direction of the gradient to be exactly opposite. Second, \eqref{eq:15} lacks a sojourn term; this is because in the definition of the birth-death process, the index set of the transition matrix is $\mathbb{N}$, which causes that the birth-death process can not visit $\infty$ for a positive duration (correspondingly, the process that allows sojourn at $\infty$ is called a \emph{generalized birth-death process}, and its index set is $\mathbb{\bN}\cup\{\infty\}$; see \cite{F59}). Naturally, we hope to find a deeper connection between the two beyond methodology.

For the symmetric case (from an analytical perspective, that is, when the transition semigroup satisfies both the Kolmogorov backward and forward equations; see \cite{F59}), there is no jumping behavior at the boundary for either process, and the boundary conditions reduce to the classic \emph{Dirichlet}, \emph{Neumann}, or \emph{Robin} boundaries (see \cite[\S2.3]{L23}). At this point, Feller's Brownian motion and birth-death process are both special cases of a more general class of processes known as \emph{quasidiffusions}. These processes are also referred to as \emph{generalized diffusions} or \emph{gap diffusions} in some literature, such as \cite{KW82, K81}. The study of quasidiffusions originated from Kac and Krein's research on the spectral theory of a class of generalized second-order differential operators \cite{KK74}. Soon after, these self-adjoint operators were realized as strong Markov processes on some closed subset of $[-\infty, \infty]$ with trajectories that satisfy the \emph{skip-free property}; see \cite{K75}. The skip-free property is a combination of the continuity of Brownian motion trajectories and the characteristic of birth-death process trajectories: the trajectory is continuous where the space is continuous and can only jump to adjacent positions where there is a gap in space. In addition to studying quasidiffusions within the classic Feller framework (see \cite{S79}), with the aid of symmetry, we can also use the Dirichlet form theory established by Fukushima (see \cite{CF12,  FOT11}) to study them, as demonstrated in several recent articles \cite{L23b, LL23, L24b}. In general, the connection between symmetric Feller's Brownian motions and symmetric birth-death processes is very clear: they can not only be unified under the more general quasidiffusion but also, the symmetric birth-death process is a time-changed process of a certain symmetric Feller's Brownian motion. For example, the minimal birth-death process corresponds to the time-changed process of {\color{blue}killed} Brownian motion, and the $(Q,1)$-process corresponds to that of reflecting Brownian motion; see \cite[\S3]{L23}.

However, for the non-symmetric case (especially the so-called `pathological' case by Feller, where $p_4$ or $\nu$ is an infinite measure), the connection between Feller's Brownian motion and birth-death process becomes very vague. Indeed, both  theories have been extensively studied and developed within their respective fields. The former has given rise to a mature theory of diffusion processes, see, e.g., \cite{IM74, M68}, which is one of the most important classes of models in the general theory of Markov processes; the latter,  originating from the discrete space, has also evolved a series of theories, such as continuous-time Markov chains  (see,  e.g., \cite{C67}) and Markov jump processes (see, e.g., \cite{C04}).  However, during their parallel development, the connection between the two is more reflected in the mutual borrowing of research methods, and few people discuss the fit as reflected in the boundary conditions \eqref{eq:12} and \eqref{eq:15}.

The goal of our paper is to construct the missing bridge between Feller's Brownian motion and birth-death process. In summary, our main result is that any Feller's Brownian motion can be transformed into a specific birth-death process through a unique time change transformation; conversely, any birth-death process can be derived through the time change transformation of some Feller's Brownian motion.

To facilitate the explanation of our results, we first perform a spatial transformation on the birth-death process as described in \S\ref{SEC41}, aligning it with the natural scale of Feller's Brownian motion and reflecting the boundary point $\infty$ so that it moves to $0$, thereby matching the boundary point of Feller's Brownian motion. The transformed space, {\color{blue}which depends
on the matrix $Q$ through the scale function \eqref{eq:22}}, is denoted by (see \eqref{eq:419})
\[
\overline{E} = \{\hat c_n: n \in \mathbb{N}\} \cup \{0\},
\]
where the subscript $n$ of $\hat c_n$ represents the point before the transformation, and $\lim_{n \to \infty} \hat c_n = 0$ is topologically consistent with the one-point compactification $\mathbb{N} \cup \{\infty\}$ of $\bN$. Under this spatial transformation, aside from the change in the symbol of the state points, the transformed process, which we denote by $\hat X$, remains indistinguishable from the original birth-death process. For convenience, we refer to $\hat X$ as a birth-death process on $\overline{E}$. Additionally, due to the change in the direction of the gradient at the boundary point, the corresponding boundary condition for $\hat X$ is modified to:
\begin{equation}\label{eq:13}
\gamma \hat F(0) - \frac{\beta}{2} \hat F^+(0) + \sum_{k \in \mathbb{N}} (\hat F(0) - \hat F(\hat c_k)) \nu_k = 0,
\end{equation}
where $\hat{F}$ is the transformed function of $F$ in \eqref{eq:15}.
After this modification, the aforementioned boundary condition closely resembles \eqref{eq:12}.

{\color{blue}Our main results, Theorems~\ref{THM42} and \ref{MainTHM}, show that the Feller's Brownian motion  $Y$, subject to the boundary condition \eqref{eq:12}, can always be transformed into a birth-death process on  $\overline{E}$ corresponding to the boundary condition \eqref{eq:13} with parameters 
\begin{equation}\label{eq:14}
\gamma = p_1, \quad \beta = 2p_2, \quad \nu_n = \mathfrak{p}_n, \; n \in \mathbb{N},
\end{equation}
where $\{\mathfrak{p}_n : n \in \mathbb{N}\}$  is the sequence obtained by allocating weights to the measure  $p_4$  in the manner of \eqref{eq:55} and \eqref{eq:56}, via a time change transformation.  
 This time change transformation is uniquely determined in the sense of Theorem~\ref{THM44} and
 induced by a positive continuous additive functional  given by \eqref{eq:41}, which is the integral of  the local times of  $Y$ with respect to the speed measure  $\mu$ determined by the matrix $Q$. It should be noted that the sojourn parameter $p_3$ of  $Y$ does not play any role in the expression \eqref{eq:14}. This is because, to obtain a birth-death process on $\overline{E}$ that experiences no sojourn at the boundary $0$, the time change transformation effectively ``bypasses" the sojourn part of $Y$, resulting in the absence of $p_3$ in \eqref{eq:14}. We will elaborate on this in \S\ref{SEC82}.}
 
 

 From the relationship between the two sets of parameters in \eqref{eq:14}, it is not difficult to further deduce that every birth-death process on  $\overline{E}$  can be obtained by time change from a certain Feller's Brownian motion; see Corollary~\ref{COR66}. This conclusion also provides a pathwise construction for all birth-death processes: the trajectory of  $\hat X$  is the \emph{trace} of the corresponding Feller's Brownian motion on  $\overline{E}$; see \S\ref{SEC52}.

Let us introduce some of the tools that will be used in the proof of our main results. 
For the symmetric case, we primarily employ the theory of Dirichlet forms. The concepts and notation we use are consistent with the foundational references \cite{CF12, FOT11}, so they will not be further elaborated upon in the text.
For the non-symmetric case where $p_4$ is a finite measure, it essentially involves the Ikeda-Nagasawa-Watanabe piecing out transformation (see \cite{INW66}) of the symmetric case. This is very similar to the discussion of birth-death processes found in \cite[\S5 and \S8]{L23}. Based on this piecing out transformation, we can directly construct the trajectories of the birth-death process without relying on Feller's Brownian motion; see \S\ref{SEC6}.
The most challenging scenario arises when $p_4$ is an infinite measure. In this case, the piecing out transformation becomes ineffective, and the jumps of Feller's Brownian motion at the boundary become very frequent and complex. Notably, our previous study \cite{L23} also regrettably stops short of studying this case for the birth-death process. To address this situation, we draw on Wang's idea of constructing an approximating sequence for the birth-death process (see \cite{WY92}), and we similarly construct a sequence of approximating processes for Feller's Brownian motion (see \S\ref{SEC7}). The boundary behavior of these processes is relatively simple, and they can be linked to the approximating sequence of the corresponding birth-death process, thereby ultimately establishing the parameter relationship \eqref{eq:14} between the target processes.

At the end of this section, we briefly describe the commonly used notation. The symbols contained in the Markov process  $Y = (\Omega, \mathscr{G}, \mathscr{G}_t, Y_t, \theta_t, \mathbf{P}_x)$  are standard, as seen in \cite{BG68, S88}. For brevity, some elements are sometimes omitted, such as writing  $Y = (\Omega, Y_t, \mathbf{P}_x)$. The expectation corresponding to the probability measure  $\mathbf{P}_x$  is denoted by  $\mathbf{E}_x$.  In this paper, the cemetery point for Markov processes is uniformly denoted by  $\partial$.  Given a topological space $E$,  $\mathcal{B}(E), \mathcal{B}_+(E),  \mathcal{B}_b(E),  C(E)$ and  $C_b(E)$  represent the space of all Borel measurable, non-negative Borel measurable, bounded Borel measurable functions,  totality of continuous functions, and bounded continuous functions, respectively. If  $E$  is an interval,  $C_c(E)$  and  $C_0(E)$  denote the set of all continuous functions with compact support and the set of all continuous functions that equal  $0$  at the open endpoints, respectively. The notation for continuous functions with superscripts indicates differentiability, for example,  $C^2$  indicates twice continuously differentiable, while  $C^\infty$  indicates infinitely continuously differentiable.  The symbol $\delta_a$ stands for the Dirac measure at $a\in [0,\infty)$. 
 
 \section{Birth-death processes}
 
In this paper, we will strive to use the same terminologies and symbols as in \cite{L23} for the birth-death processes. The following restates only some of the important content.
 
 \subsection{Elements of $Q$-processes}
 
We consider a birth-death {\color{blue}$Q$-matrix} as follows:
\begin{equation}\label{eq:11}
Q=(q_{ij})_{i,j\in \bN}:=\left(\begin{array}{ccccc}
-q_0 &  b_0 & 0 & 0 &\cdots \\
a_1 & -q_1 & b_1 & 0 & \cdots \\
0 & a_2 & -q_2 & b_2 & \cdots \\
\cdots & \cdots &\cdots &\cdots &\cdots 
\end{array}  \right),
\end{equation}
{where $a_k>0$ for $k \geq 1$ and $b_k>0,  q_k=a_k+b_k$ for $k\geq 0$.} (Set $a_0=0$ for convenience.) A continuous-time Markov chain $X=(X_t)_{t\geq 0}$ is called a \emph{birth-death $Q$-process} (or \emph{$Q$-process} for short) if its transition matrix $(p_{ij}(t))_{i,j\in \bN}$ is \emph{standard} and its {\color{blue}derivative at time $0$} is $Q$,  i.e., $p'_{ij}(0)=q_{ij}$ for $i,j\in \bN$.  Readers are referred to \cite{C67, WY92} for the terminologies concerning continuous-time Markov chains; see also \cite{L23}.  In our consideration, two  $Q$-processes with the same transition matrix will not be distinguished. For convenience, we will also refer to such  $(p_{ij}(t))_{i,j\in \bN}$ as a $Q$-process when there is no risk of confusion.


There always exists at least one $Q$-process $X^\text{min}$, known as the \emph{minimal $Q$-process}. This process is killed at the first time it almost reaches $\infty$. Similar to a regular diffusion on an interval, the minimal $Q$-process {\color{blue}has a \emph{scale function} on $\mathbb{N}$:}
\begin{equation}\label{eq:22}
	c_0=0,  \quad c_1=\frac{1}{2b_0}, \quad c_k=\frac{1}{2b_0}+\sum_{i=2}^k \frac{a_1a_2\cdots a_{i-1}}{2b_0b_1\cdots b_{i-1}}, \;k\geq 2
\end{equation}
and a \emph{speed measure} $\mu$ on $\mathbb{N}$:
\begin{equation}\label{eq:03}
\mu(\{0\}):=\mu_0=1,  \quad \mu(\{k\}):=\mu_k=\frac{b_0b_1\cdots b_{k-1}}{a_1a_2\cdots a_k}, \; k\geq 1.
\end{equation}
{\color{blue}
These two quantities determine the Dirichlet form associated with the minimal $Q$-process (see \cite[\S3.1]{L23} for details).  A key identity for the scale function is
\begin{equation}\label{eq:231}
	\bP^\text{min}_k(\sigma^\text{min}_{k+1}<\sigma^\text{min}_{k-1})=\frac{c_k-c_{k-1}}{c_{k+1}-c_{k-1}},\quad k\geq 1,
\end{equation}
where $\bP^\text{min}_k$ denotes the law of  $X^\text{min}$ starting from $k$, and $\sigma^\text{min}_k:=\inf\{t>0:X^\text{min}_t=k\}$ for $k\in \bN$. 
Moreover, the condition $c_\infty:=\lim_{k\rightarrow \infty}c_k<\infty$ is equivalent to the transience of $X^\text{min}$.}
The transition matrix $(p^\text{min}_{ij}(t))$ of $X^\text{min}$ is symmetric with respect to $\mu$ in the sense that $\mu_ip_{ij}^\text{min}(t)=\mu_jp_{ji}^\text{min}(t)$ for all $i,j\in \bN$ and $t\geq 0$.  {\color{blue}In some texts, $\mu$ is also called the (sub-)invariant measure.} 

{\color{blue}
Note that, conversely, the coefficients $a_k$ and $b_k$ can be recovered from the scale function $c$ and the speed measure $\mu$. To proceed, let us introduce two auxiliary quantities:
\[
	R:=\sum_{k=0}^\infty (c_{k+1}-c_k)\cdot \sum_{i=0}^k \mu_i,\quad S:=\sum_{k=0}^\infty c_k\mu_k.
\]
In Feller's classification, the boundary point $\infty$ (for $X^{\text{min}}$) is said to be \emph{regular} if $R < \infty$ and $S < \infty$, or an \emph{exit} if $R < \infty$ and $S = \infty$. We refer the reader to \cite[Definition 3.3]{L23} for the full classification of the boundary point $\infty$ in the sense of Feller.

The minimal $Q$-process is the unique $Q$-process when the boundary point $\infty$ is neither regular nor an exit. In contrast, when $\infty$ is regular or an exit, we have $c_\infty < \infty$, and multiple $Q$-processes may exist beyond the minimal one, including Doob processes and the $(Q,1)$-process.  A Doob process modifies $X^{\text{min}}$ by choosing a return state according to a prescribed distribution each time the process is about to reach $\infty$, then restarting its evolution from that state. This can be formally realized via a piecing-out transformation, as described in \cite[\S5]{L23}.
The $(Q,1)$-process, which is defined only in the regular case, resembles reflecting Brownian motion: it follows the minimal $Q$-process until reaching $\infty$, where it instantaneously reflects back into $\mathbb{N}$ and resumes its evolution. See \cite[\S3.3]{L23} for further details.
}

\subsection{Resolvent representation of $Q$-processes}\label{SEC22}

{\color{blue}Assume that $\infty$ is regular or an exit.}
In what follows,  let us present a well-known analytic approach to characterize all birth-death processes by solving their resolvents.  
For $\alpha>0$ and $i\in \bN$,  define 
\[
	u_\alpha(i)=\mathbf{E}_i^\text{min} e^{-\alpha \zeta^{\text{min}}},
\]
where $\bE^\text{min}_i$ stands for the expectation of $X^\text{min}$ starting from $i$ and $\zeta^{\text{min}}$ is the lifetime of $X^\text{min}$. 
Denote by $(R^\text{min}_\alpha)_{\alpha>0}$ the resolvent of the minimal $Q$-process.  Set
\[
\Phi_{ij}(\alpha):=R^\text{min}_\alpha(i, \{j\}),\quad \alpha>0,  i,j\in \bN. 
\] 
Let $\nu$ be a positive measure on $\bN$ and let $\gamma, \beta\geq 0$ be two constants.  Set $|\nu|:=\sum_{k\geq 1}\nu_k$.  {When both 
\begin{equation}\label{eq:B1}
	\sum_{k\geq 0} \nu_k \left(\sum_{j=k}^\infty (c_{j+1}-c_j) \sum_{i=0}^j \mu_i \right)<\infty,\quad |\nu| +\beta\neq 0, 
\end{equation}
and
\begin{equation}\label{eq:B4}
\beta=0,\quad \text{if }\infty\text{ is an exit}
\end{equation}
are satisfied},  define, for $\alpha>0$ and $i,j\in \bN$,
\begin{equation}\label{eq:B2}
	\Psi_{ij}(\alpha):=\Phi_{ij}(\alpha)+u_\alpha(i) \frac{\sum_{k\geq 0} \nu_k \Phi_{kj}(\alpha)+\beta\mu_j u_\alpha(j)}{\gamma+\sum_{k\geq 0}\nu_k(1-u_\alpha(k))+\beta\alpha\sum_{k\geq 0}\mu_k u_\alpha(k)}.  
\end{equation}
The matrix $(\Psi_{ij}(\alpha))_{i,j\in \bN}$ is called the \emph{$(Q,  \gamma, \beta, \nu)$-resolvent matrix. } For a constant $M>0$,  the $(Q,  M\gamma, M\beta, M\nu)$-resolvent matrix is obviously the same as the $(Q, \gamma,\beta, \nu)$-resolvent matrix.  

The following theorem is attributed to \cite[\S7.6]{WY92}.  Note that a $Q$-process is called \emph{honest} if its transition semigroup $(p_{ij}(t))$ satisfies $\sum_{j\in \bN}p_{ij}(t)=1$ for all $i\in \bN$ and $t\geq 0$.

\begin{theorem}\label{THMB1}
The transition matrix $(p_{ij}(t))_{i,j\in \bN}$ is a $Q$-process that is not the minimal one,  if and only if there exists a unique (up to a multiplicative positive constant) triple $(\nu, \gamma, \beta)\geq 0$ with \eqref{eq:B1} and \eqref{eq:B4} such that the resolvent of $(p_{ij}(t))_{i,j\in \bN}$ is given by
\[
	R_\alpha(i,\{j\})=\Psi_{ij}(\alpha),\quad \alpha>0,  i,j\in \bN,
\]
where $(\Psi_{ij}(\alpha))_{i,j\in \bN}$ is the $(Q, \gamma, \beta, \nu)$-resolvent matrix.  Furthermore,
\begin{itemize}
\item[(1)] $(p_{ij}(t))_{i,j\in \bN}$ is honest,  if and only if $\gamma=0$.  
\item[(2)] $(p_{ij}(t))_{i,j\in \bN}$ is a Doob process,  if and only if $0<|\nu|<\infty$ and $\beta=0$.
\end{itemize}
\end{theorem}


\subsection{Ray-Knight compactification of $Q$-processes}\label{SEC23}

In a previous study \cite{L23}, we obtained a c\`adl\`ag modification for each $Q$-process $X$, denoted by $\bar X=(\bar X_t)_{t\geq 0}$, using the Ray-Knight compactification.  The state space of $\bar X$ is $\mathbb N\cup \{\infty\}$, which corresponds to the Alexandroff compactification of $\mathbb{N}$.  It is important to note that the Ray-Knight compactifications of Doob processes are not normal, whereas those of other $Q$-processes are Feller processes on $\mathbb N\cup \{\infty\}$; see \cite[Corollary~5.2]{L23}.   If no ambiguity arises, we will not distinguish between the $Q$-process and its Ray-Knight compactification. Additionally, we will refer to a non-minimal and non-Doob $Q$-process as a \emph{Feller $Q$-process}.

Regarding a Feller $Q$-process $X$,  we can further derive its infinitesimal generator by utilizing the resolvent representation in Theorem~\ref{THMB1}; see \cite[Theorem~6.3]{L23}.  The crucial fact is that every function $F$ in the domain of the infinitesimal generator satisfies the following boundary condition at $\infty$:
\begin{equation}\label{eq:27}
	\frac{\beta}{2} F^+(\infty)+\sum_{k\in \mathbb N}(F(\infty)-F(k))\nu_k+\gamma F(\infty)=0,
\end{equation}
where $F(\infty):=\lim_{k\rightarrow \infty}F(k)$ (if it exists),  $F^+(k):=(F(k+1)-F(k))/(c_{k+1}-c_k)$ for $k\in \bN$ and $F^+(\infty):=\lim_{k\rightarrow \infty}F^+(k)$ (if it exists).  
(When deriving the boundary condition \eqref{eq:27} in \cite{L23}, we mistakenly wrote the parameter $\frac{\beta}{2}$ as $\beta$. This error occurred because the scale function used in \cite{L23}, specifically \eqref{eq:22}, is half of the scale function in \cite{F59}. However, in the proof of \cite[Theorem~6.3]{L23}, when citing the result from \cite{F59}, particularly the equation above \cite[(6.8)]{L23}, we forgot to multiply by two accordingly. As a result, the parameter $\frac{\beta}{2}$ in \cite[(6.10)]{L23} was incorrectly written as $\beta$. Unfortunately, this mistake was not corrected before the publication of \cite{L23}.)

\section{Feller's Brownian motions}

The terminology of Feller's Brownian motions is borrowed from the renowned article  \cite{IM63} by It\^o and McKean.  This class of Markov processes was first discovered by Feller in \cite{F52}, with its detailed definition found in \cite[\S5]{IM63}.   Recall that a Markov process on $[0,\infty)$ is called a Feller process if its semigroup $(T_t)_{t\geq 0}$ satisfies the following conditions: $T_0$ is the identify mapping, $T_tC_0([0,\infty))\subset C_0([0,\infty))$ and $T_tf\rightarrow f$ in $C_0([0,\infty))$ as $t\rightarrow 0$ for any $f\in C_0([0,\infty))$,  where $C_0([0,\infty))$ is the Banach space consisting of all continuous functions $f$ on $[0,\infty)$ such that $\lim_{x\rightarrow \infty}f(x)=0$.  We provide a definition for Feller's Brownian motions in a modern manner as follows. 

\begin{definition}\label{DEF31}
A Feller process $Y:=(\Omega, \sG,\sG_t,Y_t, \theta_t, (\mathbf{P}_x)_{x\in [0,\infty)})$ with lifetime $\zeta$ on $[0,\infty)$  is called a Feller's Brownian motion if its killed process upon hitting $0$ is identical in law to the {\color{blue}killed} Brownian motion on $(0,\infty)$. 
\end{definition}

\begin{remark}\label{RM32}
This definition excludes the special cases discussed in \cite[\S6]{IM63}.  The `Brownian motions' in these special cases do not satisfy the \emph{normal property} at $0$,  i.e.,  $\bP_0(Y_0=0)\neq 1$,  or the \emph{quasi-left-continuity} at the first time reaching $0$,  i.e.,  for $x>0$, $\bP_x(Y_{\lim_{n\rightarrow \infty}\sigma_n}=0)<1$, where $\sigma_n:=\inf\{t>0:Y_t=\varepsilon_n\}$ for any sequence $\varepsilon_n\downarrow 0$.  

The case 4a. in \cite[\S6]{IM63} {\color{blue}corresponding to $p_2=p_3=0, |p_4|<\infty$} warrants special mention. Similar to Doob processes in the context of $Q$-processes,  {\color{blue}the `Brownian motion' $Y$ for this case is} the piecing out of the {\color{blue}killed} Brownian motion $Y^0$ with lifetime $\zeta^0$ on $(0,\infty)$  with respect to the probability measure $\lambda$ on $(0,\infty)\cup\{\partial\}$.   This process $Y$ is a right process on $(0,\infty)$,  but not a Feller process.  {\color{blue}However, it admits an extension to a Ray process $\bar Y$ on $[0,\infty]$,  where the cemetery $\partial$ is identified with the compactification point $\infty$ of $[0,\infty)$}, and the initial transition function at $0$ is given by $\bar T_0(0,\cdot):=\bar\bP_0(\bar{Y}_0\in \cdot)=\lambda(\cdot)$.  
Analogous to \cite[Theorem~5.1]{L23}, the following expression for the resolvent of $\bar{Y}$ can be easily obtained:
\begin{equation}\label{eq:325}
	\bar{G}_\alpha h(x)=G^0_\alpha h(x)+\bE_x \left( e^{-\alpha \zeta^0}\right)\cdot \frac{\int_{(0,\infty)}G^0_\alpha h(x)\lambda(\dd x)}{1-\bE_\lambda \left( e^{-\alpha \zeta^0}\right)},\quad \forall h\in \mathcal{B}_b((0,\infty)),
\end{equation}
where $G^0_\alpha$ is the resolvent of $Y^0$ and  $\bE_\lambda\left( e^{-\alpha \zeta^0}\right):=\int_{(0,\infty)}\bE_x\left( e^{-\alpha \zeta^0}\right)\lambda(\dd x)$. {\color{blue}
For convenience, we refer to this `Brownian motion' $Y$ as a \emph{Doob's Brownian motion} with instantaneous distribution $\lambda$.}
\end{remark}

{\color{blue}Note that} the reflecting Brownian motion on $[0,\infty)$ is a classic example of a Feller's Brownian motion. {\color{blue}It is also important} to note that not all Feller's Brownian motions exhibit a.s. continuous sample paths. Some Feller's Brownian motions may experience jumps into $(0,\infty)$ from $0$ or just before reaching $0$; {\color{blue}see \S\ref{SEC32}}. 


\subsection{Infinitesimal generators}

It is well known that a function $f\in C_0([0,\infty))$ is said to belong to the domain $\cD(\sL)$ of the infinitesimal generator of $Y$ if the limit
\[
		\sL f:=\lim_{t\downarrow 0}\frac{T_t f-f}{t}
\]
exists in $C_0([0,\infty))$.  The operator $\sL: \cD(\sL)\rightarrow C_0([0,\infty))$ thus defined is called the \emph{infinitesimal generator} of the process $Y$.  
The following characterization of infinitesimal generators of Feller's Brownian motions is attributed to Feller \cite{F52}.  

\begin{theorem}\label{THM32}
A Markov process  $Y=(Y_t)_{t\geq 0}$ on $[0,\infty)$ is a Feller's Brownian motion,  if and only if it is a Feller process on $[0,\infty)$ whose infinitesimal generator on $C_0([0,\infty))$ is $\sL f=\frac{1}{2}f''$ with domain
\begin{equation}\label{eq:37}
\begin{aligned}
\cD(\sL)=&\bigg\{f\in C_0([0,\infty))\cap C^2({\color{blue}(}0,\infty)): f''\in C_0([0,\infty)), \\
	&\quad p_1f(0)-p_2f'(0)+\frac{p_3}{2}f''(0)+\int_{(0,\infty)}\left(f(0)-f(x)\right)p_4(\dd x)=0\bigg\},
\end{aligned}\end{equation}
where  $p_1,p_2,p_3$ are non-negative numbers and $p_4$ is a positive measure on $(0,\infty)$ such that
\begin{equation}\label{eq:330}
	p_1+p_2+p_3+\int_{(0,\infty)}(x\wedge 1)p_4(\dd x)=1
\end{equation}
and
\begin{equation}\label{eq:38}
	|p_4|:=p_4\left((0,\infty) \right)=+\infty \quad \text{if}\quad p_2=p_3=0.
\end{equation}
The parameters $p_1,p_2,p_3,p_4$ are uniquely determined for each Feller's Brownian motion.
\end{theorem}

{\color{blue}
For the reader' convenience, the necessary details of this proof are deferred to Appendix~\ref{APPB}. Here, we collect some facts about the resolvent $(G_\alpha)_{\alpha>0}$ of $Y$.  Fix $\alpha>0$. By the Hille-Yosida theorem,  we have $\cD(\sL)=G_\alpha C_0([0,\infty))\subset C_0([0,\infty))$, and for any $f=G_\alpha h$ with $h\in C_0([0,\infty))$, it holds that
\begin{equation}\label{eq:32}
\sL f=\alpha G_\alpha h-h.  
\end{equation}
Set $\tau_0:=\inf\{t>0: Y_t=0\}$.  It follows from Definition~\ref{DEF31} and Dynkin's formula that
\begin{equation}\label{eq:33}
	G_\alpha h(x)=G_\alpha^0 h(x)+G_\alpha h(0)\mathbf{E}_x e^{-\alpha \tau_0},\quad x>0,
\end{equation}
where $G^0_\alpha$ denotes the resolvent of the {\color{blue}killed} Brownian motion $Y^0$  on $(0,\infty)$.  It is well known that
\begin{equation}\label{eq:362}
	\mathbf{E}_x e^{-\alpha \tau_0}=e^{-\sqrt{2\alpha}x},
\end{equation}
(see, e.g., \cite[page 26, 5)]{IM74}), and the resolvent density $g^0_\alpha(x,y)$ of $Y^0$,  i.e., $G^0_\alpha(x,\dd y)=g^0_\alpha(x,y)\dd y$ is given by
\begin{equation}\label{eq:310}
g_\alpha^0(x,y)=\left\lbrace
\begin{aligned}
	&W^{-1}u_-(x)u_+(y),\quad  &x\leq y,\\
	&W^{-1}u_-(y)u_+(x),\quad &y\leq x,
\end{aligned} \right.
\end{equation}
where $u_-(x)=\sinh (\sqrt{2\alpha}x), u_+(x)=e^{-\sqrt{2\alpha}x}$, and $W:=W(u_-,u_+)=\frac{\sqrt{2\alpha}}{2}$ is the Wronskian of $u_-,u_+$ (with $W^{-1}:=1/W$); see, e.g., \cite[II\S3, \#7]{M68}. 
 As established in \cite[\S15, 4.]{IM63},  for any $h\in C_0([0,\infty))$,  one has
\begin{equation}\label{eq:39}
	G_\alpha h(0)=\frac{2p_2\int_{(0,\infty)}e^{-\sqrt{2\alpha}x}h(x)\dd x+p_3h(0)+\int_{(0,\infty)}G^0_\alpha h(x)p_4(\dd x)}{p_1+\sqrt{2\alpha}p_2+\alpha p_3+\int_{(0,\infty)}\left(1-e^{-\sqrt{2\alpha}x}\right)p_4(\dd x)}.
\end{equation}
Taken together, equations \eqref{eq:362}, \eqref{eq:310}, and \eqref{eq:39} provide an explicit expression for $G_\alpha h$.  

It is worth emphasizing that in \cite{IM63}, the parameter tuple $(p_1, p_2, p_3, p_4)$ associated with the process $Y$ is obtained by evaluating the expression \eqref{eq:39}. Moreover, the boundary condition \eqref{eq:12}, satisfied by $f = G_\alpha h \in \mathcal{D}(\mathscr{L})$, is equivalent to this expression for $G_\alpha h(0)$. We briefly explain this equivalence below.  In fact, for any $h\in C_0([0,\infty))$,  we have 
\[
	G^0_\alpha h\in C^2((0,\infty))\cap C_0((0,\infty)).
\]
Moreover, its derivative and second order derivative can be expressed as
\[
	W\cdot \left(G^0_\alpha h\right)'(x)=u'_+(x)\int_0^xu_-(y)h(y)\dd y+u'_-(x)\int_x^\infty u_+(y)h(y)\dd y
\]
and
\begin{equation}\label{eq:31}
\frac{1}{2}\left(G^0_\alpha h\right)''=\alpha G^0_\alpha h-h.  
\end{equation}
(See,  e.g., \cite[II\S3, \#7 and \#8]{M68}.)
In particular, this implies that
\begin{equation}\label{eq:315}
 \left(G^0_\alpha h\right)'(0):=\lim_{x\rightarrow 0}\left(G^0_\alpha h\right)'(x)=2\int_0^\infty h(y)e^{-\sqrt{2\alpha}y}\dd y.
\end{equation}
Combining this with \eqref{eq:33}, we obtain
\begin{equation}\label{eq:3112}
	(G_\alpha h)'(0)=2\int_0^\infty h(y)e^{-\sqrt{2\alpha}y}\dd y-\sqrt{2\alpha} G_\alpha h(0).
\end{equation}
In addition,  we have
\begin{equation}\label{eq:34}
\begin{aligned}
	\left(G_\alpha h\right)''(x)&=\left(G_\alpha^0h\right )'' (x)+G_\alpha h(0)\cdot 2\alpha\cdot e^{-\sqrt{2\alpha}x}\\
	&=2\alpha G^0_\alpha h(x)-2h(x)+2\alpha G_\alpha h(0)e^{-\sqrt{2\alpha}x},
\end{aligned}\end{equation}
and hence
\begin{equation}\label{eq:35}
	\left(G_\alpha h\right)''(0):=\lim_{x\rightarrow 0}\left(G_\alpha h\right)''(x)=-2h(0)+2\alpha G_\alpha h(0). 
\end{equation}
Eventually, with the aid of \eqref{eq:33}, \eqref{eq:3112}, and \eqref{eq:35}, we can establish the equivalence between the expression \eqref{eq:39} for $G_\alpha h(0)$ and the boundary condition \eqref{eq:12}.
}

\begin{remark}
When $p_2=p_3=|p_4|=0$ and $p_1=1$,  the condition \eqref{eq:38} is not satisfied.  Nevertheless, this case corresponds to the {\color{blue}killed} Brownian motion on $(0,\infty)$,  {\color{blue}since \eqref{eq:39} evaluates to zero and $G_\alpha$ coincides with $G^0_\alpha$ in this setting.} 
\end{remark}

The triple $(p_1,p_2, p_4)$ plays a role analogous to $(\gamma, \beta,\nu)$ for a $Q$-process,  with some heuristic explanations for the latter triple presented in  \cite[\S2]{L23}.  The additional parameter $p_3$ measures the \emph{sojourn} of the Feller's Brownian motion at $0$.  One one hand,  this can be observed by examining the paths of the Feller's Brownian motion as constructed by It\^o and McKean, as illustrated in \cite[\S15]{IM63}; see also \eqref{eq:321}.  On the other hand,  we can examine the symmetric case,  where the transition semigroup $(T_t)_{t\geq 0}$ of $Y$ is symmetric with respect to some $\sigma$-finite measure $m$ in the sense that 
\[
\int_{[0,\infty)}T_tf(x)g(x)m(\dd x)=\int_{[0,\infty)}f(x)T_tg(x)m(\dd x),\quad \forall t\geq 0, f,g\in \mathcal{B}_b([0,\infty)),
\]
as stated in the following corollary. In this special case, the parameter $p_3/(2p_2)$ represents the mass of the symmetric measure $m$ at $0$. It is well known that when this symmetric measure, also known as the \emph{speed measure} of $Y$, is larger in a specific region, the motion of $Y_t$ will be slower when passing through that region.

\begin{corollary}\label{COR34}
Let $Y=(Y_t)_{t\geq 0}$ be a Feller's Brownian motion on $[0,\infty)$ with parameters $p_1,p_2,p_3,p_4$ as described in Theorem~\ref{THM32}.  The process $Y$ is symmetric if and only if $p_2>0, |p_4|=0$.  Furthermore,  the symmetric measure must be
\begin{equation}\label{eq:312}
	m(\dd x)=\frac{p_3}{2p_2}\delta_0(\dd x)+1_{(0,\infty)}(x)\dd x
\end{equation}
up to a multiplicative constant, 
where $\delta_0$ denotes the Dirac measure at $0$, and the Dirichlet form associated with this symmetric Feller's Brownian motion on $L^2([0,\infty),m)$ is
\begin{equation}\label{eq:313}
\begin{aligned}
	\sF&=H^1([0,\infty)), \\
	\sE(f,g)&=\frac{1}{2}\int_0^\infty f'(x)g'(x)\dd x+\frac{p_1}{2p_2}f(0)g(0),\quad f,g\in \sF,
\end{aligned}\end{equation}
where 
\[
	H^1([0,\infty)):=\{f\in L^2([0,\infty)): f\text{ is absolutely continuous and }f'\in L^2([0,\infty))\}.
\] 
\end{corollary}
\begin{proof}
To demonstrate the necessity,  let $m$ be a symmetric measure of $Y$.   From \cite[Lemma~4.1.3]{FOT11},  it follows that $m|_{(0,\infty)}$ is the symmetric  measure of $Y^0$.  Hence, without loss of generality, we may assume that $m|_{(0,\infty)}$ is the Lebesgue measure on $(0,\infty)$.  For any $h_1,h_2\in C_0((0,\infty))$,  substituting \eqref{eq:33} and \eqref{eq:39} into
\begin{equation}\label{eq:311}
	\int_{[0,\infty)}G_\alpha h_1(x)h_2(x)m(\dd x)=\int_{[0,\infty)}h_1(x)G_\alpha h_2(x)m(\dd x),
\end{equation}
we obtain
\[
\begin{aligned}
&\iint_{(0,\infty)\times (0,\infty)}h_1(x)h_2(y)e^{-\sqrt{2\alpha} x}m(\dd x)\left(p_4G^0_\alpha\right)(\dd y) \\
&\quad=\iint_{(0,\infty)\times (0,\infty)}h_1(y)h_2(x)e^{-\sqrt{2\alpha} x}m(\dd x)\left(p_4G^0_\alpha\right)(\dd y),\quad \forall h_1,h_2\in C_0((0,\infty)),
\end{aligned}\]
where $\left(p_4G^0_\alpha\right)(\dd y)=\int_{(0,\infty)}p_4(\dd x)G_\alpha^0(x,\dd y)=\left( \int_{(0,\infty)}g^0_\alpha(x,y)p_4(\dd x)\right)\dd y$.  This implies that
\[
H(y):=e^{\sqrt{2\alpha}y}	\int_{(0,\infty)}g^0_\alpha(x,y)p_4(\dd x)
\]
is constant.  Substituting \eqref{eq:310} into $H(y)$,  we have
\[
W\cdot H(y)=\int_{(0,\infty)} H_y(x)p_4(\dd x),
\]
where 
\[
	H_y(x)=\left\lbrace
	\begin{aligned}
		&\sinh(\sqrt{2\alpha}x),\quad &x\leq y,\\
		&\sinh(\sqrt{2\alpha}y)e^{-\sqrt{2\alpha}(x-y)},  &y\leq x.
	\end{aligned} \right.
\]
It is straightforward to verify that $H_y(x)$ is increasing in $y$.  Letting $y\downarrow 0$,  we conclude that $H(y)\equiv 0$.  Since $H_y$ is strictly positive for $y>0$,  it must hold that $|p_4|=0$.  Substituting \eqref{eq:33} and \eqref{eq:39} into \eqref{eq:311} again, but now taking $h_1,h_2\in C_0([0,\infty))$,  we further obtain
\[
	\left(p_3-2p_2m(\{0\})\right)\cdot \left(h_1(0)\int_0^\infty h_2(x)e^{-\sqrt{2\alpha}x}\dd x-h_2(0)\int_0^\infty h_1(x)e^{-\sqrt{2\alpha}x}\dd x \right)=0.
\]
This clearly implies
\[
	p_3=2p_2\cdot m(\{0\}).
\]
{\color{blue}In particular},  if $p_2=0$,  then $p_3=0$.  This contradicts the condition \eqref{eq:38}.  Therefore, we conclude that $p_2>0, |p_4|=0$,  and the symmetric measure is precisely \eqref{eq:312}. 

Next,  we consider the case where $p_2>0$ and $|p_4|=0$.  Note that \eqref{eq:313} is, in fact,  a regular Dirichlet form on $L^2([0,\infty), m)$; see, e.g., \cite[\S5.3]{F14}.  It suffices to show that
\[
	G_\alpha h\in \sF,\quad \sE_\alpha(G_\alpha h,g)=(h,g)_m,\quad \forall h,g\in C_c^\infty([0,\infty)). 
\]
Indeed,  $G_\alpha^0h\in H^1_0((0,\infty))=\{f\in H^1([0,\infty)): f(0)=0\}$, and hence \eqref{eq:33} implies that $G_\alpha h\in \sF$.  Using \eqref{eq:33},  we further have
\begin{equation}\label{eq:314}
\begin{aligned}
	\sE_\alpha(G_\alpha h,g)&=\frac{1}{2}\int_0^\infty \left(G^0_\alpha h\right)'(x)g'(x)\dd x+\frac{G_\alpha h(0)}{2}\int_0^\infty \left(e^{-\sqrt{2\alpha}x}\right)'g'(x)\dd x  \\
	&\quad + \frac{p_1}{2p_2}G_\alpha h(0)g(0)+\alpha \int_0^\infty G_\alpha^0h(x)g(x)\dd x\\
	&\quad +\alpha G_\alpha h(0)\int_0^\infty e^{-\sqrt{2\alpha}x}g(x)\dd x+\frac{\alpha p_3}{2p_2}G_\alpha h(0)g(0).
\end{aligned}\end{equation}
Substituting \eqref{eq:39}, \eqref{eq:31},  and \eqref{eq:315} into \eqref{eq:314},  we can obtain that
\[
	\sE_\alpha(G_\alpha h,g)=\int_0^\infty h(x)g(x)\dd x+\frac{p_3}{2p_2}h(0)g(0)=(h,g)_m.
\]
This completes the proof.
\end{proof}

It is worth noting that the $L^2$-generator of symmetric Feller's Brownian motion,  i.e.,  the $L^2$-generator of the Dirichlet form \eqref{eq:313},  is derived in, e.g., \cite[Theorem~5.3]{F14}.

\subsection{Pathwise construction}\label{SEC32}

The sample paths of a Feller's Brownian motion were constructed by It\^o and McKean in \cite{IM63}.  The symmetric cases characterized in Corollary~\ref{COR34} correspond to either reflecting Brownian motion (for $p_1=0$) or elastic Brownian motion (for $p_1>0$), up to a time change transformation induced by the speed measure \eqref{eq:312}.  Their pathwise constructions are quite clear; see also \cite[\S10]{IM63}.  
For the readers' convenience,  we will restate some necessary details of the pathwise construction for non-symmetric cases in this subsection.

\subsubsection{$p_2=0,p_3>0, |p_4|<\infty$}\label{SEC321}

This case was examined in \cite[\S9]{IM63}.  The sample paths of Feller's Brownian motion can be constructed as follows: Given a reflecting Brownian motion with sample paths $t\mapsto W^+_t$ starting at a point $x\in [0,\infty)$,  let $Y_t:=W^+_t$ up to the first hitting time $\tau^+_0:=\inf\{t>0:W^+_t=0\}$ of $0$. Then, make $Y$ wait at $0$ for an exponential holding time $\mathfrak{e}$ {\color{blue}with mean $p_3/(p_1+|p_4|)$, which is independent of $W^+$};
at the end of this time $\mathfrak{e}$ (only applicable in the case where $p_1+|p_4|>0$), let it jump to a point in $(0,\infty)\cup \{\partial\}$ according to the distribution $\lambda$ given by
\begin{equation}\label{eq:320}
	\lambda(\dd x)|_{(0,\infty)}:=\frac{p_4(\dd x)}{p_1+|p_4|},\quad \lambda(\{\partial\}):=\frac{p_1}{p_1+|p_4|}.
\end{equation}
If the reaching point is in $(0,\infty)$,  let it start afresh; if it jumps to $\partial$,  let $Y_t:=\partial$ at all later times. 


\subsubsection{$p_2>0,0<|p_4|<\infty$}\label{SEC322}

This case was addressed in \cite[\S12]{IM63} using an increasing L\'evy process,  defined by \eqref{eq:319}.  In fact,  we can provide an alternative construction via the Ikeda-Nagasawa-Watanabe piecing out procedure (see \cite{INW66}), similar to the approach in \cite[Theorem~8.1]{L23} for $Q$-processes with $|\nu|<+\infty$; see also \cite[III\S4]{M68}.  

To be precise,  let us begin with a symmetric Feller's Brownian motion $Y^1=(Y^1_t)_{t\geq 0}$ with parameters $(p_1+|p_4|,p_2,p_3,0)$.  Let $\lambda$ be the probability measure on $(0,\infty)\cup \{\partial\}$ defined as in \eqref{eq:320}. 
Then the Feller's Brownian motion $Y=(Y_t)_{t\geq 0}$ with parameters $(p_1,p_2,p_3,p_4)$ such that $p_2>0,0<|p_4|<\infty$ is actually the piecing out of $Y^1$ with respect to $\lambda$,  as described in \cite[Appendix~A]{L23}. 
Intuitively, $Y$ repeatedly splices resurrection paths at the death times of $Y^1_t$, with the resurrection points randomly determined by the distribution $\lambda$. 


\subsubsection{$|p_4|=\infty$}\label{SEC323}

When $|p_4|=\infty$,  the pathwise construction becomes significantly more challenging. This was accomplished by It\^o and McKean in \cite[\S12-\S15]{IM63}. 

We first consider the case where $p_1=p_3=0$ and $|p_4|=+\infty$ ($p_2\geq 0)$.  Let $Z=(Z(t))_{t\geq 0}$ be a \emph{subordinator},  i.e.,  an increasing L\'evy process on $\mathbb{R}$ with $Z(0)=0$,  whose distribution has the Laplace transform
\begin{equation}\label{eq:319}
\bE e^{-xZ(t)}=\exp\left\{-t\left(p_2x+\int_{(0,\infty)}(1-e^{-xy})p_4(\dd y)\right) \right\}.  
\end{equation}
In \cite{IM63},  this subordinator is referred to as an \emph{increasing differential process}.  For every $s\geq 0$,  define
\[
	Z^{-1}(s):=\inf\{t>0:Z(t)>s\}.
\]
Further, let $W^+=(W^+_t)_{t\geq 0}$ be a reflecting Brownian motion on $[0,\infty)$, independent of $Z$,  with local time $(\ell_t)_{t\geq 0}$ at $0$ given by
\begin{equation}\label{eq:324}
\ell_t=\lim_{\varepsilon\downarrow 0}\frac{1}{2\varepsilon}\int_0^t 1_{\{W^+_s<\varepsilon\}}\dd s,\quad \forall t\geq 0.
\end{equation}
Note that the Revuz measure of $\ell$ with respect to $W^+$ is $\frac{\delta_0}{2}$; see,  e.g.,  \cite[X.  Proposition~2.4]{RY99}.  It was shown in \cite[\S13]{IM63} that
\begin{equation}\label{eq:316}
Y^1_t:=Z(Z^{-1}(\ell_t))-\ell_t+W^+_t,\quad t\geq 0
\end{equation}
is indeed a Feller's Brownian motion with parameters $(0,p_2,0,p_4)$.  
For a detailed explanation of the sample paths defined by \eqref{eq:316}, please refer to \cite[\S12]{IM63}. Note that this pathwise construction \eqref{eq:316} also applies to the case $p_1=p_3=0, p_2>0$, and $|p_4|<+\infty$.  

Regarding the general case,  we note that
\begin{equation}\label{eq:322}
	\ell^{Y^1}_t:=Z^{-1}(\ell_t),\quad t\geq 0
\end{equation}
is the local time of the process defined by \eqref{eq:316} at $0$,  as established in \cite[\S14]{IM63}.  Define $\mathfrak{f}(t):=t+p_3\ell^{Y^1}_t$ and $\mathfrak{f}^{-1}(t):=\inf\{s>0:\mathfrak{f}(s)>t\}$.  Then, the time-changed process of \eqref{eq:316} with respect to $\mathfrak{f}$ is
\begin{equation}\label{eq:321}
	Y^2_t:=Y^1_{\mathfrak{f}^{-1}(t)},\quad t\geq 0.
\end{equation}
This process $Y^2$ is a Feller's Brownian motion with parameters $(0,p_2,p_3,p_4)$.  
The desired Feller’s Brownian motion $Y$ with parameters $(p_1,p_2,p_3,p_4)$ is  obtained as the subprocess of \eqref{eq:321} perturbed by the \emph{multiplicative functional} 
\begin{equation}\label{eq:327}
	M_t:=e^{-p_1 \ell^{Y^1}_{\mathfrak{f}^{-1}(t)}}, \quad  t\geq 0.
\end{equation}
A rigorous construction of the subprocess can be found in, e.g., \cite[III, \S3]{BG68}.  Roughly speaking,  one can take a random time $\zeta$ such that
\begin{equation}\label{eq:326}
	\bP_x(\zeta>t|Y^1)=e^{-p_1 \ell^{Y^1}_{\mathfrak{f}^{-1}(t)}},\quad \forall t>0,
\end{equation}
and then kill the process $Y^2$ at time $\zeta$.  


\subsection{Local times of Feller's Brownian motion}

In \cite[\S14]{IM63},  It\^o and McKean examined the local time \eqref{eq:322} of the special Feller's Brownian motion defined by \eqref{eq:316} at $0$.  What we focus on here is the local time of a general Feller's Brownian motion at a given point $a>0$, {\color{blue}which is regular for $\{a\}$ in the sense of \cite[Definition~11.1]{BG68}}.  

Consider a Feller's Brownian motion on $[0,\infty)$:
\[
Y=(\Omega, \sG,\sG_t,Y_t, \theta_t, (\mathbf{P}_x)_{x\in [0,\infty)})
\] 
 with lifetime $\zeta$,  where $(\sG_t)_{t\geq 0}$ is the \emph{augmented natural filtration} on $\Omega$.  Let $\sG_\infty:=\bigvee_{t\geq 0}\sG_t$ denote the $\sigma$-algebra generated by $\bigcup_{t\geq 0}\sG_t$.   We adjoin to $\Omega$ the dead path $[\partial]$ with  $Y_t([\partial]):=\partial$ for all $t\geq 0$.
A positive continuous additive functional of $Y$ is defined as follows. For detailed discussions, see, e.g., \cite[IV\S1]{BG68} and \cite[Definition~A.3.1]{CF12}. 

\begin{definition}\label{DEF35}
A family $A=(A_t)_{t\geq 0}$ of functions from $\Omega$ to $[0,\infty]$ is called a \emph{positive continuous additive functional} (PCAF for short) of $Y$ if there exists $\Lambda \in \sG_\infty$ such that
\[
	\bP_x(\Lambda)=1 \text{ for } x\in [0,\infty)\quad \text{ and }\quad \theta_t\Lambda \subset \Lambda \text{ for } t\geq 0,
\]
and the following conditions are satisfied:
\begin{itemize}
\item[(A1)] For each $t\geq 0$,  $A_t|_\Lambda\in \sG_t|_\Lambda$. 
\item[(A2)] For every $\omega\in \Lambda$,  $A_\cdot(\omega)$ is continuous on $[0,\infty)$,  $A_0(\omega)=0$,  $A_t(\omega)<\infty$ for $t<\zeta(\omega)$, and $A_t(\omega)=A_{\zeta(\omega)}(\omega)$ for $t\geq \zeta(\omega)$.  
\item[(A3)] For $\omega\in \Lambda$ and every $t,s\geq 0$,  $A_{t+s}(\omega)=A_t(\omega)+A_s(\theta_t\omega)$.  
\end{itemize}
The set $\Lambda$ is called a \emph{defining set} of $A$.  We further make the convention $A_t([\partial])= 0$ for all $t\geq 0$.  
\end{definition}

The \emph{fine support} of a PCAF $(A_t)_{t\geq 0}$ is defined as 
\[
\text{Supp}(A):=\{x\in [0,\infty): \bP_x(R=0)=1\},
\]
where $R(\omega):=\inf\{t>0:A_t(\omega)>0\}$.  According to \cite[V. Theorem~3.13]{BG68},  for a given $a>0$, there exists a PCAF (unique up to a multiplicative constant)  of $Y$ with fine support $\{a\}$. 

Recall that $\tau_0=\inf\{t>0: Y_t=0\}$. The killed process $Y^0$ can be written as
\[
	Y^0=\left(\Omega, \sG,\sG_t, Y^0_t,\theta^0_t,(\bP_x)_{x\in (0,\infty)}\right),
\]
where $Y^0_t$ is defined as in \eqref{eq:317},  $\theta^0_t(\omega):=\theta_t(\omega)$ for $t<\tau_0(\omega)$ and $\theta^0_t(\omega):=[\partial]$ for $t\geq \tau_0(\omega)$; see, e.g., \cite[(12.21i)]{S88}.  The lifetime of $Y^0$ is $\zeta^0:=\zeta\wedge \tau_0(=\tau_0)$.  We can  define the PCAFs for $Y^0$ and their fine supports analogously.  The following fact is crucial to our discussion.

\begin{lemma}\label{LM36}
Given $a>0$,  let $\ell^a=(\ell^a_t)_{t\geq 0}$ be a PCAF of $Y$ with fine support $\{a\}$.  Then 
\[
	\ell^{a,0}_t:=\left\lbrace
	\begin{aligned}
		&\ell^a_t,\quad &t<\zeta^0, \\
		&\ell^a_{\zeta^0},\quad &t\geq \zeta^0
	\end{aligned}
	\right.
\]
is a PCAF of $Y^0$ with fine support $\{a\}$.  
\end{lemma}
\begin{proof}
Assuming without loss of generality that the defining set of $\ell^a$ is $\Omega$, 
it is straightforward to verify that $\ell^{a,0}$ satisfies properties (A1) and (A2).
To establish property (A3) for $\ell^{a,0}$, it suffices to consider the case where $t+s\geq \tau_0(\omega)$.  If $t\geq \tau_0(\omega)$,  then $\ell^{a,0}_t(\omega)=\ell^a_{\tau_0(\omega)}(\omega)$ and $\ell^{a,0}_s(\theta^0_t\omega)=\ell^{a,0}_s([\partial])=\ell^a_0([\partial])=0$. Hence, we have
\begin{equation}\label{eq:318}
	\ell^{a,0}_{t+s}(\omega)=\ell^{a,0}_t(\omega)+\ell^{a,0}_s(\theta^0_t\omega).
\end{equation}
If $t<\tau_0(\omega)$,  \eqref{eq:318} can be verified using $\tau_0(\theta_t\omega)=\tau_0(\omega)-t\leq s$.  Thus,  $\ell^{a,0}$ is indeed a PCAF of $Y^0$.  The fine support of $\ell^{a,0}$ is $\{a\}$ because $\ell^{a,0}_t\leq \ell^a_t$ and $\mathbf{P}_a(\tau_0>0)=1$. 
\end{proof}

{\color{blue}Note that $Y^0$ is the {\color{blue}killed} Brownian motion on $(0,\infty)$}.  Therefore, we can define the \emph{Revuz measure} $\nu_a$ of $\ell^{a,0}$ with respect to $Y^0$ as follows:
\[
\begin{aligned}
	\int_{(0,\infty)}f(x)\nu_a(\dd x)&=\lim_{t\downarrow 0}\frac{1}{t}\int_{(0,\infty)}\mathbf{E}_x\left( \int_0^tf(Y^0_t)\dd\ell^{a,0}_t\right)\dd x \\
	&=\lim_{t\downarrow 0}\frac{1}{t}\int_{(0,\infty)}\mathbf{E}_x\left( \int_0^{t\wedge \tau_0}f(Y_t)\dd\ell^{a}_t \right)\dd x.
\end{aligned}\]
Note that $\nu_a$ is a constant multiple of $\delta_a$. 
Unlike the approach of normalizing local times by the values of their potentials as described in \cite[V. Theorem~3.13]{BG68},  we opt to use the unique PCAF with fine support $\{a\}$ in the following sense.

\begin{definition}\label{DEF37}
Given $a>0$,  a PCAF $L^a=(L^a_t)_{t\geq 0}$ with $\text{Supp}(L^a)=\{a\}$ is called the local time of $Y$ at $a$ if the Revuz measure of $L^{a,0}:=(L^a_{t\wedge \zeta^0})_{t\geq 0}$ with respect to $Y^0$ is $\delta_a$.
\end{definition}

In the symmetric case,  where $p_2>0$ and $|p_4|=0$,  we can also define the Revuz measure of the local time $L^a$ with respect to $Y$ and the symmetric measure $m$.  According to \cite[Proposition~4.1.10]{CF12},  this Revuz measure is also equal to the Dirac measure $\delta_a$.  Therefore, the definition of local time provided here is consistent with the definition of local time in the theory of Dirichlet forms.

\section{Time-changed Feller's Brownian motions are birth-death processes}

Consider the birth-death {\color{blue}$Q$-matrix} \eqref{eq:11} and from now on, assume that $\infty$ is regular for the minimal $Q$-process.  {\color{blue}In particular},  the scale function $(c_k)_{k\geq 0}$ given by \eqref{eq:22} satisfies 
\[
	c_\infty=\lim_{k\rightarrow\infty}c_k<\infty,
\]
and the speed measure $\mu$,  as defined in \eqref{eq:03}, is finite; see, e.g., \cite[Remark~3.4]{L23}.

The aim of this section is to demonstrate that any Feller's Brownian motion can be converted into a $Q$ -process through a time change transformation and a spatial homeomorphism. The special cases of {\color{blue}killed} and reflecting Brownian motions have been analyzed in \cite[\S3]{L23}.  In these cases, the transformed  $Q$-processes are the minimal  $Q$-process and the  $(Q,1)$-process, respectively.

\subsection{Spatial transformation}\label{SEC41}

The formulation we will present encounters a significant issue because the boundary point $0$ of Feller’s Brownian motion and the boundary point $\infty$ of the $Q$-process are located at opposite ends of their respective state spaces. Additionally, unlike Feller’s Brownian motion, the $Q$-process is not on the natural scale, {\color{blue}that is, the function $s(k)=k$ on $\bN$ does not satisfy
\[
	\bP^\text{min}_k(\sigma^\text{min}_{k+1}<\sigma^\text{min}_{k-1})=\frac{s(k)-s(k-1)}{s(k+1)-s(k-1)},\quad k\geq 1;
\]
see \eqref{eq:231}.} 

{\color{blue}Fortunately}, both issues can be resolved by applying a straightforward spatial transformation to the $Q$-process.
To address this, define $\hat c_n:=c_\infty-c_n$ for $n\in \mathbb{N}$.  Let
\begin{equation}\label{eq:419}
	E:=\{\hat{c}_n: n\in \mathbb{N}\},\quad \overline E=E\cup \{0\}
\end{equation}
and 
\begin{equation}\label{eq:410}
	\Xi:E\rightarrow \mathbb{N},\quad \hat{c}_n\mapsto n.
\end{equation}
Clearly, $\Xi$ can be extended to a homeomorphism between $\overline{E}$ and $\bN\cup \{\infty\}$,  where $\overline{E}$ is endowed with the relative topology of $\bR$.  
For each $Q$-process $X=(X_t)_{t\geq 0}$ on $\mathbb{N}\cup \{\infty\}$,
\begin{equation}\label{eq:4311}
	\hat{X}_t:=\Xi^{-1}(X_t),\quad t\geq 0
\end{equation}
defines a continuous-time Markov chain on $\overline E$. {\color{blue}According to \eqref{eq:231}, the function $\hat{s}(x):=x$, $x\in \overline{E}$, satisfies 
\[
	\hat\bP^\text{min}_{\hat c_k}(\hat \sigma^\text{min}_{\hat c_{k-1}}<\hat\sigma^\text{min}_{\hat c_{k+1}})=\frac{\hat s(\hat c_{k})-\hat s(\hat c_{k+1})}{\hat s(\hat c_{k-1})-\hat s(\hat c_{k+1})},\quad k\geq 1,
\]
where $\hat{X}^\text{min}$ denotes the process $\hat{X}$ killed upon its first hitting time of $0$, $\hat\bP^\text{min}_{\hat c_k}$ is the law of $\hat{X}^\text{min}$ starting from $\hat c_k$, and $\hat \sigma^\text{min}_{\hat c_{k}}:=\inf\{t>0:\hat{X}^\text{min}_t=\hat{c}_k\}$ for $k\in \bN$.}

{\color{blue}The process $\hat{X}$ defined as \eqref{eq:4311}} is a Feller process on $\overline{E}$ whenever $X$ is a Feller $Q$-process.  For convenience,  we also refer to $\hat{X}=(\hat{X}_t)_{t\geq 0}$ as a Doob process or a Feller $Q$-process (on $E$ or $\overline{E}$).  


\subsection{Time change}

We first prepare the ingredient,  specifically the PCAF,  for the time change transformation on a Feller's Brownian motion $Y$.  For $\hat c_n\in E$,  let $L^{\hat c_n}=(L^{\hat c_n}_t)_{t\geq 0}$ denote the local time of $Y$ at $\hat c_n$ as defined in Definition~\ref{DEF37}.  Define
\begin{equation}\label{eq:41}
	A_t=\sum_{n\in \mathbb{N}}\mu_nL^{\hat{c}_n}_t,\quad t\geq 0.
\end{equation}
We will show that $A=(A_t)_{t\geq 0}$ is a PCAF of $Y$.

\begin{lemma}\label{LM41}
The family of functions $A=(A_t)_{t\geq 0}$ defined as \eqref{eq:41} is a PCAF of the Feller's Brownian motion $Y$.  Furthermore,  
\begin{itemize}
\item[(1)] If $p_2=0,p_3>0$, and $|p_4|<\infty$,  then the fine support of $A$ is $E$;
\item[(2)] Otherwise, the fine support of $A$ is $\overline{E}$.
\end{itemize}
\end{lemma}
\begin{proof}
We first demonstrate that 
\begin{equation}\label{eq:46}
\mathbf{E}_x \int_0^\infty e^{-t}\dd A_t=\sum_{n\in \mathbb{N}}\mu_n\mathbf{E}_x\int_0^\infty e^{-t}\dd L^{\hat{c}_n}_t<\infty,\quad \forall x\in [0,\infty).
\end{equation}
The main task is to estimate $\mathbf{E}_x\int_0^\infty e^{-t}\dd L^{\hat{c}_n}_t$ for each $n\in \mathbb{N}$.  Let $T_n:=\inf\{t>0: Y_t=\hat{c}_n\}$.  According to \cite[V. Theorem~3.13]{BG68}, there exists a positive constant $k_n$ such that
\begin{equation}\label{eq:42}
	\mathbf{E}_x\int_0^\infty e^{-t}\dd L^{\hat{c}_n}_t=k_n\mathbf{E}_x e^{-T_n},\quad \forall x\in [0,\infty).
\end{equation}
By the strong Markov property of $Y$ and \cite[IV,  Proposition~1.13]{BG68},  we obtain
\[
\mathbf{E}_x\int_0^\infty e^{-t}\dd L^{\hat{c}_n}_t=\mathbf{E}_x\int_0^{\tau_0} e^{-t}\dd L^{\hat{c}_n}_t+\mathbf{E}_x e^{-\tau_0}\cdot \mathbf{E}_0\int_0^\infty e^{-t}\dd L^{\hat{c}_n}_t.
\]
Integrating both sides with respect to the Lebesgue measure $\mathfrak{m}$ on $[0,\infty)$,  we have
\[
\begin{aligned}
	\mathbf{E}_\mathfrak{m}\int_0^\infty e^{-t}\dd L^{\hat{c}_n}_t&=\mathbf{E}_\mathfrak{m}\int_0^{\tau_0} e^{-t}\dd L^{\hat{c}_n}_t+\mathbf{E}_\mathfrak{m} e^{-\tau_0}\cdot \mathbf{E}_0\int_0^\infty e^{-t}\dd L^{\hat{c}_n}_t \\
	&=(1-\mathbf{E}_{\hat{c}_n}e^{-\tau_0})+\mathbf{E}_\mathfrak{m} e^{-\tau_0}\cdot \mathbf{E}_0\int_0^\infty e^{-t}\dd L^{\hat{c}_n}_t,
\end{aligned}
\]
where the second equality follows from \cite[(4.1.3)]{CF12}.  Substituting \eqref{eq:42} and $\mathbf{E}_x e^{-\tau_0}=e^{-\sqrt{2}x}$ for $x>0$ into the above equation, we obtain
\begin{equation}\label{eq:43}
k_n=\frac{\sqrt{2}(1-e^{-\sqrt{2}\hat{c}_n})}{\sqrt{2}\mathbf{E}_\mathfrak{m} e^{-T_n}-\mathbf{E}_0e^{-T_n}}.
\end{equation}
The strong Markov property also implies that for any $0<x<\hat c_n$,
\begin{equation}\label{eq:44}
\begin{aligned}
	\mathbf{E}_xe^{-T_n}&=\mathbf{E}_x\left(e^{-T_n}; T_n<\tau_0\right)+\mathbf{E}_x\left(e^{-T_n}; T_n>\tau_0\right) \\
	&=\mathbf{E}_x\left(e^{-T_n}; T_n<\tau_0\right)+\mathbf{E}_x\left(e^{-\tau_0}; T_n>\tau_0\right)\mathbf{E}_0 e^{-T_n}.
\end{aligned}\end{equation}
For $x>\hat c_n$,  it holds that 
\begin{equation}\label{eq:45}
	\mathbf{E}_xe^{-T_n}=e^{-\sqrt{2}(x-\hat{c}_n)}.
\end{equation}
Substituting \eqref{eq:44}, \eqref{eq:45} into \eqref{eq:43} and using \cite[Problem 6 of page 29]{IM74},  we deduce that
\begin{equation}\label{eq:415}
k_n=\frac{\sqrt{2}(e^{\sqrt{2}\hat c_n}-e^{-\sqrt{2}\hat c_n})}{2(e^{\sqrt{2}\hat c_n}-\mathbf{E}_0 e^{-T_n})}.
\end{equation}
Since $\mathbf{E}_0e^{-T_n}\leq 1$,  it follows that 
\[
\mathbf{E}_x\int_0^\infty e^{-t}\dd L^{\hat{c}_n}_t\leq k_n\leq \frac{\sqrt{2}(e^{\sqrt{2}\hat c_n}-e^{-\sqrt{2}\hat c_n})}{2(e^{\sqrt{2}\hat c_n}-1)}=\frac{\sqrt{2}}{2}(1+e^{-\sqrt{2}\hat{c}_n})\leq \sqrt{2}.
\]
Note that $\sum_{n\in \mathbb{N}}\mu_n<\infty$.  Therefore,  \eqref{eq:46} can be concluded.

To prove that \eqref{eq:41} is a PCAF of $Y$,  we start by considering the defining sets $\Lambda_n\in \sG_\infty$ of $L^{\hat{c}_n}$.  Define
\[
	\Lambda:=\left(\bigcap_{n\in \mathbb{N}}\Lambda_n \right) \bigcap  \left\{\omega\in \Omega:A_t(\omega)<\infty,\forall t<\zeta(\omega) \right\}.
\]
Clearly,  $A_t|_{\bigcap_n \Lambda_n}\in \sG_t|_{\bigcap_n \Lambda_n}$.  
Note that 
\begin{equation}\label{eq:47}
	\left\{\omega\in \Omega:A_t(\omega)<\infty,\forall t<\zeta(\omega) \right\}=\bigcap_{t>0}\left( \{A_t<\infty, t<\zeta\}\cup \{\zeta\leq t\}\right).
\end{equation}
Thus,  it is straightforward to verify that $\Lambda\in \sG_\infty$ and that $\Lambda$ satisfies all the conditions in Definition~\ref{DEF35} except for
\begin{equation}\label{eq:48}
	\mathbf{P}_x(\Lambda)=1,\quad \forall x\in [0,\infty).  
\end{equation}
To show \eqref{eq:48}, assume for contradiction that
\[
	\mathbf{P}_x\left( \left(\bigcap_{t>0}\left( \{A_t<\infty, t<\zeta\}\cup \{\zeta\leq t\}\right) \right)^c \right)=\mathbf{P}_x\left(\bigcup_{t>0}\{A_t=\infty, t<\zeta\} \right)>0.
\]
This implies $\mathbf{P}_x(\bigcup_{t>0}\{A_t=\infty\})>0$.  Since $A_t$ is increasing in $t$,  $\{A_t=\infty\}$ is also increasing in $t$.  Therefore,  there exists $t_0>0$ such that $\mathbf{P}_x(A_{t_0}=\infty)>0$.   We have
\[
\begin{aligned}
	\mathbf{E}_x\int_0^\infty e^{-t}\dd A_t&\geq \mathbf{E}_x\left(\int_0^{t_0} e^{-t}\dd A_t; \{A_{t_0}=\infty\} \right)\\
	& \geq e^{-t_0}\mathbf{E}_x\left(\int_0^{t_0}\dd A_t; \{A_{t_0}=\infty\} \right)\\
	&=\infty. 
\end{aligned}\]
This contradicts \eqref{eq:46}.  Thus, \eqref{eq:48} holds.

Finally,  we examine the fine support of $A$. 
Note that $\dd L^{\hat{c}_n}_t$ (as a measure in $t$) vanishes outside $\{t: Y_t=\hat{c}_n\}$.  Thus, $\dd A_t$ vanishes on $\{t: Y_t\notin \overline{E}\}$.  Consequently, by \cite[V. Corollary 3.10]{BG68}, we have $\text{Supp}(A)\subset \overline{E}$.  On the other hand,  since $A_t\geq \mu_nL^{\hat{c}_n}_t$,  it follows from the definition that $\hat{c}_n\in \text{Supp}(A)$.  Therefore,  $E\subset \text{Supp}(A)$.  

If $p_2=0,p_3>0$, and $|p_4|<\infty$,  then $Y_t=0$ for all $0\leq t<\mathfrak{e}$,  $\bP_0$-a.s.,  where $\mathfrak{e}$ is the exponential holding time given in \S\ref{SEC321}.  Note that $\bP_0(\mathfrak{e}>0)=1$.  Hence, $\bP_0(R\geq \mathfrak{e}>0)=1$,  where $R=\inf\{t>0:A_t>0\}$.  By the definition of fine support,  it follows that $0\notin \text{Supp}(A)$.  Therefore,  $\text{Supp}(A)=E$. 

It remains to show $0\in \text{Supp}(A)$ for the remaining cases.  We proceed by contradiction. Suppose $0\notin \text{Supp}(A)$.  Then $\mathbf{P}_0(R>0)=1$.  For $\bP_0$-a.s.  $\omega\in \Omega$ and $0<t<R(\omega)$,  we have $L^{\hat{c}_n}_t(\omega)=0$ for all $n\in \mathbb{N}$.  According to \cite[V Theorem~3.8]{BG68},  $Y_t(\omega)\notin E$ for all $0<t<R(\omega)$.  Note that $Y_t(\omega)$ is c\`adl\`ag in $t$. From the pathwise representation of $Y$ (see \S\ref{SEC32}),  we can obtain that if $Y_{t-}(\omega)\neq Y_t(\omega)$,  then $Y_{t-}(\omega)=0$.  This fact, combined with $Y_t(\omega)\notin E$ for $0<t<R(\omega)$ and $Y_0(\omega)=0$,  implies that $Y_t(\omega)=0$ for all $0\leq t<R(\omega)$.  {\color{blue}In particular},  $Y_t(\omega)$ is continuous in $t\in [0,R(\omega))$.  This is impossible when $p_2>0, |p_4|<\infty$,  because before the first jumping time,  the sample paths of $Y$ are those of a symmetric Feller's Brownian motion with parameters $(p_1+|p_4|,p_2,p_3,0)$.  This process,  as a \emph{regular diffusion process} on $[0,\infty)$,  can not stay at any point for an extended period; see, e.g., \cite[V. (47.1)]{RW87}. 
For the case where $|p_4|=\infty$, the pathwise representation of $Y$ in \S\ref{SEC323} indicates that $Y_t(\omega)$ is exactly a Brownian path $W^+_t(\omega)$ (up to a transformation of time change \eqref{eq:321}) for $t\in [0,R(\omega))$.  This also contradicts $Y_t(\omega)=0$ for $t\in [0,R(\omega))$.    
\end{proof}

With the Feller's Brownian motion $Y$ and its PCAF $A=(A_t)_{t\geq 0}$ given by \eqref{eq:41},  we can now introduce the time-changed process of $Y$ with respect to $A$.  Define the right-continuous inverse of $A_t(\omega)$ for each $\omega\in \Omega$ as
\[
\gamma_t(\omega):=\left\lbrace
\begin{aligned}
	&\inf\{s: A_s(\omega)>t\}\quad &\text{for }t<A_{\zeta(\omega)-}(\omega), \\
	&\infty\quad &\text{for }t\geq A_{\zeta(\omega)-}(\omega).  
\end{aligned}
\right.
\]
Further, let 
\[
\check{Y}_t(\omega):=Y_{\gamma_t(\omega)}(\omega),\quad \check{\zeta}(\omega):=A_{\zeta(\omega)-}(\omega)(=A_{\zeta(\omega)}(\omega)),\quad t\geq 0,\omega\in \Omega.
\]
Note that $Y_t(\omega):=\partial$ for $\zeta(\omega)\leq t\leq \infty$,  so $\check{Y}_t(\omega)=\partial$ for $t\geq \check{\zeta}(\omega)$.  
According to \cite[Proposition~A.3.8(iv)]{CF12},  we may assume without loss of generality that $\check Y_{t}(\omega)\in \text{Supp}(A)\cup \{\partial\}$ for all $t\geq 0$ and all $\omega\in \Omega$.  Set $\check{\sG}_t:=\sG_{\gamma_t}$ and $\check{\theta}_t:=\theta_{\gamma_t}$.  It is well known that the time-changed process 
\begin{equation}\label{eq:49}
\check{Y}:=\left(\Omega, \sG, \check{\sG}_t,\check{Y}_t,\check{\theta}_t,\left(\mathbf{P}_x\right)_{x\in \text{Supp}(A)}\right)
\end{equation}
with lifetime $\check{\zeta}$ is a \emph{right process} on $\text{Supp}(A)$;  see, e.g., \cite[Theorem~A.3.11]{CF12}. 

\begin{theorem}\label{THM42}
Let $Y$ be a  Feller's Brownian motion with parameters $(p_1,p_2,p_3,p_4)$ as specified in Theorem~\ref{THM32}, and let
$\check{Y}$ be the time-changed process \eqref{eq:49} of $Y$ with respect to the PCAF \eqref{eq:41}.  Then $\Xi(\check{Y}):=\left(\Xi(\check{Y}_t)\right)_{t\geq 0}$ is a $Q$-process whose birth-death {\color{blue}$Q$-matrix} is \eqref{eq:11},  where $\Xi$ is defined as \eqref{eq:410} {\color{blue}and extended to $\overline{E}$}.  Furthermore,
\begin{itemize}
\item[(1)] If $p_2=0,p_3>0$, and $|p_4|=0$,  then $\Xi(\check{Y})$ is the minimal $Q$-process;
\item[(2)] If $p_2=0,p_3>0$, and $0<|p_4|<\infty$,  then $\Xi(\check{Y})$ is a Doob process;
\item[(3)] Otherwise,  $\Xi(\check{Y})$ is a Feller $Q$-process.
\end{itemize}
\end{theorem}
\begin{proof}
Denote the transition semigroup and resolvent of $\check{Y}$ by $(\check{T}_t)_{t\geq 0}$ and $(\check{G}_\alpha)_{\alpha>0}$, respectively.  Define
\[
	\check p_{ij}(t):=\check{T}_t1_{\{\hat c_j\}}(\hat{c}_i),\quad t\geq 0, i,j\in \mathbb{N},
\]
and 
\[
\check{\Psi}_{ij}(\alpha):=\int_0^\infty e^{-\alpha t}\check{p}_{ij}(t)\dd t,\quad \alpha>0,  i,j\in \bN.
\]
The goal is to prove that $(\check{p}_{ij}(t))_{i,j\in \mathbb{N}}$ is a $Q$-process.  

We first demonstrate that $(\check{p}_{ij}(t))_{i,j\in \mathbb{N}}$ is a standard transition matrix.  According to \cite[\S2.5, Theorem~1]{WY92},  it suffices to verify that $\check{\Psi}_{ij}(\alpha)$ satisfies the following conditions:
\begin{equation}\label{eq:411}
\begin{aligned}
&\check{\Psi}_{ij}(\alpha)\geq 0,\quad \alpha \sum_{j\in \bN}\check{\Psi}_{ij}(\alpha)\leq 1, \\
&\check{\Psi}_{ij}(\alpha)-\check{\Psi}_{ij}(\beta)+(\alpha-\beta)\sum_{k\in \bN}\check{\Psi}_{ik}(\alpha)\check{\Psi}_{kj}(\beta)=0,\quad \forall \alpha,\beta>0,\\
&\lim_{\alpha\rightarrow \infty}\alpha\check{\Psi}_{ij}(\alpha)=\delta_{ij},
\end{aligned}
\end{equation}
where $\delta_{ij}$ is the Kronecker delta.  To accomplish this,  note that 
\[
	\check{\Psi}_{ij}(\alpha)=\check{G}_\alpha 1_{\{\hat{c}_j\}}(\hat{c}_i).
\]
Thus,  the first condition in \eqref{eq:411} is straightforward, and the third condition follows from the right continuity of  $\check{Y}$.  To prove the second condition, it is sufficient to demonstrate
\begin{equation}\label{eq:412}
	\check{G}_\alpha 1_{\{0\}}(\hat{c}_i)=0
\end{equation}
and then apply the resolvent equation of $\check G_\alpha$.  In fact,  it follows from the definition of $\check{Y}$,  \cite[Proposition~4.9 of page 8]{RY99}, and the Fubini theorem that
\[
\begin{aligned}
	\check{G}_\alpha 1_{\{0\}}(\hat{c}_i)&=\mathbf{E}_{\hat{c}_i}\int_0^\infty e^{-\alpha A_{\gamma_t}}1_{\{0\}}(Y_{\gamma_t})\dd t =\mathbf{E}_{\hat{c}_i}\int_0^\infty e^{-\alpha A_{t}}1_{\{0\}}(Y_{t})\dd A_t \\
	&=\sum_{n\in \mathbb{N}} \mu_n\mathbf{E}_{\hat{c}_i}\int_0^\infty e^{-\alpha A_{t}}1_{\{0\}}(Y_{t})\dd L^{\hat{c}_n}_t=0.  
\end{aligned}\]
Thus, \eqref{eq:412} is established.  

Recall that $Y^0$, with lifetime $\zeta^0$, is a {\color{blue}killed} Brownian motion on $(0,\infty)$.  According to Lemma~\ref{LM36} and Lemma~\ref{LM41},  we have that
\[
	A^0_t:=A_{t\wedge \zeta^0},\quad t\geq 0
\]
is a PACF of $Y^0$,  with Revuz measure $\sum_{n\in \bN}\mu_n\delta_{\hat{c}_n}$.  Define 
\[
\hat{Y}^0=\left(\Omega, \sG,\hat{\sG}_t, \hat{Y}^0_t,\hat{\theta}^0_t, (\bP_x)_{x\in E}\right),
\] 
with lifetime $\hat{\zeta}^0:=A^0_{\zeta^0}=A_{\zeta^0}$, as the time-changed process of $Y^0$ with respect to the PACF $A^0$. Specifically,
\[
	\hat{\sG}_t:=\sG_{\gamma^0_t},\quad \hat{\theta}^0_t:=\theta^0_{\gamma^0_t},\quad \hat{Y}^0_t:=Y^0_{\gamma^0_t},
\]
with
\[
\gamma^0_t(\omega):=\left\lbrace
\begin{aligned}
	&\inf\{s: A^0_s(\omega)>t\}\quad &\text{for }t<\hat{\zeta}^0, \\
	&\infty\quad &\text{for }t\geq \hat{\zeta}^0.  
\end{aligned}
\right.
\]
{\color{blue}By mimicking the proof of \cite[Lemma~3.1]{L23}, it can be established} that $\Xi(\hat{Y}^0)$ is exactly the minimal $Q$-process.  

Next,  let us examine the killed process of $\check{Y}$ upon hitting $0\in \overline{E}$.  More precisely,  let $\check{\eta}_n:=\inf\{t>0:\check{Y}_t=\hat{c}_n\}$ for all $n\geq 1$.  From the definition \eqref{eq:49} of $\check{Y}$,  we find that for $\bP_{\hat{c}_i}$-a.s. $\omega\in \Omega$,  $\check{\eta}_n(\omega)$ is increasing in $n$ for $n> i$.  Hence,
\begin{equation}\label{eq:420}
	\check{\eta}_\infty(\omega):=\lim_{n\rightarrow \infty}\check{\eta}_n(\omega)
\end{equation}
is well defined for $\bP_{\hat{c}_i}$-a.s.  $\omega$ and all $i\in \bN$.  We will show that $\check{\eta}_\infty=A_{\zeta^0}=\hat{\zeta}^0$,  and hence the killed process
\[
	\check{Y}^0_t:=\left\lbrace
\begin{aligned}		
		&\check Y_t,\quad &t<\check{\zeta}\wedge \check{\eta}_\infty, \\
		&\partial,\quad &t\geq \check{\zeta}\wedge \check{\eta}_\infty,
\end{aligned}	\right.
\] 
with lifetime $\check{\zeta}^0:=\check{\zeta}\wedge \check{\eta}_\infty (=\check{\eta}_\infty)$,  on $E$  is identical  to $\hat{Y}^0$.  To prove this,  fix $\bP_{\hat{c}_i}$ and let $\eta_n:=\inf\{t>0:Y_t=\hat{c}_n\}$ for $n>i$.  Note that $Y_{\eta_n}=\hat{c}_n$ and $Y_t\neq \hat{c}_n$ for all $t<\eta_n$,  $\bP_{\hat{c}_i}$-a.s.  It follows from \cite[V. Theorem~3.8]{BG68} that $A_{\eta_n}=A_{\eta_n-\varepsilon_0}$ for some $\varepsilon_0>0$ and $A_{\eta_n+\varepsilon}>A_{\eta_n}$ for all $\varepsilon>0$.  Consequently,
\[
\gamma_{A_{\eta_n}}=\inf\{s>0:A_s>A_{\eta_n}\}=\eta_n,
\]
and
\[
\gamma_t\leq \eta_n-\varepsilon_0,\quad \forall t<A_{\eta_n}.
\]
These yield
\[
\check{Y}_{A_{\eta_n}}=Y_{\gamma_{A_{\eta_n}}}=Y_{\eta_n}=\hat{c}_n
\]
and 
\[
	\check{Y}_t=Y_{\gamma_t}\neq \hat{c}_n,\quad \forall t<A_{\eta_n}.
\]
In other words,  $\check{\eta}_n=A_{\eta_n}$, $\bP_{\hat{c}_i}$-a.s.  Noting that $\lim_{n\rightarrow \infty}\eta_n=\zeta^0$,  we obtain that $\check{\eta}_\infty=\lim_{n\rightarrow \infty}\check{\eta}_n=\lim_{n\rightarrow \infty}A_{\eta_n}=A_{\zeta^0}$.

According to the argument in the previous two paragraphs,  we can conclude that $\check{p}^0_{ij}(t):=\mathbf{P}_{\hat{c}_i}(\check{Y}_t=\hat c_j, t<\check{\eta}_\infty)$ is the transition matrix of the minimal $Q$-process.  This implies that
\begin{equation}\label{eq:413}
	\lim_{t\rightarrow 0}\frac{\check{p}^0_{ij}(t)-\delta_{ij}}{t}=q_{ij}.
\end{equation}
Mimicking the proof of \cite[Proposition~3.6]{L23},  we can also obtain that
\begin{equation}\label{eq:414}
\lim_{t\rightarrow 0}\frac{\mathbf{P}_{\hat c_i}(\check{\eta}_\infty\leq t)}{t}=0.
\end{equation}
Combining \eqref{eq:413} and \eqref{eq:414}, we obtain
\[
	\lim_{t\rightarrow 0}\frac{\check{p}_{ij}(t)-\delta_{ij}}{t}=q_{ij}.
\]
In other words,  $(\check{p}_{ij}(t))_{i,j\in \mathbb{N}}$ is a $Q$-process.  

Finally,  let us classify the $Q$-process $\Xi(\check{Y})$ for different cases.  
For the first case,  where $p_2=0,p_3>0$ and $|p_4|=0$,  we observe that $t\mapsto A_t$ does not increase after time $\zeta^0$ (according to the sample path representation in \S\ref{SEC321}).  Consequently, $\gamma_t=\infty$ for all $t\geq A_{\zeta^0}$,  which implies that $\hat{Y}_t=\partial$ for all $t\geq A_{\zeta^0}=\check{\eta}_\infty$.  Thus, $\Xi(\check{Y})$ aligns with the minimal $Q$-process.
In the second case, where $p_2=0,p_3>0$ and $0<|p_4|<\infty$,  $t\mapsto A_t$ may continue to increase after time $\zeta^0$.  A similar argument shows that $\Xi(\check{Y})$ is not the minimal $Q$-process.  However, according to Lemma~\ref{LM41}, $\check{Y}$ is a right process with state space $\text{Supp}(A)=E$.  Therefore, $\check{Y}_{\check{\eta}_\infty}\in E\cup \{\partial\}$,  $\bP_{\hat{c}_i}$-a.s.  for all $i\in \bN$.  Based on \cite[Corollary~5.2]{L23},  it follows that $\Xi(\check{Y})$ is a Doob process.

For the remaining cases,  it suffices to show that $\check{Y}_{\check{\eta}_\infty}=0$,  $\bP_{\hat{c}_i}$-a.s.  for all $i\in \bN$.  
To demonstrate this, consider
\[
\gamma_{\check{\eta}_\infty}=\inf\{t>0:A_t>\check{\eta}_\infty\}=\inf\{t>0:A_t>A_{\zeta^0}\}. 
\]
 If it were false that $\gamma_{\check{\eta}_\infty}=\zeta^0$,  then there would exist some $\varepsilon_0>0$ such that $A_{\zeta^0+\varepsilon_0}=A_{\zeta^0}$.   This implies that $Y_t\notin E$ for all $\zeta^0<t<\zeta^0+\varepsilon_0$ by \cite[V. Theorem~3.8]{BG68}.  Since $Y_{\zeta^0}=0$ and $Y$ is right continuous, it follows that  $Y_t=0$ for all $\zeta^0\leq t<\zeta^0+\varepsilon_0$.  However, this contradicts the previous argument that $Y$ cannot remain at $0$ for any extended period, as noted in the last paragraph of the proof of Lemma~\ref{LM41}.
Therefore,  we have $\check{Y}_{\check{\eta}_\infty}=Y_{\zeta^0}=0$. 
\end{proof}

In the context of  general Markov process,  the Feller's Brownian motion $Y$ with parameters $(p_1,p_2,p_3,p_4)$ is termed \emph{conservative} if $\bP_x(\zeta<\infty)=0$ for any $x\in [0,\infty)$, which is equivalent to $p_1=0$ (see \cite[\S15]{IM63}).  In terms of continuous-time Markov chains,  a conservative $Q$-process is also referred to as an \emph{honest} $Q$-process (see \S\ref{SEC22}).  

\begin{corollary}
If the Feller's Brownian motion $Y$ is conservative,  then for any $a>0$ and $x\in [0,\infty)$,
\begin{equation}\label{eq:418}
	\bP_x(L^a_\infty=\infty)=1.
\end{equation} 
{\color{blue}In particular},  the $Q$-process $\left(\Xi(\check{Y}_t)\right)_{t\geq 0}$, obtained in Theorem~\ref{THM42}, is also conservative.  
\end{corollary}
\begin{proof}
Write $T:=T_a:=\inf\{t>0:Y_t=a\}$.  We first show that
\begin{equation}\label{eq:416}
\mathbf{P}_x(T<\infty)=1.  
\end{equation}
To demonstrate this, let {\color{blue}$\sigma:=\inf\{t>0:Y_t\geq a\}\leq T$}.  Observe that
\[
 \{\sigma=\infty\}=\{Y_t<a,\forall t>0\}.
\]
The sample path of $Y$ is described by \eqref{eq:316} up to a time change transformation (see \eqref{eq:321}).  Given that $p_1=0$,  it follows that
\[
	\{\sigma=\infty\}\subset\{W^+_t<a,\forall t>0\}.
\]
Therefore,  $\bP_x(\sigma=\infty)\leq \bP_x(W^+_t<a,\forall t>0)=0$.  On the other hand,  by the strong Markov property, we have
\[
	\bP_x(T=\infty, \sigma<\infty)=\bE_x \left(\bP_{Y_\sigma}(T=\infty);\{\sigma<\infty\} \right).
\]
Since $Y_\sigma\geq a$,  $\bP_x$-a.s. on $\{\sigma<\infty\}$,  it follows that $\bP_{Y_\sigma}(T=\infty)=0$,  $\bP_x$-a.s. on $\{\sigma<\infty\}$.  Therefore,  $\bP_x(T=\infty, \sigma<\infty)=0$.  Combining this with $\bP_x(\sigma=\infty)=0$,  we can eventually derive \eqref{eq:416}.

Next,  by mimicking the computation in \eqref{eq:415},  we can show that for $\alpha>0$, 
\[
\bE_a\int_0^\infty e^{-\alpha t}\dd L^a_t=\frac{1}{\sqrt{2\alpha}}\frac{e^{\sqrt{2\alpha}a}-e^{-\sqrt{2\alpha}a}}{e^{\sqrt{2\alpha}a}-\bE_0e^{-\alpha T}}.
\]
It follows from \eqref{eq:416} (with $x=0$) that
\begin{equation}\label{eq:417}
	\lim_{\alpha \rightarrow 0}\bE_a\int_0^\infty e^{-\alpha t}\dd L^a_t=\infty.  
\end{equation}

Finally,  by using \eqref{eq:416} and \eqref{eq:417}, we can apply \cite[V. Theorem~3.17]{BG68} to conclude \eqref{eq:418}.
\end{proof}

\subsection{Uniqueness of PCAF for time change}

We continue to examine the time change transformation in Theorem~\ref{THM42}.  The goal is to demonstrate that \eqref{eq:41} is indeed the unique PCAF of $Y$ for which the corresponding time-changed process is a $Q$-process with the given {\color{blue}$Q$-matrix} \eqref{eq:11}.  

\begin{theorem}\label{THM44}
Let $Y$ be a Feller's Brownian motion and $A^1=(A^1_t)_{t\geq 0}$ be a PCAF of $Y$ such that $\text{Supp}(A^1)\subset \overline{E}$.  Denote by $\check{Y}^1$ the time-changed process of $Y$ with respect to the PCAF $A^1$.  If $X^1:=\Xi(\check{Y}^1)$ is a $Q$-process with the given {\color{blue}$Q$-matrix} \eqref{eq:11},  then 
\[
	A^1_t=A_t,\quad \forall t\geq 0,
\]
where $A=(A_t)_{t\geq 0}$ is defined as \eqref{eq:41}. 
\end{theorem}
\begin{proof}
Let $\ell=(\ell_t)_{t\geq 0}$ denote the local time of $Y$ at $0$,  as discussed in \cite{IM63} or \cite[V\S3]{BG68}.  (We will soon see that the normalization of this local time is not necessary for our proof.) Define
\[
	\bar{A}_t:=A_t+\ell_t,\quad t\geq 0.
\]
Clearly, $\bar{A}=(\bar{A}_t)_{t\geq 0}$ is a PCAF of $Y$.  To verify the condition in \cite[V\S3, Proposition~3.11]{BG68} for $A^1$ and $\bar{A}$,  consider $g\in \mathcal{B}_+([0,\infty))$ such that 
\[
	\bE_x\int_0^\infty g(Y_t)\dd\bar{A}_t=0,\quad \forall x\in [0,\infty).
\]
Then $g|_{\overline{E}}\equiv 0$.  Since $\text{Supp}(A^1)\subset \overline{E}$,  we have
 \[
	\bE_x\int_0^\infty g(Y_t)\dd A^1_t=0,\quad \forall x\in [0,\infty).
\]
Applying \cite[V\S3, Proposition~3.11]{BG68} to $A^1$ and $\bar{A}$,  we obtain that
\[
	A^1_t=\sum_{n\in \bN}\tilde{\mu}_nL^{\hat{c}_n}_t+\tilde{\mu}_\infty \ell_t,
\]
for some $\tilde{\mu}_n\geq 0$ and $\tilde{\mu}_\infty\geq 0$.  

The fact that $X^1$ is a $Q$-process implies that
\begin{equation}\label{eq:421}
	\bE_x\int_0^\infty e^{-\alpha t}1_{\{0\}}(\check{Y}^1_t)\dd t=0
\end{equation}
for $\alpha>0$ and $x\in \text{Supp}(A^1)$.  
However, the left-hand side of \eqref{eq:421} can be expressed as
\[
	\bE_x \int_0^\infty e^{-\alpha A^1_t}1_{\{0\}}(Y_t)\dd A^1_t=\tilde{\mu}_\infty\cdot  \bE_x\int_0^\infty e^{-\alpha A^1_t}\dd\ell_t.  
\]
Hence, we conclude that $\tilde{\mu}_\infty=0$.  

On the other hand,  it is straightforward (or follows from the proof of Theorem~\ref{THM42}) that the time-changed process of the {\color{blue}killed} Brownian motion $Y^0$, with lifetime $\zeta^0$, with respect to the PCAF $(A^1_{t\wedge \zeta^0})_{t\geq 0}$ corresponds precisely to the minimal $Q$-process.  {\color{blue}In particular}, the Revuz measure of 
\[
	A^1_{t\wedge \zeta^0}=\sum_{n\in \bN}\tilde{\mu}_n L^{\hat{c}_n}_{t\wedge \zeta^0}
\]
with respect to $Y^0$ is actually the speed measure $\mu$.  Therefore,  we can conclude that $\tilde{\mu}_n=\mu_n$ for all $n\in \bN$.  
\end{proof}

\section{Parameters of birth-death processes obtained by time change}


According to Theorem~\ref{THMB1},   the $Q$-process $\Xi(\check{Y})$ obtained in Theorem~\ref{THM42} (excluding the first case,  which yields the minimal $Q$-process) admits a resolvent representation given by a triple $(\gamma,\beta,\nu)$,  which is unique up to a multiplicative constant.  It is certainly interesting to explore how the parameters $(p_1,p_2,p_3,p_4)$ of Feller's Brownian motion determine $(\gamma,\beta,\nu)$.  

\subsection{Main result}

Recall that the sequence $\{\hat{c}_n:n\in \bN\}$ is given in \eqref{eq:419}.  Based on the measure $p_4$ on $(0,\infty)$,  we define a sequence $\{\mathfrak{p}_n: n\in \bN\}$ as follows: 
\begin{equation}\label{eq:55}
	\mathfrak{p}_0:=\int_{(\hat{c}_{1},\hat{c}_0]}\frac{x-\hat{c}_{1}}{\hat{c}_0-\hat{c}_{1}}p_4(\dd x)+p_4\left((\hat{c}_0,\infty)\right)
\end{equation}
and
\begin{equation}\label{eq:56}
	\mathfrak{p}_n:=\int_{(\hat{c}_{n+1},\hat{c}_n]}\frac{x-\hat{c}_{n+1}}{\hat{c}_n-\hat{c}_{n+1}}p_4(\dd x)+\int_{(\hat{c}_{n},\hat{c}_{n-1})}\frac{\hat{c}_{n-1}-x}{\hat{c}_{n-1}-\hat{c}_{n}}p_4(\dd x),\quad n\geq 1.
\end{equation}
Note that if $|p_4|<\infty$,  then $|p_4|=\sum_{n\in \bN}\mathfrak{p}_n$.  Our main result is as follows. 


\begin{theorem}\label{MainTHM}
Let $Y$ be a Feller's Brownian motion with parameters $(p_1,p_2,p_3,p_4)$ as described in Theorem~\ref{THM32} (excluding the first case,  i.e., $p_2=0,p_3>0$, and $|p_4|=0$, in Theorem~\ref{THM42}).  Then the parameters $(\gamma,\beta,\nu)$ that determine the resolvent matrix of the $Q$-process $\Xi(\check{Y})$, obtained in Theorem~\ref{THM42}, are given by
\[
	\gamma=p_1,\quad \beta=2p_2,\quad \nu_n=\mathfrak{p}_n,\; n\in \bN
\]
up to a multiplicative constant, 
where $\{\mathfrak{p}_n\}$ is  the sequence defined by \eqref{eq:55} and \eqref{eq:56}.
\end{theorem}

The proof of Theorem~\ref{MainTHM} will be completed in the the following three sections.  Here, we present a consequence: not only is every time-changed Feller's Brownian motion a $Q$-process, as shown in Theorem~\ref{THM42}, but also every $Q$-process can be derived from a Feller's Brownian motion through time change.

\begin{corollary}\label{COR66}
For every $Q$-process $X$,  there exists a Feller's Brownian motion $Y$ such that the $Q$-process obtained from $Y$ through time change, as described in Theorem~\ref{THM42}, is identical in law to $X$.
\end{corollary}
\begin{proof}
Let $(\gamma,\beta,\nu)$ be the triple determining the resolvent matrix of $X$.  
When $X$ is the minimal $Q$-process, {\color{blue}i.e., $\beta=0$ and $\nu\equiv 0$,}  we can choose $Y$ with $p_2=0,p_3>0$ and $|p_4|=0$.   In this case, the first case of Theorem~\ref{THM42} applies.  When $X$ is a Doob process, i.e., $\beta=0,|\nu|<\infty$, we can choose
\[
	p_1=\gamma,\quad p_2=0,\quad p_3>0,\quad  p_4=\sum_{n\in \bN}\nu_n\delta_{\hat{c}_n}.
\]
Then, we apply the second case of Theorem~\ref{THM42} and Theorem~\ref{MainTHM}. 
When $X$ is non-minimal and non-Doob,  we can choose
\[
	p_1=\gamma,\quad p_2=\frac{\beta}{2},\quad p_3=0,\quad  p_4=\sum_{n\in \bN}\nu_n\delta_{\hat{c}_n}.
\]
(To satisfy \eqref{eq:330}, it may be necessary to multiply all parameters by a positive constant.) 
\end{proof}

\subsection{Pathwise construction of $Q$-processes}\label{SEC52}

In this subsection, we discuss the pathwise construction of a Feller $Q$-process $X$ based on Theorem~\ref{MainTHM}. For simplicity, we focus on the honest case where $\gamma = 0$,  and consider the Feller process $\hat{X}=\Xi^{-1}(X)$ on $\overline{E}$ with parameters $(\nu,\beta, 0)$. 

In the context of $Q$-processes, the $(Q,1)$-process plays a role analogous to that of the reflecting Brownian motion in the context of Feller's Brownian motion. We denote its corresponding $Q$-process on $\overline{E}$ by $\hat W^+$, {\color{blue}which is,  in fact,  a time-changed reflecting Brownian motion. More precisely,  let $W^+$ be the reflecting Brownian motion on $[0,\infty)$, and let $L^{+, \hat c_n}$ denote the local time of $W^+$ at $\hat c_n$, as defined in Definition~\ref{DEF37}. Then $\hat W^+$ is the time-changed process of $W^+$ on $\overline{E}$ induced by the PCAF $A^+_t := \sum_{n \in \mathbb{N}} \mu_n L^{+, \hat c_n}_t$.}  Let $\gamma^+$ be the right-continuous inverse of $A^+$. The local time of $\hat W^+$ at $0$, denoted by $\hat \ell$, can be derived from the Brownian local time $\ell$ through the corresponding time change transformation, specifically, $\hat{\ell}_t=\ell_{\gamma^+_t}$.  

We initially considered replacing $W^+$ with $\hat W^+$ and constructing the general $Q$-process similarly to \eqref{eq:316}, by defining 
\begin{equation}\label{eq:431}
\hat X^1_t := Z(Z^{-1}(\hat \ell_t)) - \hat \ell_t + \hat W^+_t,\quad t\geq 0,
\end{equation}
where $Z$ is the subordinator \eqref{eq:319} with $p_2=\frac{\beta}{2}$ and $p_4=\sum_{n\in \bN}\nu_n\delta_{\hat{c}_n}$,  which is independent of  $\hat W^+$.  (A similar idea appears in \cite[\S20]{IM63}.) However, at a time $t$ when $\hat \ell_t$ increases to a discontinuous point $x$ ($=\hat{\ell}_t$) of $Z(Z^{-1})$,  where $h(x) := Z(Z^{-1}(x)) - x=l$ for some $0 < l \in E$ and $h(x-)=0$, the term $h(\hat \ell_t)$ in \eqref{eq:431} will continuously decrease from $l$ for a short period thereafter (see \cite[\S12]{IM63}).  Meanwhile, $\hat W^+$ `diffuses’ within the space $\overline E$ near $0$. This discrepancy causes the process defined by  \eqref{eq:431} to move outside of $\overline{E}$.

Note that \eqref{eq:431} can be rewritten as
\[
\hat X^1_t= Z(Z^{-1}(\ell_{\gamma^+_t})) - \ell_{\gamma^+_t} +  W^+_{\gamma^+_t},\quad t\geq 0.
\]
However, according to \eqref{eq:316} and Theorem~\ref{MainTHM},  the correct pathwise representation of $\hat{X}$ should be
\[
	\hat X_t=Z(Z^{-1}(\ell_{\gamma_t})) -  \ell_{\gamma_t} +  W^+_{\gamma_t},\quad t\geq 0,
\]
where $\gamma$ is the right-continuous inverse of \eqref{eq:41}.  In other words,  \eqref{eq:431} incorrectly reverses the order of adding jumps using $Z$ and the transformation of time change.

\section{Proof of Theorem~\ref{MainTHM} for $|p_4|<\infty$}\label{SEC6}

In this section,  we will prove Theorem~\ref{MainTHM} for various cases where $|p_4|<\infty$.  Simultaneously,  we will provide a more comprehensive characterization of the corresponding $Q$-processes.  From now on,   we will denote $X_t:=\Xi(\check{Y}_t)$ for $t\geq 0$ and use $X:=(X_t)_{t\geq 0}$ for convenience. 

\subsection{Doob processes}

According to Theorem~\ref{THM42},  $X$ is a Doob process (but not the minimal $Q$-process) if and only if $p_2=0,p_3>0$, and $0<|p_4|<\infty$.  

\begin{theorem}
If $p_2=0,p_3>0$ and $0<|p_4|<\infty$,  then the parameters $(\gamma,\beta,\nu)$ that determine the resolvent matrix of the Doob process $X=\Xi(\check{Y})$ are given by 
\[
	\gamma=p_1,\quad \beta=0,\quad \nu_n=\mathfrak{p}_n,\quad n\in \bN.
\]
{\color{blue}In particular},  the instantaneous distribution of $X$ is
\[
	\pi(\{\partial\})=\frac{p_1}{p_1+|p_4|},\quad \pi(\{n\})=\frac{\mathfrak{p}_n}{p_1+|p_4|},\; n\in \bN.
\]
\end{theorem}
\begin{proof}
The fact that $\beta=0$ follows directly from Theorems \ref{THMB1} and \ref{THM42}.  Fix $\bP_{\hat{c}_i}$ for some $i\in \bN$, and let $\check{\eta}_\infty$ be defined as \eqref{eq:420},  representing the first flying time (see \cite[Corollary~4.7]{L23}) of $\check{Y}$.  According to \cite[Theorem~5.1]{L23},  it suffices to demonstrate that
\begin{equation}\label{eq:51}
\bP_{\hat c_i}(\check{Y}_{\check{\eta}_\infty}=\partial)=\frac{p_1}{p_1+|p_4|},\quad \bP_{\hat c_i}(\check{Y}_{\check{\eta}_\infty}=\hat{c}_n)=\frac{\mathfrak{p}_n}{p_1+|p_4|},\quad n\in \bN.
\end{equation}
We will use the same symbols as those in the proof of Theorem~\ref{THM42}.  
Define
\[
\sigma_1:=\inf\{t>\zeta^0:Y_t\neq 0\},\quad \sigma_2:=\inf\{t>\sigma_1: Y_t\in E\},
\]
and 
\[
	\sigma:=\inf\{t>0:Y_t\in E\}.
\]
Note that $\check{\eta}_\infty=A_{\zeta^0}$. 
From the pathwise representation of $Y$ in \S\ref{SEC321},  it follows that
\[
	\gamma_{\check{\eta}_\infty}=\inf\{t>0:A_t>A_{\zeta^0}\}=\sigma_2.  
\]
It is straightforward to verify that $Y_{\sigma_2}=Y_{\sigma}\circ \theta_{\sigma_1}$ and that $\{Y_{\sigma_2}=\hat{c}_n\}\subset \{Y_{\sigma_1}\in (\hat{c}_{n+1},\hat{c}_{n-1})\}$ (with the convention $\hat{c}_{-1}:=\infty$).   
By the strong Markov property of $Y$,  we have
\[
\begin{aligned}
	\bP_{\hat c_i}(\check{Y}_{\check{\eta}_\infty}=\hat{c}_n)&=\bP_{\hat{c}_i}(Y_{\sigma_2}=\hat{c}_n)=\bP_{\hat{c}_i}\left(Y_\sigma\circ \theta_{\sigma_1}=\hat{c}_n, Y_{\sigma_1}\in (\hat{c}_{n+1},\hat{c}_{n-1})\right) \\
	&=\bP_{\hat{c}_i}\left(\bP_{Y_{\sigma_1}}(Y_\sigma=\hat{c}_n); Y_{\sigma_1}\in (\hat{c}_{n+1},\hat{c}_{n-1})  \right).
\end{aligned}\]
Note that $\bP_{\hat{c}_i}\left(Y_{\sigma_1}\in \dd x \right)=\lambda(\dd x)$,  where $\lambda$ is given by \eqref{eq:320}.  For $x\in (\hat{c}_{n+1},\hat{c}_{n-1})$, we have
\[
\bP_x(Y_\sigma=\hat{c}_n)=\left\lbrace
	\begin{aligned}
	&\frac{x-\hat{c}_{n+1}}{\hat{c}_n-\hat{c}_{n+1}}, \quad & x\in (\hat c_{n+1},\hat{c}_n], \\
	&\frac{\hat{c}_{n-1}-x}{\hat{c}_{n-1}-\hat{c}_{n}}, \quad  & x\in (\hat{c}_n,\hat{c}_{n-1}),
	\end{aligned}
\right.
\]
with the convention $\frac{\infty}{\infty}:=1$.  Thus,  a straightforward computation yields
\[
	\bP_{\hat c_i}(\check{Y}_{\check{\eta}_\infty}=\hat c_n)=\int_{(\hat{c}_{n+1},\hat{c}_{n-1})} \bP_x(Y_\sigma=\hat{c}_n)\lambda(\dd x)=\frac{\mathfrak{p}_n}{p_1+|p_4|}.  
\]
Finally,  $\bP_{\hat c_i}(\check{Y}_{\check{\eta}_\infty}=\partial)=1-\sum_{n\in \bN}\bP_{\hat c_i}(\check{Y}_{\check{\eta}_\infty}=\hat c_n)=\frac{p_1}{p_1+|p_4|}$.  Therefore, \eqref{eq:51} is established.  This completes the proof.
\end{proof}

\subsection{Symmetric case}

The symmetric case,  where $p_2>0$ and $|p_4|=0$,  can be analyzed using Dirichlet form theory. In this framework, the time change of a Markov process corresponds to a trace Dirichlet form (see \cite[Chapter 5]{CF12}). 

To state our results,  we first define the following quadratic form for a function $f$ on $\bN$:
\[
	\sA(f,f):=\sum_{k\in \bN}\frac{(f(k+1)-f(k))^2}{c_{k+1}-c_k}.  
\]
If $\sA(f,f)<\infty$,  it follows from $c_\infty=\lim_{k\rightarrow \infty}c_k<\infty$ and the Cauchy-Schwarz inequality that $f(\infty):=\lim_{k\rightarrow \infty}f(k)$ exists.  

\begin{theorem}\label{THM52}
If $p_2>0$ and $|p_4|=0$,  then the parameters $(\gamma,\beta,\nu)$ determining the resolvent matrix of the $Q$-process $X=\Xi(\check{Y})$ are 
\begin{equation}\label{eq:54}
	\gamma=p_1,\quad \beta=2p_2,\quad \nu_n=0,\; n\in \bN.
\end{equation}
Furthermore,  $X$ is symmetric with respect to the speed measure $\mu$,  and the associated Dirichlet form on $L^2(\bN\cup\{\infty\},\mu)$ is given by
\[
	\begin{aligned}
		&\check{\sF}=\{f\in L^2(\bN\cup \{\infty\},\mu): \sA(f,f)<\infty\} \\
		&\check{\sE}(f,g)=\frac{1}{2}\sum_{k\in \bN}\frac{(f(k+1)-f(k))(g(k+1)-g(k))}{c_{k+1}-c_k}+\frac{p_1}{2p_2}f(\infty)g(\infty),\quad f,g\in \check{\sF}.
	\end{aligned}
\]
\end{theorem}
\begin{proof}
In this case,  the Revuz measure of the PCAF \eqref{eq:41} with respect to $Y$ is exactly $\tilde\mu:=\sum_{k\in \bN}\mu_k\delta_{\hat{c}_k}$,  by Definition~\ref{DEF37} and \cite[Proposition~4.1.10]{CF12}.  According to \cite[Theorem~5.2.2]{CF12}, the time-changed process $\check{Y}$ is symmetric with respect to $\tilde{\mu}$,  and its associated Dirichlet form $(\tilde\sE, \tilde \sF)$ on $L^2(\overline{E},\tilde{\mu})$ is derived in \cite[(5.2.4)]{CF12}. 

 In what follows,  we will compute $(\tilde{\sE},\tilde{\sF})$.  Regarding $\tilde{\sF}$,  we note that the extended Dirichlet space of \eqref{eq:313} is 
 \[
 	H^1_e([0,\infty)):=\{f\text{ is absolutely continuous on }[0,\infty)\text{ and }f'\in L^2([0,\infty))\},
 \]
as detailed in \cite[Theorem~2.2.11]{CF12}.  According to \cite[(5.2.4)]{CF12},  it is straightforward to compute that
 \begin{equation}\label{eq:52}
 \tilde{\sF}=H^1_e([0,\infty))|_{\overline{E}}\cap L^2(\overline{E},\tilde{\mu})=\{\Xi(f):f\in \check{\sF}\},
 \end{equation}
 where $\Xi(f)(\hat{c}_n):=f(n)$ for all $n\in \bN$.  The quadratic form $\tilde{\sE}$ can be formulated using \cite[Theorem~5.5.9]{CF12}.  The crucial step is to compute the \emph{Feller measure} $U$ on $\overline{E}\times \overline{E}$ and the \emph{supplementary Feller measure} $V$ on $\overline{E}$,  as defined in \cite[(5.5.7)]{CF12}.  In fact,  by mimicking the proof of \cite[Theorem~2.1]{LY17},  we can show that the strongly local part of $\tilde{\sE}$ vanishes, and $U$ is supported on $\{(\hat{c}_n,\hat{c}_{n+1}), (\hat{c}_{n+1},\hat c_n): n\in \bN\}\subset \overline{E}\times \overline{E}$ (with the notation $(\cdot, \cdot)$ indicating a pair of points, not an interval) with
 \[
 	U(\{(\hat{c}_n,\hat{c}_{n+1})\})=U(\{(\hat{c}_{n+1},\hat c_n) \})=\frac{1}{2|\hat c_{n}-\hat c_{n+1}|}=\frac{1}{2|c_{n+1}-c_n|}.  
 \]
 Regarding the measure $V$,  we note that $\bP_x(\tau\geq \zeta)=0$ for any $x\in [0,\infty)\setminus \overline{E}$,  where $\tau:=\inf\{t\in [0,\zeta]: Y_t\notin [0,\infty)\setminus \overline{E}\}$.  Thus,  by the definition \cite[(5.5.7)]{CF12} of $V$,  we have $V\equiv 0$.  Therefore,  applying \cite[Theorem~5.5.9]{CF12}, we obtain that for $u\in \tilde{\sF}$,
 \begin{equation}\label{eq:53}
 	\tilde\sE(u,u)=\frac{1}{2}\sum_{k\in \bN}\frac{(u(\hat{c}_{k+1})-u(\hat{c}_k))^2}{c_{k+1}-c_k}+\frac{p_1}{2p_2}u(0)^2.
 \end{equation}
 Applying the spatial transformation $\Xi$ to \eqref{eq:52} and \eqref{eq:53} yields the expression for $(\check{\sE},\check{\sF})$.  
 
 Finally,  the parameters $(\gamma,\beta,\nu)$ are given by  \eqref{eq:54},  as discussed in \cite[Lemma~7.1 and Theorem~8.1]{L23}. 
\end{proof}

\subsection{Non-symmetric case with finite jumping measure}

We turn to examine the non-symmetric case where $p_2>0$ and $0<|p_4|<\infty$.  The sample path representation of Feller's Brownian motion for this case is detailed in \S\ref{SEC322}.  Let $\pi$ be a probability measure on $\bN\cup \{\partial\}$ given by
\begin{equation}\label{eq:57}
\pi(\{\partial\}):=\frac{p_1}{p_1+|p_4|},\quad 	\pi(\{n\})=\pi_n:=\frac{\mathfrak{p}_n}{p_1+|p_4|},\quad n\in \bN,
\end{equation}
where $\{\mathfrak{p}_n\}$ is specified by \eqref{eq:55} and \eqref{eq:56}.

\begin{theorem}
If $p_2>0$ and $0<|p_4|<\infty$,  then the parameters $(\gamma,\beta,\nu)$ determining the resolvent matrix of the $Q$-process $X=\Xi(\check{Y})$ are 
\begin{equation}\label{eq:510}
	\gamma=p_1,\quad \beta=2p_2,\quad \nu_n=\mathfrak{p}_n,\quad n\in \bN.
\end{equation}
Furthermore,  $X$ can be obtained by piecing out $X^1$ with respect to the probability measure $\pi$ on $\bN\cup \{\partial\}$,  where $X^1$ is the symmetric $Q$-process corresponding to the parameters $(p_1+|p_4|, 2p_2, 0)$ and $\pi$ is defined as \eqref{eq:57}.
\end{theorem}

\begin{proof}
Let $Y^1$, with lifetime $\zeta^1$, be the symmetric Feller's Brownian motion with parameters $(p_1+|p_4|, p_2,p_3,0)$,  as discussed in \S\ref{SEC322}.  It represents the killed process of $Y$ at $\zeta^1$,  and $Y$ can be obtained by piecing out $Y^1$ with respect to the probability measure $\lambda$ given by \eqref{eq:320}.  {\color{blue}In particular},  we have
\[
	\bP_x(Y_{\zeta^1}\in \cdot)=\lambda(\cdot),\quad \forall x\in [0,\infty).
\]
By following the steps used to obtain $\hat{Y}^0=\check{Y}^0$ in the proof of Theorem~\ref{THM42},  we can also show that the killed process $\check{Y}^1$ of $\check{Y}$ at time $\check{\zeta}^1:=A_{\zeta^1}$ is identical in law to the time-changed process of $Y^1$ with respect to the PCAF $A^1_t:=A_{t\wedge \zeta^1}, t\geq 0$.  Note that the Revuz measure of $A^1$ with respect to $Y^1$ is exactly $\sum_{n\in \bN}\mu_n\delta_{\hat{c}_n}$.  Hence,  according to Theorem~\ref{THM52},  $\Xi(\check{Y}^1)$ is identical in law to $X^1$.  

Denote by $\tilde R_\alpha^1$ and $\tilde{R}_\alpha$ the resolvents of $\check{Y}^1$ and $\check{Y}$, respectively.  For a bounded function $f$ on $\overline{E}$,  Dynkin's formula gives us:
\begin{equation}\label{eq:58}
	\tilde{R}_\alpha f(\hat{c})=\bE_{\hat{c}}\int_0^\infty e^{-\alpha t}f(\check{Y}_t)\dd t=\tilde{R}^1_\alpha f(\hat{c})+\bE_{\hat{c}}\left(e^{-\alpha \check{\zeta}^1}\tilde{R}_\alpha f(\check{Y}_{\check{\zeta}^1})\right),\quad \forall \hat{c}\in \overline{E}.
\end{equation}
Let $\sigma:=\inf\{t>0:Y_t\in E\}$ and $\sigma_1:=\inf\{t>\zeta^1: Y_t\in E\}$.  Then,
\[
	\gamma_{\check{\zeta}^1}=\inf\{t>0:A_t>A_{\zeta^1}\}=\sigma_1.  
\]
Note that $Y_{\sigma_1}=Y_\sigma\circ \theta_{\zeta^1}$. According to the construction procedures of piecing out as stated in \cite[Appendix~A]{L23},  it is straightforward to see that $Y_{\zeta^1}$ (with distribution $\lambda$) is independent of $Y^1$.  Thus, $Y_{\zeta^1}$ is also independent of $\check{\zeta}^1=A_{\zeta^1}$.  From this independence and the strong Markov property of $Y$,  it follows that
\begin{equation}\label{eq:59}
\begin{aligned}
\bE_{\hat{c}}\left(e^{-\alpha \check{\zeta}^1}\tilde{R}_\alpha f(\check{Y}_{\check{\zeta}^1})\right)&=\bE_{\hat{c}}\left(e^{-\alpha \check{\zeta}^1}\tilde{R}_\alpha f(Y_\sigma\circ \theta_{\zeta^1})\right)\\
&=\bE_{\hat{c}}\left(e^{-\alpha \check{\zeta}^1}\bE_{Y_{\zeta^1}}\left(\tilde{R}_\alpha f(Y_\sigma)\right)\right) \\
&=\bE_{\hat{c}} \left(e^{-\alpha \check{\zeta}^1}\right)\cdot \bE_{\lambda}\left(\tilde{R}_\alpha f(Y_\sigma)\right). 
\end{aligned}\end{equation}
{\color{blue}Similar to \eqref{eq:51}},  we can also deduce that $\bP_\lambda(Y_\sigma\in \cdot)=\tilde\pi(\cdot)$, where $\tilde \pi(\{\hat{c}_n\}):=\pi_n$ for all $n\in \bN$ and $\tilde{\pi}(\{\partial\}):=\pi(\{\partial\})$.  Substituting \eqref{eq:59} into \eqref{eq:58} yields
\[
\tilde{R}_\alpha f(\hat{c})=\tilde{R}^1_\alpha f(\hat{c})+\bE_{\hat{c}}\left(e^{-\alpha \check{\zeta}^1}\right)\cdot \tilde{\pi}(\tilde{R}_\alpha f).
\]
Integrating both sides with respect to $\tilde{\pi}$,  we get 
\[
	\tilde{\pi}(\tilde{R}_\alpha f)=\frac{\tilde{\pi}(\tilde{R}^1_\alpha f)}{1-\bE_{\tilde{\pi}} \left(e^{-\alpha \check{\zeta}^1}\right)}.
\]
It follows that
\[
	\tilde{R}_\alpha f(\hat{c})=\tilde{R}^1_\alpha f(\hat{c})+\frac{\tilde{\pi}(\tilde{R}^1_\alpha f)}{1-\bE_{\tilde{\pi}} \left(e^{-\alpha \check{\zeta}^1}\right)}\cdot \bE_{\hat{c}}\left(e^{-\alpha \check{\zeta}^1}\right),\quad \forall \hat{c}\in \overline{E}.
\]
Finally,  repeating the argument in the proof of \cite[Theorem~8.1]{L23} (the steps after (8.5)) and applying the spatial transformation to $\check{Y}$,  we can conclude that $X$ is the piecing out of $X^1$ with respect to $\pi$.  {\color{blue}In particular},  the parameters $(\gamma,\beta,\nu)$ for $X$ are given by \eqref{eq:510}. 
\end{proof}

\section{Approximation of Feller's Brownian motion}\label{SEC7}

We now turn to the `pathological’ case where $|p_4|=\infty$. Our primary method involves constructing a sequence of Markov processes that converge to Feller’s Brownian motion, using the strategy outlined in the construction of $Q$-processes in \cite[\S6.1-\S6.6]{WY92}. Each constituent process is the piecing out of the {\color{blue}killed} Brownian motion with respect to a specific probability measure $\lambda^{(n)}$ on $(0,\infty)\cup \{\partial\}$, whose resolvent is expressed as \eqref{eq:325} with $\lambda=\lambda^{(n)}$. {\color{blue}Recall that these processes are referred to as Doob's Brownian motions in Remark~\ref{RM32}, with $\lambda^{(n)}$ referred to as the instantaneous distribution (of the piecing out transformation)}.


\subsection{Approximating sequence of Doob's Brownian motions}

We begin by introducing an important transformation on the sample paths. Let $x(t)$ be a right-continuous function on $[0,\infty)\cup \{\partial\}$.  Consider two sequences of positive constants $(\alpha_m)$ and $(\beta_m)$ such that 
\[
	0(=:\beta_0)<\alpha_1\leq \beta_1<\alpha_2\leq \beta_2<\cdots.
\]
({\color{blue}These sequences can be finite}.) We say that the function $y(t)$ is obtained from $x(t)$ by the $C(\alpha_m,\beta_m)$-transformation if 
\[
\begin{aligned}
	&y(t)=x(t),  &\quad 0\leq t<\alpha_1,\\
	&y(d_m+t)=x(\beta_m+t), &\quad 0\leq t<\alpha_{m+1}-\beta_m,
\end{aligned}\]
where $d_1:=\alpha_1$ and $d_{m+1}:=d_m+(\alpha_{m+1}-\beta_m)$.  Intuitively speaking, the $C(\alpha_m,\beta_m)$-transformation discards the path of $x(t)$ corresponding to the interval $[\alpha_m, \beta_m)$, keeps the segment $[0, \alpha_1)$ unchanged, and shifts the remaining parts to the left, connecting them in the original order without intersection, thereby obtaining a new right-continuous path $y(t)$.

Fix $n\in \bN$.  Define $\eta:=\inf\{t>0:Y_{t-}=0\}$ and $\sigma^{(n)}:=\inf\{t> 0: Y_t\geq \hat{c}_n\}\wedge \zeta$ ($\inf \emptyset:=\infty$).  We then define a sequence of stopping times as follows:
 \[
 \eta^{(n)}_1:=\eta,\quad	\sigma^{(n)}_1:=\inf\{t\geq \eta_1^{(n)}: {\color{blue}Y_t}\geq \hat{c}_n\}\wedge \zeta,
 \]
 and if $\eta^{(n)}_{m-1}, \sigma^{(n)}_{m-1}$ are already defined,  we set
 \[
 	\eta^{(n)}_m:=\inf\{t\geq \sigma^{(n)}_{m-1}: Y_{t-}=0\}\wedge \zeta
 \]
and 
\[
	\sigma^{(n)}_m:=\inf\{t\geq \eta^{(n)}_m: Y(t)\geq \hat{c}_n\}\wedge \zeta.
\]
Note that if $\eta^{(n)}_{m}<\zeta$,  then $Y_{\eta^{(n)}_m}=0$ due to the quasi-left-continuity of $Y$.  {\color{blue}In particular}, $\eta=\tau_0=\inf\{t>0:Y_t=0\}$.  We define $\eta$ using the left limit of $Y$ instead of directly using $\tau_0$ because, although our focus is on Feller's Brownian motion, the discussion in this section also applies to Doob's Brownian motion.  For Doob's Brownian motion, defining the `return’ time to $0$ requires the use of the left limit.

For $n\in \bN$ and every $\omega\in \Omega$,   by performing the $C(\eta^{(n)}_m(\omega),\sigma^{(n)}_m(\omega))$-transformation on $Y_t(\omega)$,  we obtain a new path,  denoted by $Y^{(n)}_t(\omega)$.  {\color{blue}Note that this new path is clearly right continuous.}

{\color{blue}
\begin{lemma}\label{LM71}
For any $t\geq 0$,  $Y^{(n)}_t$ is $\sG$-measurable. 
\end{lemma}
\begin{proof}
Let $\zeta_1:=\eta^{(n)}_1=\eta$ and for $m\geq 2$, define
 \begin{equation}\label{eq:7111}
 	\zeta_m:=\eta^{(n)}_m-\sigma^{(n)}_{m-1}.
 \end{equation}
By the above construction of $Y^{(n)}$,  we have
\[
	Y^{(n)}_t =Y_t\cdot 1_{[0,\zeta_1)}(t)+\sum_{m\geq 2}Y_{t-\sum_{i=1}^m\zeta_i+\sigma^{(n)}_m}\cdot 1_{[\sum_{i=1}^m\zeta_i,  \sum_{i=1}^{m+1}\zeta_i)}(t)+\partial \cdot 1_{[\sum_{i= 1}^\infty\zeta_i,\infty)}(t). 
\]
Since both $\eta^{(n)}_m$ and $\sigma^{(n)}_m$ are $\sG_t$-stopping times,  it is straightforward to verify that $Y^{(n)}_t$ is $\sG$-measurable. 
\end{proof}

Let $\sG^{(n)}_t$ denote the augmentation of the $\sigma$-algebra generated by $\{Y^{(n)}_s:0\leq s\leq t\}$ under $\bP_\mu$, where $\mu$ ranges over all probability measures on $(0,\infty)$. The preceding lemma implies that $\sG^{(n)}_t\subset \sG$.  Next, define the shift operator $\theta^{(n)}_t:\Omega\rightarrow \Omega$ by requiring that $Y^{(n)}_s(\theta^{(n)}_t\omega)=Y^{(n)}_{s+t}(\omega)$ for all $s \geq 0$.  For any $f\in \mathcal{B}_b((0,\infty))$ and $t\geq 0$,  set
\[
	T^{(n)}_t f(x):=\bE_x f(Y^{(n)}_t),\quad x\in (0,\infty).
\]
Since the path $t\mapsto Y^{(n)}_t$ is right continuous,  it follows that for any $f\in C_b((0,\infty))$, 
\begin{equation}\label{eq:72222}
	t\mapsto T_t^{(n)}f(x)
\end{equation}
is right continuous.
}

\begin{lemma}\label{LM61}
The process 
\begin{equation}\label{eq:7222}
	Y^{(n)}:=\left(\Omega, \sG,  {\color{blue}\sG^{(n)}_t}, Y^{(n)}_t,{\color{blue}\theta^{(n)}_t}, (\bP_x)_{x\in (0,\infty)}\right)
\end{equation}
 is a Doob's Brownian motion with instantaneous distribution 
 \[
 \lambda^{(n)}(\cdot)=\bP_0(Y_{\sigma^{(n)}}\in \cdot)
\]
supported on $[\hat{c}_n,\infty)\cup \{\partial\}$. 
\end{lemma}  
\begin{proof} 
Fix $x\in (0,\infty)$ and $\alpha>0$.  We aim to compute the resolvent of $Y^{(n)}$ for a positive and bounded Borel measurable function $f$ on $(0,\infty)$:
\[
	G^{(n)}_\alpha f(x)=\bE_x \int_0^\infty e^{-\alpha t}f(Y^{(n)}_t)\dd t.
\]
By the definition of $Y^{(n)}_t$,  we have
\begin{equation}\label{eq:62}
	G^{(n)}_\alpha f(x)=\bE_x\int_0^{\zeta_1} e^{-\alpha t}f(Y_t)\dd t+\sum_{m\geq 1}\bE_x\int_{\sum_{i=1}^m\zeta_i}^{\sum_{i=1}^{m+1}\zeta_i}e^{-\alpha t}f(Y_{t-\sum_{i=1}^m\zeta_i+\sigma^{(n)}_m})\dd t,
\end{equation}
where $\zeta_1=\eta$ and $\zeta_m$, $m\geq 2$, are defined in \eqref{eq:7111}. 
Set 
\[
\begin{aligned}
	J_m&:=\bE_x\int_{\sum_{i=1}^m\zeta_i}^{\sum_{i=1}^{m+1}\zeta_i}e^{-\alpha t}f(Y_{t-\sum_{i=1}^m\zeta_i+\sigma^{(n)}_m})\dd t\\
	&=\bE_x e^{-\alpha \sum_{i=1}^m\zeta_i}\int_0^{\zeta_{m+1}}e^{-\alpha t}f(Y_{t+\sigma^{(n)}_m})\dd t.
\end{aligned}\]
Note that $\zeta_{m+1}=\eta\circ \theta_{\sigma^{(n)}_m}$.  It follows from the strong Markov property that
\[
\begin{aligned}
	J_m&=\bE_x \left(e^{-\alpha \sum_{i=1}^m\zeta_i}\bE_{Y_{\sigma^{(n)}_m}}\int_0^\eta e^{-\alpha t}f(Y_t)\dd t\right) \\
	&=\bE_x \left(e^{-\alpha \sum_{i=1}^m\zeta_i}G^0_\alpha f(Y_{\sigma^{(n)}_m})\right),
	\end{aligned}
\]
where $G^0_\alpha$ is the resolvent of the {\color{blue}killed} Brownian motion. 

Let us prove that
\begin{equation}\label{eq:61}
	J_m=\bE_x e^{-\alpha \eta}\cdot \left(\bE_{\lambda^{(n)}}e^{-\alpha \eta}\right)^{m-1}\cdot \int_0^\infty G^0_\alpha f(x) \lambda^{(n)}(\dd x),\quad m\geq 1. 
\end{equation}
For $m=1$,  we have $Y_{\sigma^{(n)}_1}=Y_{\sigma^{(n)}}\circ \theta_{\eta}$,  and thus, by the strong Markov property of $Y$,  
{\color{blue}\[
\begin{aligned}
	J_1&=\bE_x \left(e^{-\alpha \eta}G^0_\alpha f(Y_{\sigma^{(n)}}\circ \theta_\eta)\right)=\bE_x \left(e^{-\alpha \eta}\bE_{Y_\eta}\left(G^0_\alpha f \left(Y_{\sigma^{(n)}}\right)\right)\right) \\
	&=\bE_x e^{-\alpha \eta} \cdot \bE_0\left(G^0_\alpha f \left(Y_{\sigma^{(n)}}\right)\right)=\bE_x e^{-\alpha \eta}\cdot \int_0^\infty G^0_\alpha f(x) \lambda^{(n)}(\dd x).
\end{aligned}\]}
In general,  note that $Y_{\sigma^{(n)}_i}=Y_{\sigma^{(n)}}\circ \theta_{\eta^{(n)}_i}$ and  $\zeta_i=\eta^{(n)}_i-\sigma^{(n)}_{i-1}=\eta\circ \theta_{\sigma^{(n)}_{i-1}}$ for all $i\geq 2$.  
Then the case $m=2$ can be deduced by the strong Markov property as follows:
\[
\begin{aligned}
	J_2&=\bE_x \left(e^{-\alpha (\zeta_1+\zeta_2)}\bE_{Y_{\eta^{(n)}_2}}\left({\color{blue}G^0_\alpha f} \left(Y_{\sigma^{(n)}}\right)\right)\right) \\
	&=\bE_x e^{-\alpha(\eta+\eta\circ \theta_{\sigma^{(n)}_1})}\cdot \int_0^\infty G^0_\alpha f(x) \lambda^{(n)}(\dd x) \\
	&=\bE_x \left(e^{-\alpha \eta} \bE_{Y_{\sigma^{(n)}_1}}\left(e^{-\alpha \eta} \right) \right)\cdot \int_0^\infty G^0_\alpha f(x) \lambda^{(n)}(\dd x). \\
	\end{aligned}
\]
Using $Y_{\sigma^{(n)}_1}=Y_{\sigma^{(n)}}\circ \theta_\eta$ and $Y_\eta=0$,  we obtain that
\[
\begin{aligned}
	J_2&=\bE_x \left(e^{-\alpha \eta} \bE_{Y_{\sigma^{(n)}}\circ \theta_\eta}\left(e^{-\alpha \eta} \right) \right)\cdot \int_0^\infty G^0_\alpha f(x) \lambda^{(n)}(\dd x) \\
	&=\bE_x e^{-\alpha \eta}\cdot \bE_0 \left(\bE_{Y_{\sigma^{(n)}}}e^{-\alpha \eta} \right)\cdot \int_0^\infty G^0_\alpha f(x) \lambda^{(n)}(\dd x)\\
	&=\bE_x e^{-\alpha \eta}\cdot \bE_{\lambda^{(n)}}\left(e^{-\alpha \eta} \right)\cdot \int_0^\infty G^0_\alpha f(x) \lambda^{(n)}(\dd x).
\end{aligned}\]
The general case can be formulated similarly.  

By substituting \eqref{eq:61} into \eqref{eq:62},  we conclude that
\begin{equation}\label{eq:625}
G^{(n)}_\alpha f(x)=G_\alpha^0 f(x)+\bE_x e^{-\alpha \eta} \frac{\int_0^\infty G^0_\alpha f(x) \lambda^{(n)}(\dd x)}{1-\bE_{\lambda^{(n)}}\left(e^{-\alpha \eta} \right)}.
\end{equation}
{\color{blue}Recall that the resolvent of a Doob's Brownian motion is expressed as \eqref{eq:325}.} This result corresponds precisely to the resolvent of the Doob's Brownian motion with instantaneous distribution $\lambda^{(n)}$.  
\end{proof}

{\color{blue}\begin{remark}
As we noted in Remark~\ref{RM32},  a Doob's Brownian motion corresponds to a Ray process on $[0,\infty]$,  where $\infty$ is regarded as the cemetery.  This means,  the extension $\bar{G}^{(n)}_\alpha$ of $G^{(n)}_\alpha$ to  $[0,\infty]$ is a Ray resolvent.  (For the definition of Ray resolvent,  see \cite[Definition~9.4]{S88}.)  According to \cite[Theorem~9.8]{S88},  $\bar{G}^{(n)}_\alpha$ corresponds to a unique Ray semigroup $(\bar T^{(n)}_t)_{t\geq 0}$,  and by the right continuity of \eqref{eq:72222},  the restriction of $\bar{T}^{(n)}_t$ to $(0,\infty)$ is exactly $T^{(n)}_t$.  In light of \cite[Theorem~9.13]{S88},  $T^{(n)}_t$ is a right semigroup in the sense of \cite[Definition~8.1]{S88}.  Then the collection in \eqref{eq:7222} is a right continuous realization of $T^{(n)}_t$ in the sense of \cite[Definition~2.2]{S88}.  Applying \cite[Theorem~19.3]{S88},  we can conclude that \eqref{eq:7222} is a (Borel) right process, which, in particular,  satisfies the strong Markov property.

It should be emphasized that in the following discussion, we rely solely on the resolvent representation \eqref{eq:625} of $Y^{(n)}$ and the right continuity of its sample paths, without invoking any further properties of the process.
\end{remark}}

The sequence of Doob's Brownian motions $Y^{(n)}$ approximates $Y$ in the case where $p_3=0$ in the following sense.

\begin{theorem}\label{THM62}
Assume that $p_3=0$.  Then, for any $x\in (0,\infty)$,
\begin{equation}\label{eq:719}
	\bP_x\left(\lim_{n\rightarrow \infty}Y^{(n)}_t=Y_t, \forall t\geq 0\right)=1.
\end{equation}
\end{theorem}
\begin{proof}
A crucial consequence of the assumption $p_3=0$ is that
\[
	\left| \{t\geq 0: Y_t=0\}\right|=0,\quad \bP_x\text{-a.s.},
\]
where $|\cdot|$ denotes the Lebesgue measure of the time set; see, e.g., \cite[\S14]{IM63}.  This can also be observed from the pathwise representation \eqref{eq:316},  as $Y_t=0$ always indicates $W^+_t=0$. 
Hence,  in principle,  we may repeat the proof of Theorem~\ref{THMA4} step by step to complete the proof.  Below, we provide a concise proof using the pathwise representation of $Y_t$.

For $t\geq 0$,  let $\rho^{(n)}_t(\omega)$ denote the total duration discarded from the path of $Y_\cdot(\omega)$ when constructing $Y^{(n)}_s$ for $0\leq s\leq t$.  Specifically, we have $Y^{(n)}_t(\omega)=Y_{{t+\rho^{(n)}_t}(\omega)}(\omega)$.   Since $\rho^{(n)}_{t'}(\omega)\leq \rho^{(n)}_t(\omega)$ for all $0\leq t'\leq t$,   it suffices to show
\begin{equation}\label{eq:616}
	\bP_x(\lim_{n\rightarrow \infty}\rho^{(n)}_t=0)=1
\end{equation}
for any fixed $t\geq 0$.  

Let us first consider the special case where $Y$ is the reflecting Brownian motion $W^+$.  To differentiate this specific case from the general one, we denote $\eta,\sigma^{(n)},\eta^{(n)}_m,\sigma^{(n)}_m,\rho^{(n)}_t$ by $\tilde\eta,\tilde\sigma^{(n)},\tilde\eta^{(n)}_m,\tilde\sigma^{(n)}_m,\tilde\rho^{(n)}_t$, respectively.  Clearly,  the instantaneous distribution of the approximating Doob's Brownian motion, denoted by $W^{+,(n)}$, is $\delta_{\hat{c}_n}$.  Since $W^{+,(n)}$ is conservative by the expression of its resolvent in \eqref{eq:325}, we have $\tilde{\rho}^{(n)}_t<\infty$ for all $n$.  Note that $\tilde\rho^{(n)}_t$ is decreasing as $n\rightarrow \infty$.  It follows that
\[
\begin{aligned}
	\tilde \rho^{(n)}_t(\omega)&\leq \left|\left\{0\leq s\leq t+\tilde{\rho}^{(n)}_t(\omega):W^+_s<\hat{c}_n\right\}\right| \\
	&\leq \left|\left\{0\leq s\leq t+\tilde{\rho}^{(0)}_t(\omega):W^+_s<\hat{c}_n\right\}\right|,
\end{aligned} \]
which implies that
\begin{equation}\label{eq:617}
\begin{aligned}
	\lim_{n\rightarrow \infty}\tilde \rho^{(n)}_t(\omega)&\leq \lim_{n\rightarrow \infty}\left|\left\{0\leq s\leq t+\tilde{\rho}^{(0)}_t(\omega):W^+_s<\hat{c}_n\right\}\right|\\
	&=\left|\left\{0\leq s\leq t+\tilde{\rho}^{(0)}_t(\omega):W^+_s=0\right\}\right|=0.
\end{aligned}\end{equation}

Now, consider a general Feller's Brownian motion $Y$.  According to its pathwise representation \eqref{eq:316} (with lifetime $\zeta$ given in \eqref{eq:326}),  $Y_t=0$ or $Y_{t-}=0$ always implies $W^+_t=0$. Since $Y_t\geq W^+_t$,  it follows that $\tilde{\sigma}^{(n)}\geq \sigma^{(n)}$.  (Recall that the notation with tilde is defined for $W^+$.) Consequently, each discarding interval $[\eta^{(n)}_m,\sigma^{(n)}_m)$ for $Y^{(n)}$ must be contained within some discarding interval for $W^{+,(n)}$ (though the converse is not necessarily true).  {\color{blue}In particular},  $\rho^{(n)}_t\leq \tilde{\rho}^{(n)}_t$.  Therefore,  \eqref{eq:616} follows from \eqref{eq:617}.
\end{proof}
\begin{remark}
If we do not assume that $p_3 = 0$, then the sequence of Doob's Brownian motions $Y^{(n)}$ will converge to a Feller's Brownian motion $\tilde Y$ with parameters $(p_1, p_2, 0, p_4)$ in the sense of \eqref{eq:719}. This is because, according to Lemma~\ref{LM82}, $Y$ can be obtained from $\tilde Y$,  with lifetime $\tilde{\zeta}$, through a time change transformation corresponding to a strictly increasing PCAF. Therefore,  they share the same sequence of distributions $\lambda^{(n)}$,  that is, $\bP_0(Y_{\sigma^{(n)}}\in \cdot)=\bP_0(\tilde Y_{\tilde\sigma^{(n)}}\in \cdot)$,  where $\tilde \sigma^{(n)}:=\inf\{t> 0: \tilde Y_t\geq \hat{c}_n\}\wedge \tilde\zeta$.
\end{remark}

\subsection{Representation of instantaneous distributions}

 The following lemma describes the relationship between the instantaneous distributions.

\begin{lemma}\label{LM63}
For every $m,n\in \bN$ with $m>n$,  the following relationship holds:
\begin{equation}\label{eq:68}
	\lambda^{(n)}=\frac{\frac{\int_{[\hat{c}_{m},\hat{c}_n)}x\lambda^{(m)}(\dd x)}{\hat{c}_n}\delta_{\hat{c}_n}+\lambda^{(m)}|_{[\hat{c}_n,\infty)\cup \{\partial\}}}{1-\int_{[\hat{c}_{m},\hat{c}_n)}\left(1-\frac{x}{\hat{c}_n}\right)\lambda^{(m)}(\dd x)}.
\end{equation}
\end{lemma}
\begin{proof}
Take a positive and bounded Borel measurable function $f$ on $(0,\infty)$.  Note that $Y_{\sigma^{(n)}}=Y_{\sigma^{(n)}}\circ \theta_{\sigma^{(m)}}$,  since $\sigma^{(n)}\geq \sigma^{(m)}$.  Using the strong Markov property of $Y$, we have
\begin{equation}\label{eq:66}
\begin{aligned}
	\lambda^{(n)}(f)=\bE_0 f(Y_{\sigma^{(n)}}\circ \theta_{\sigma^{(m)}})=\bE_0 \left(\bE_{Y_{\sigma^{(m)}}}f(Y_{\sigma^{(n)}})\right).
\end{aligned}\end{equation}
For $y\geq \hat{c}_n$,  it holds $\bP_y$-a.s. that $Y_{\sigma^{(n)}}=Y_0=y$.  Thus,
\begin{equation}\label{eq:67}
\begin{aligned}
\bE_0 \left(\bE_{Y_{\sigma^{(m)}}}f(Y_{\sigma^{(n)}}); Y_{\sigma^{(m)}} \geq \hat{c}_n\right)&=\bE_0\left(f(Y_{\sigma^{(m)}}); Y_{\sigma^{(m)}} \geq \hat{c}_n\right)\\&=\int_{[\hat{c}_n,\infty)}f(x)\lambda^{(m)}(\dd x).
\end{aligned}\end{equation}
For $\hat{c}_{m}\leq y <\hat{c}_n$,  we have
\begin{equation}\label{eq:63}
	\bE_y f(Y_{\sigma^{(n)}})=\bE_y \left(f(Y_{\sigma^{(n)}}); \sigma^{(n)}<\eta\right)+\bE_y \left(f(Y_{\sigma^{(n)}}); \sigma^{(n)}>\eta\right).
\end{equation}
Since $Y_{\sigma^{(n)}}=\hat{c}_n$ on $\{\sigma^{(n)}<\eta\}$, it follows that
\begin{equation}\label{eq:64}
\bE_y \left(f(Y_{\sigma^{(n)}}); \sigma^{(n)}<\eta\right)=f(\hat{c}_n)\cdot \frac{y}{\hat{c}_n}.
\end{equation}
Using $Y_{\sigma^{(n)}}=Y_{\sigma^{(n)}}\circ \theta_\eta$ on $\{\sigma^{(n)}>\eta\}$ and applying the strong Markov property of $Y$,  we get
\begin{equation}\label{eq:65}
\begin{aligned}
	\bE_y \left(f(Y_{\sigma^{(n)}}); \sigma^{(n)}>\eta\right)&=\bE_y \left( \bE_{Y_\eta} f(Y_{\sigma^{(n)}}); \sigma^{(n)}>\eta\right)\\
	&=\bE_0f(Y_{\sigma^{(n)}})\cdot \bE_y(\sigma^{(n)}>\eta) \\
	&=\lambda^{(n)}(f)\cdot \left(1-\frac{y}{\hat{c}_n}\right).
\end{aligned}\end{equation}
Substituting \eqref{eq:64} and \eqref{eq:65} into \eqref{eq:63} yields
\[
\bE_y f(Y_{\sigma^{(n)}})=f(\hat{c}_n)\cdot \frac{y}{\hat{c}_n}+\lambda^{(n)}(f)\cdot \left(1-\frac{y}{\hat{c}_n}\right),\quad \forall \hat{c}_{m}\leq y<\hat{c}_n.
\]
Therefore,
\[
\begin{aligned}
\bE_0& \left(\bE_{Y_{\sigma^{(m)}}}f(Y_{\sigma^{(n)}}); \hat
c_{m}\leq Y_{\sigma^{(m)}} <\hat{c}_n\right)\\&=\bE_0\left(f(\hat{c}_n)\cdot \frac{Y_{\sigma^{(m)}}}{\hat{c}_n}+\lambda^{(n)}(f)\cdot \left(1-\frac{Y_{\sigma^{(m)}}}{\hat{c}_n}\right) ; \hat
c_{m}\leq Y_{\sigma^{(m)}} <\hat{c}_n\right)\\&=\int_{[\hat{c}_{m},\hat c_n)}\left(f(\hat{c}_n)\cdot \frac{y}{\hat{c}_n}+\lambda^{(n)}(f)\cdot \left(1-\frac{y}{\hat{c}_n}\right)\right)\lambda^{(m)}(\dd y).
\end{aligned}\]
Combining this with \eqref{eq:66} and \eqref{eq:67},  we can verify \eqref{eq:68}. 
\end{proof}

Based on this lemma, we can define a measure $\lambda$ on $(0,\infty)\cup \{\partial\}$ as follows:
\begin{equation}\label{eq:623}
	\lambda|_{(\hat c_n,\infty)\cup \{\partial\}}:=\frac{\lambda^{(n)}|_{(\hat c_n,\infty)\cup \{\partial\}}}{\Lambda_n},\quad n\in \bN,
\end{equation}
where $\Lambda_n:= 1-\int_{[\hat{c}_{n},\hat{c}_0)}\left(1-\frac{x}{\hat{c}_0}\right)\lambda^{(n)}(\dd x)$.  {\color{blue}A straightforward computation shows that} $\lambda$ is a well-defined $\sigma$-finite measure on $(0,\infty)\cup \{\partial\}$.  The following lemma provides the representation of $\lambda^{(n)}$ in terms of this measure $ \lambda$.

\begin{lemma}\label{LM64}
For each $n\in \bN$,  the following holds:
\begin{equation}\label{eq:624}
	\frac{1}{\Lambda_n}=\lambda\left((\hat c_n,\infty) \cup \{\partial\}\right)+\frac{\hat c_0\left(1-\lambda\left((\hat c_0,\infty) \cup \{\partial\}\right) \right)-\int_{(\hat c_n,\hat c_0]}x\lambda(\dd x)}{\hat c_n},
\end{equation}
and
\begin{equation}\label{eq:622}
\begin{aligned}
	&\lambda^{(n)}|_{(\hat c_n,\infty)\cup \{\partial\}}=\Lambda_n \cdot \lambda|_{(\hat c_n,\infty)\cup \{\partial\}} \\
	&\lambda^{(n)}(\{\hat c_n\})=\frac{\Lambda_n}{\hat c_n}\left(\hat c_0\left(1-\lambda\left((\hat c_0,\infty) \cup \{\partial\}\right) \right)-\int_{(\hat c_n,\hat c_0]}x\lambda(\dd x) \right).
\end{aligned}
\end{equation}
{\color{blue}In particular},
\begin{equation}\label{eq:621}
	\int_{(0,\infty)}\left(1\wedge x\right) \lambda(\dd x)<\infty.
\end{equation}
\end{lemma}
\begin{proof}
According to Lemma~\ref{LM63}, we have
\[
	\lambda^{(n)}(\{\hat c_n\})=\frac{\Lambda_n}{\Lambda_{n+1}}\cdot \left(\int_{[\hat c_{n+1},\hat c_n)}\frac{x}{\hat c_n}\lambda^{(n+1)}(\dd x)+\lambda^{(n+1)}(\{\hat c_n\})\right).
\]
It follows that
\[
\begin{aligned}
\frac{\hat c_n \lambda^{(n)}(\{\hat c_n\})}{\Lambda_n}&= \frac{\hat c_{n+1} \lambda^{(n+1)}(\{\hat c_{n+1}\})}{\Lambda_{n+1}}+ \frac{\int_{(\hat c_{n+1},\hat c_n)}x \lambda^{(n+1)}(\dd x)+\hat c_n\lambda^{(n+1)}(\{\hat c_n\})}{\Lambda_{n+1}} \\
&= \frac{\hat c_{n+1} \lambda^{(n+1)}(\{\hat c_{n+1}\})}{\Lambda_{n+1}}+ \int_{(\hat c_{n+1},\hat c_n]}x \lambda (\dd x).
\end{aligned}
\]
Let $h_n:= \frac{\hat c_n \lambda^{(n)}(\{\hat c_n\})}{\Lambda_n}$. Then $h_n$ is decreasing in $n$, and
\[
h_0-\lim_{n\rightarrow \infty}h_n=\sum_{n=0}(h_n-h_{n+1})=\sum_{n=0}^\infty \int_{(\hat c_{n+1},\hat c_n]}x\lambda(\dd x)=\int_{(0,\hat c_0]}x \lambda(\dd x).
\]
{\color{blue}In particular}, \eqref{eq:621} holds true. 
The induction also yields that
\[
h_n= h_{n+m}+ \int_{(\hat c_{n+m},\hat c_n]}x \lambda (\dd x),\quad \forall m\geq 1,
\]
and by letting $m\rightarrow \infty$,  we obtain
\begin{equation}\label{eq:627}
	h_n=\lim_{m\rightarrow \infty}h_m+\int_{(0,\hat c_n]}x\lambda(\dd x)=h_0-\int_{(\hat c_n,\hat c_0]}x\lambda(\dd x).
\end{equation}
Therefore,  the {\color{blue}second} identity in \eqref{eq:622} is established. The {\color{blue}first} identity in \eqref{eq:622} follows directly from \eqref{eq:623}. Finally, substituting \eqref{eq:622} into the definition of $\Lambda_n$, we obtain \eqref{eq:624}. 
\end{proof}

\subsection{Parameters of Feller's Brownian motion}

The aim of this subsection is to express the parameters of the original Feller's Brownian motion in terms of the measure $\lambda$, which is defined based on the approximating Doob's Brownian motions.

\begin{theorem}\label{THM65}
Let $Y$ be a Feller's Brownian motion with parameters $(p_1, p_2, 0, p_4)$, and let $\lambda$ be the measure \eqref{eq:623} on $(0,\infty)\cup \{\partial\}$ in terms of the approximating Doob's Brownian motions $Y^{(n)}$. Then, up to a certain multiplicative constant needed to satisfy \eqref{eq:330}, it holds that
\begin{equation}\label{eq:626}
	p_1=\lambda(\{\partial\}),\quad p_2=\hat c_0\left(1-\lambda\left(\{\partial \}\right)\right)-\int_{(0,\infty)}\left(x\wedge \hat c_0\right)\lambda(\dd x),\quad p_4=\lambda|_{(0,\infty)}.
\end{equation}
\end{theorem}
\begin{proof}
Define $\tilde p^n_2:=h_n=\hat c_0 \lambda^{(0)}(\{\hat c_0\})-\int_{(\hat c_n,\hat c_0]}x\lambda (\dd x)$ for $n\in \bN$ (see \eqref{eq:627}). This sequence is decreasing and converges to
\begin{equation}\label{eq:718}
\tilde{p}_2:= \hat c_0\left(1-\lambda\left(\{\partial \}\right)\right)-\int_{(0,\infty)}\left(x\wedge \hat c_0\right)\lambda(\dd x).
\end{equation}
According to \eqref{eq:625} and Lemma~\ref{LM64}, the resolvent of $Y^{(n)}$ can be expressed as
\[
\begin{aligned}
	G^{(n)}_\alpha& f(x)\\&=G^0_\alpha f(x)+\bE_x e^{-\alpha \eta} \cdot \frac{\int_{(\hat c_n, \infty)}G^0_\alpha f(x)\lambda(\dd x)+\tilde p^n_2\cdot \frac{G^0_\alpha f(\hat c_n)}{\hat c_n}}{\frac{1}{\Lambda_n}-\int_{(\hat c_n, \infty)}e^{-\sqrt{2\alpha}x}\lambda (\dd x)-\tilde p^n_2 \cdot \frac{e^{-\sqrt{2\alpha}\hat c_n}}{\hat c_n}}\\
	&= G^0_\alpha f(x)+\bE_x e^{-\alpha \eta} \cdot \frac{\int_{(\hat c_n, \infty)}G^0_\alpha f(x)\lambda(\dd x)+\tilde p^n_2\cdot \frac{G^0_\alpha f(\hat c_n)}{\hat c_n}}{\lambda(\{\partial\})+\int_{(\hat c_n, \infty)}(1-e^{-\sqrt{2\alpha}x})\lambda (\dd x)+\tilde p^n_2 \cdot \frac{1-e^{-\sqrt{2\alpha}\hat c_n}}{\hat c_n}}
\end{aligned}\] 
for all {\color{blue}$f\in C_0([0,\infty))$} and $x\in (0,\infty)$. {\color{blue}Since $f\in C_0([0,\infty))$, a computation yields that}
\[
	\lim_{n\rightarrow \infty} \frac{G^0_\alpha f(\hat c_n)}{\hat c_n}=\lim_{x\rightarrow 0} \frac{G^0_\alpha f(x)}{x}=2\int_{(0,\infty)}e^{-\sqrt{2\alpha}x}f(x)\dd x
\]
{\color{blue}by virtue of $G^0_\alpha f(0)=0$, \eqref{eq:315} and L'H\^opital's rule.}
Additionally,
\[
	\lim_{n\rightarrow \infty} \frac{1-e^{-\sqrt{2\alpha}\hat c_n}}{\hat c_n}=\sqrt{2\alpha}. 
\]
Given \eqref{eq:621},  we have
\[
\begin{aligned}
	\lim_{n\rightarrow \infty}& G^{(n)}_\alpha f(x) \\
	&= G^0_\alpha f(x)+\bE_x e^{-\alpha \eta} \cdot \frac{\int_{(0, \infty)}G^0_\alpha f(x)\lambda(\dd x)+2\tilde p_2\cdot \int_{(0,\infty)}e^{-\sqrt{2\alpha}x}f(x)\dd x}{\lambda(\{\partial\})+\int_{(0, \infty)}(1-e^{-\sqrt{2\alpha}x})\lambda (\dd x)+\tilde p_2 \cdot \sqrt{2\alpha}}.
\end{aligned}\]
By Theorem~\ref{THM62}, this limit is exactly the resolvent of $Y$. Comparing this expression with \eqref{eq:39} {\color{blue}and noting the uniqueness result established in Proposition~\ref{PROA2}}, we can eventually derive \eqref{eq:626}.
\end{proof}

\section{Proof of Theorem~\ref{MainTHM} for $|p_4|=\infty$}


{\color{blue}
In this section, we prove Theorem~\ref{MainTHM} in the case where $|p_4|=\infty$. To this end, we begin by recalling Wang’s approximation of birth-death processes, which serves as a foundational prototype for the results established in the preceding section.

\subsection{Wang's approximation of birth-death process}

Given a Feller $Q$-process $X$,
we aim to construct  a sequence of Doob processes $X^{(n)}$  following the same procedure outlined in Lemma~\ref{LM61}.  
Let us consider the corresponding $Q$-process $\hat{X}_t=\Xi^{-1}(X_t)$ on $\overline{E}$,  with lifetime denoted by $\hat{\zeta}$.  Define $$\hat{\eta}:=\inf\{t>0: \hat{X}_{t-}=0\}$$ and $$\hat\sigma^{(n)}:=\inf\left\{t>0: \hat{X}_t\in \{\hat c_0,\cdots,\hat{c}_n\}\right\}\wedge \hat\zeta,\quad n\in \bN.$$ 
Analogous sequences of stopping times $\{\hat{\eta}^{(n)}_m: m\geq 1\}$ and $\{\hat{\sigma}^{(n)}_m:m\geq 1\}$ can be defined by repeating the procedures used before Lemma~\ref{LM71} (see also \cite[\S6.3]{WY92}).  Applying the $C(\hat\eta^{(n)}_m,\hat\sigma^{(n)}_m)$-transformation to $\hat{X}$ yields a sequence of Doob processes $\hat{X}^{(n)}$ on $E$,  with instantaneous distribution 
$\hat{\lambda}^{(n)}(\cdot)=\bP_0(\hat{X}_{\hat{\sigma}^{(n)}}\in \cdot)$.

The main result of {\color{blue}\cite[Chapter 6]{WY92}} establishes the convergence of $\hat X^{(n)}$ to $\hat X$ in the following sense. 

\begin{theorem}\label{THMA4}
Let $X$ be a Feller $Q$-process 
and let $X^{(n)}=\Xi(\hat{X}^{(n)})$ denote the approximating sequence of Doob processes. 
Then, for any $i\in \bN$,
\begin{equation}\label{eq:612}
	\bP_i\left(\lim_{n\rightarrow \infty}X^{(n)}_t=X_t,\forall t\geq 0\right)=1.
\end{equation}
\end{theorem}

The original proof in \cite{WY92} is quite lengthy and only addresses the honest case. We will provide an alternative proof using the right continuity of $X$ in Appendix~\ref{APPA}. 

\subsection{No sojourn case $p_3=0$}\label{SEC81}

 In this subsection, we consider the case where $|p_4|=\infty$ and further assume  that $p_3=0$. According to Theorem~\ref{THM62}, we can construct a sequence of Doob's Brownian motions $Y^{(n)}$ that approximates the Feller's Brownian motion $Y$. Their instantaneous distributions $\lambda^{(n)}$ are characterized in Lemma~\ref{LM64} in terms of the measure $\lambda$ defined by \eqref{eq:623}. 
Regarding the time-changed Feller's Brownian motion $\hat X:=\check Y$, we can also construct a sequence of Doob processes $\hat X^{(n)}$ with instantaneous distribution 
$\hat{\lambda}^{(n)}$, as described in Theorem~\ref{THMA4},  that approximates $\hat X$.}

\begin{lemma}\label{LM81}
For all $n\in \bN$, the following holds:
\[
\begin{aligned}
	&\hat\lambda^{(n)}(\{\hat c_0\})=\Lambda_n \cdot\left(\int_{(\hat c_1, \hat c_0]}\frac{x-\hat c_1}{\hat c_0-\hat c_1}\lambda(\dd x) +\lambda\left((\hat c_0,\infty) \right) \right),\\
	&\hat \lambda^{(n)}(\{\hat c_i\}) =\Lambda_n \cdot\left(\int_{(\hat c_{i+1}, \hat c_i]}\frac{x-\hat c_{i+1}}{\hat c_i-\hat c_{i+1}}\lambda(\dd x) + \int_{(\hat c_{i}, \hat c_{i-1})}\frac{\hat c_{i-1}-x}{\hat c_{i-1}-\hat c_{i}}\lambda(\dd x) \right)
\end{aligned}
\]
for $1\leq i\leq n-1$, 
and 
\[
\hat \lambda^{(n)}(\{\hat c_n\}) =\Lambda_n \cdot\left(\frac{\tilde p^n_2}{\hat c_n}+\int_{(\hat c_n, \hat c_{n-1})}\frac{\hat c_{n-1}-x}{\hat c_{n-1}-\hat c_n}\lambda(\dd x)\right),
\]
where $\tilde p^n_2=\hat c_0 \lambda^{(0)}(\{\hat c_0\})-\int_{(\hat c_n,\hat c_0]}x\lambda (\dd x)$.
\end{lemma}
\begin{proof}
{\color{blue}We consider the case $1 \leq i \leq n-1$. The other two cases can be addressed in a similar manner.}
Note that
	\[
		\{\hat X_{\hat \sigma^{(n)}}=\hat c_i\}=\{Y_{\sigma^{(n)}}\in (\hat c_{i+1}, \hat c_{i-1}), \tau_{\hat c_i}\circ \theta_{\sigma^{(n)}}<\left(\tau_{\hat c_{i+1}} \wedge \tau_{\hat c_{i-1}}\right)\circ \theta_{\sigma^{(n)}} \},
	\]
	where $\tau_a:=\inf\{t>0: Y_t=a\}$ for $a\in (0,\infty)$. It follows from the strong Markov property of $Y$ that
	\[
	\begin{aligned}
		\hat \lambda^{(n)}(\{\hat c_i\})&=\bP_0\left(\bP_{Y_{\sigma^{(n)}}}(\tau_{\hat c_i}<\tau_{\hat c_{i+1}} \wedge \tau_{\hat c_{i-1}}); Y_{\sigma^{(n)}}\in (\hat c_{i+1}, \hat c_{i-1}) \right)\\
		&= \int_{(\hat c_{i+1}, \hat c_i]}\frac{x-\hat c_{i+1}}{\hat c_i-\hat c_{i+1}}\lambda^{(n)}(\dd x) + \int_{(\hat c_{i}, \hat c_{i-1})}\frac{\hat c_{i-1}-x}{\hat c_{i-1}-\hat c_{i}}\lambda^{(n)}(\dd x).
	\end{aligned}
	\]
	Then applying Lemma~\ref{LM64}, we can obtain the desired expression of $\lambda^{(n)}(\{\hat c_i\}) $. 
\end{proof}

We are now prepared to prove Theorem~\ref{MainTHM} for the case where $p_3=0$ and $|p_4|=\infty$.

\begin{proof}[Proof of Theorem~\ref{MainTHM} for $p_3=0,|p_4|=\infty$]
Let
\[
	\tilde{\mathfrak{p}}_0:=\int_{(\hat{c}_{1},\hat{c}_0]}\frac{x-\hat{c}_{1}}{\hat{c}_0-\hat{c}_{1}}\lambda(\dd x)+\lambda\left((\hat{c}_0,\infty)\right)
\]
and
\[
	\tilde{\mathfrak{p}}_n:=\int_{(\hat{c}_{n+1},\hat{c}_n]}\frac{x-\hat{c}_{n+1}}{\hat{c}_n-\hat{c}_{n+1}}\lambda(\dd x)+\int_{(\hat{c}_{n},\hat{c}_{n-1})}\frac{\hat{c}_{n-1}-x}{\hat{c}_{n-1}-\hat{c}_{n}}\lambda(\dd x),\quad n\geq 1.
\]
	Substituting the expression of $\hat \lambda^{(n)}$ from Lemma~\ref{LM81} into \eqref{eq:B2},  we can formulate the resolvent matrix $\hat{\Psi}^{(n)}_{ij}(\alpha):=\bE_{\hat{c}_i}\int_0^\infty e^{-\alpha t}1_{\{\hat c_j\}}(\hat{X}^{(n)}_t)\dd t$ of $\hat X^{(n)}$ for $\alpha>0$ and $i,j\in \bN$ as follows:
\[
\begin{aligned}
&\hat{\Psi}^{(n)}_{ij}(\alpha)-\Phi_{ij}(\alpha)\\
=&u_\alpha(i) \frac{\sum_{k=0}^{n-1} \tilde{\mathfrak p}_k \Phi_{kj}(\alpha)+\left(\tilde p^n_2+\hat c_n\int_{(\hat c_n,\hat c_{n-1})}\frac{\hat c_{n-1}-x}{\hat c_{n-1}-\hat c_n}\lambda(\dd x)\right)\frac{\Phi_{nj}(\alpha)}{\hat{c}_n}}{\frac{1}{\Lambda_n}-\sum_{k=0}^{n-1}\tilde{\mathfrak p}_k u_\alpha(k)-\left(\tilde p^n_2+\hat c_n\int_{(\hat c_n,\hat c_{n-1})}\frac{\hat c_{n-1}-x}{\hat c_{n-1}-\hat c_n}\lambda(\dd x)\right)\frac{u_\alpha(n)}{\hat{c}_n}} \\
=&u_\alpha(i) \frac{\sum_{k=0}^{n-1} \tilde{\mathfrak p}_k \Phi_{kj}(\alpha)+\left(\tilde p^n_2+\hat c_n\int_{(\hat c_n,\hat c_{n-1})}\frac{\hat c_{n-1}-x}{\hat c_{n-1}-\hat c_n}\lambda(\dd x)\right)\frac{\Phi_{nj}(\alpha)}{\hat{c}_n}}{\lambda(\{\partial\})+\sum_{k=0}^{n-1}\tilde{\mathfrak p}_k (1- u_\alpha(k))+ \left(\tilde p^n_2+\hat c_n\int_{(\hat c_n,\hat c_{n-1})}\frac{\hat c_{n-1}-x}{\hat c_{n-1}-\hat c_n}\lambda(\dd x)\right)\frac{1-u_\alpha(n)}{\hat{c}_n}}.
\end{aligned}\]
Note that
\[
\hat c_n\int_{(\hat c_n,\hat c_{n-1})}\frac{\hat c_{n-1}-x}{\hat c_{n-1}-\hat c_n}\lambda(\dd x)\leq \int_{(\hat c_n,\hat c_{n-1})}x\lambda(\dd x)\rightarrow 0
\]
as $n\rightarrow \infty$, due to \eqref{eq:621}. In addition, 
(see,  e.g., \cite[\S7.10, (3) and (9)]{WY92})
\[
	\lim_{n\rightarrow \infty}\frac{\Phi_{nj}(\alpha)}{\hat{c}_n}=2u_\alpha(j)\mu_j,\quad 	\lim_{n\rightarrow \infty}\frac{1- u_\alpha(n)}{\hat{c}_n}=2\alpha \sum_{k\in \bN}\mu_ku_\alpha(k).
\]
Therefore,
\[
\begin{aligned}
\lim_{n\rightarrow \infty}\hat{\Psi}^{(n)}_{ij}&(\alpha)=\Phi_{ij}(\alpha)\\
&+u_\alpha(i) \frac{\sum_{k=0}^{\infty} \tilde{\mathfrak p}_k \Phi_{kj}(\alpha)+2\tilde p_2u_\alpha(j)\mu_j}{\lambda(\{\partial\})+\sum_{k=0}^{\infty}\tilde{\mathfrak p}_k (1- u_\alpha(k))+2\alpha \tilde p_2 \sum_{k\in \bN}\mu_ku_\alpha(k)},
\end{aligned}\]
where $\tilde{p}_2$ is defined as \eqref{eq:718}.
According to Theorem~\ref{THMA4}, this limit is precisely the resolvent matrix of $\hat X$. Hence, the parameters of $\hat X$ are (up to a multiplicative constant):
\[
	\gamma=\lambda(\{\partial\}),\quad \beta=2\tilde p_2,\quad \nu_k=\tilde{\mathfrak p}_k,\; k\in \bN.
\]
Using Theorem~\ref{THM65}, we can eventually obtain the desired conclusion.
\end{proof}

\subsection{Sojourn case $p_3>0$}\label{SEC82}

Let us consider the final case where $p_3>0$ and $|p_4|=\infty$.  The basic idea is to transform this case into one with $p_3=0$.  
Recall that $Y^1$, defined by \eqref{eq:316}, is a Feller's Brownian motion with parameters $(0,p_2,0,p_4)$ and $\ell^{Y^1}$, defined by \eqref{eq:322}, is its local time at $0$.  According to \S\ref{SEC323},  the subprocess $Y^3$ of $Y^1$,  perturbed by the multiplicative functional
\[
	M^3_t:=e^{-p_1\ell^{Y^1}_t},\quad t\geq 0,
\]
is a Feller's Brownian motion with parameters $(p_1,p_2,0,p_4)$.  Denote by $\zeta_3$ the lifetime of $Y^3$.  Mimicking the proof of Lemma~\ref{LM36},  we can easily show that
\[
	\mathfrak{f}_3(t):=t\wedge \zeta_3+p_3\ell^{Y^1}_{t\wedge \zeta_3},\quad t\geq 0
\]
is a PCAF of $Y^3$ with $\text{Supp}(\mathfrak{f}_3)=[0,\infty)$.  

\begin{lemma}\label{LM82}
The time-changed process of $Y^3$ with respect to the PCAF $\mathfrak{f}_3$ is a Feller's Brownian motion with parameters $(p_1,p_2,p_3,p_4)$.  
\end{lemma}
\begin{proof}
Given $Y^1$ expressed as $Y^1=(\Omega^1, \sG^1,  Y^1_t,  (\bP^1_x)_{x\in [0,\infty)})$,  according to \cite[III, Theorem~3.3]{BG68},  we can write $Y^3=(\Omega^3,\sG^3, Y^3_t, \bP^3_x)$ as follows:
\begin{equation}\label{eq:81}
\begin{aligned}
	&\Omega^3=\Omega^1\times [0,\infty],\quad \sG^3:=\sG^1\times \mathcal{B}([0,\infty]),  \\
	&Y^3_t(\omega, \kappa):=Y^1_t(\omega)\text{ for }t<\kappa,\text{ and }Y^3_t(\omega, \kappa):=\partial\text{ for }t\geq \kappa,
\end{aligned}\end{equation}
and for $\Gamma\in \sG^3$ with $\Gamma^\omega:=\{\kappa \in [0,\infty]: (\omega,\kappa)\in \Gamma\}$ for $\omega\in \Omega^1$, 
\[
	\bP^3_x(\Gamma):=\bE^1_x\left(\alpha_\omega(\Gamma^\omega)\right),
\]
where $\alpha_\omega$ is a probability measure on $[0,\infty]$ such that $\alpha_\omega((t,\infty])=M^3_t(\omega)$ for all $t\geq 0$.  Note that $\zeta_3(\omega,\kappa)=\kappa$.  

Let $\mathfrak{f}^{-1}_3(t)$ denote the right-continuous inverse of $\mathfrak{f}_3(t)$.  This inverse is continuous and strictly increasing up to $\zeta_3$.  The time-changed process $Y^4=(\Omega^4, \sG^4,Y^4_t, \bP^4_x)$ of $Y^3$ with respect to $\mathfrak{f}_3$ can be expressed as 
\[
	\Omega^4=\Omega^3,\quad \sG^4=\sG^3,\quad \bP^4_x=\bP^3_x,
\]
and
\[
	Y^4_t:=Y^3_{\mathfrak{f}^{-1}_3(t)}\text{ for }t<\zeta_4:=\mathfrak{f}_3(\zeta_3),\text{ and }Y^4_t:=\partial\text{ for } t\geq \zeta_4,
\]
where $\zeta_4$ is the lifetime of $Y^4$.  By substituting the expression \eqref{eq:81} of $Y^3$ into this expression of $Y^4$,  we find that for $(\omega,\kappa)\in \Omega^1\times [0,\infty]$,  
\[
	Y^4_t(\omega,\kappa)=Y^1_{\mathfrak{f}^{-1}(t)}(\omega),\quad t<\zeta_4(\omega, \kappa)=\kappa+p_3\ell^{Y^1}_\kappa(\omega)
\]
where $\mathfrak{f}^{-1}$ is the right-continuous inverse of $\mathfrak{f}(t)=t+p_3\ell^{Y^1}_t(\omega)$ ($t\geq 0$).  In other words,  $Y^4$ is the killed process of $Y^2$, defined in \eqref{eq:321}, at time $\zeta_4$.  Note that for all $t\geq 0$,
\[
	\alpha_\omega\left(\{\kappa\in [0,\infty]: \zeta_4(\omega,\kappa)\geq t\}\right))=\alpha_\omega\left(\{\kappa\in [0,\infty]: \kappa \geq \mathfrak{f}^{-1}(t)\} \right)=e^{-p_1 \ell^{Y^1}_{\mathfrak{f}^{-1}(t)}(\omega)}.
\]
Thus,  $Y^4$ is the subprocess of $Y^2$ perturbed by the multiplicative functional in \eqref{eq:327}.  According to the pathwise representation in \S\ref{SEC323},  we conclude that $Y^4$ is a Feller's Brownian motion with parameters $(p_1,p_2,p_3,p_4)$.  
\end{proof}

Now, we complete the proof of Theorem~\ref{MainTHM} as follows.

\begin{proof}[Proof of Theorem~\ref{MainTHM} for $p_3>0,|p_4|=\infty$]
Let $Y^3=(\Omega^3,Y^3_t,\bP^3_x)$, with lifetime $\zeta_3$, be a Feller's Brownian motion with parameters $(p_1,p_2,0,p_4)$.  According to Lemma~\ref{LM82},  the Feller's Brownian motion $Y=(\Omega, Y_t,\bP_x)$ with parameters $(p_1,p_2,p_3,p_4)$ can be represented as the time-changed process of $Y^3$ with respect to $\mathfrak{f}_3$.  The lifetime of $Y$ is $\zeta=\mathfrak{f}_3(\zeta_3)$.

Denote by $\check{Y}^3$ and $\check{Y}$ the time-changed process  obtained in Theorem~\ref{THM42} for $Y^3$ and $Y$, respectively.  The parameters determining the resolvent matrix of $\check{Y}^3$ have been examined in \S\ref{SEC81}.  It suffices to show that $\check{Y}$ is identical in law to $\check{Y}^3$.  Let $L^{3,\hat{c}_n}_t$ denote the local time of $Y^3$ at $\hat{c}_n$ as defined in Definition~\ref{DEF31},  and set $A^3_t:=\sum_{n\in \bN}\mu_n L^{3,\hat{c}_n}_t$ for all $t\geq 0$.  It is straightforward to verify that
\[
	L^{\hat{c}_n}_t:=L^{3,\hat{c}_n}_{\mathfrak{f}^{-1}_3(t)},\quad t\geq 0
\]
is the local time of $Y$ at $\hat{c}_n$ in the sense of Definition~\ref{DEF31}.  Thus, the PCAF of $Y$ given by \eqref{eq:41} is $A_t=A^3_{\mathfrak{f}^{-1}_3(t)}$,  and the right-continuous inverse of $A$ is 
\[
	\gamma_t=\inf\{s>0:A_s>t\}=\inf\{s>0:A^3_{\mathfrak{f}^{-1}_3(s)}>t\},\quad \forall t<A_{\zeta}=A^3_{\zeta_3}.
\]
Since $\mathfrak{f}_3$ is strictly increasing up to $\zeta_3$,  it follows that
\[
	\mathfrak{f}^{-1}_3(\gamma_t)=\inf\{s>0:A^3_s>t\}=:\gamma^3_t,\quad \forall t<A_{\zeta}=A^3_{\zeta_3},
\]
where $\gamma^3_t$ is the right-continuous inverse of $A^3$.  {\color{blue}In particular},  for $t<A_{\zeta}=A^3_{\zeta_3}$,
\begin{equation}\label{eq:82}
	\check{Y}_t=Y_{\gamma_t}=Y^3_{\mathfrak{f}^{-1}_3(\gamma_t)}=Y^3_{\gamma^3_t}=\check{Y}^3_t.
\end{equation}
Note that $A_\zeta$ is the lifetime of $\check{Y}$ and $A^3_{\zeta_3}$ is the lifetime of $\check{Y}^3$.  Therefore,  \eqref{eq:82} establishes the identification between $\check{Y}$ and $\check{Y}^3$.
\end{proof}

\appendix

{\color{blue}
\section{Proof of Theorem~\ref{THM32}}\label{APPB}


In this appendix, we present the necessary details for the proof of Theorem~\ref{THM32}. Our approach follows an argument developed for general one-dimensional diffusions in \cite{M68}. To this end, we begin by introducing some notation.

Let $[0,\infty]$ denote the one-point compactification of $[0,\infty)$, and let $C([0,\infty])$ denote the space of all continuous functions on $[0,\infty]$.  Note that $C_0([0,\infty))=\{\bar{f}|_{[0,\infty)}: \bar{f}(\infty)=0\}$, and that $C([0,\infty])$ is the linear span of $C_0([0,\infty))$ and the constant functions.  For each function $\bar f\in C([0,\infty])$,  define its projection onto $C_0([0,\infty))$ by
\[
	f:= \left(\bar{f}-\bar{f}(\infty)\cdot 1_{[0,\infty]}\right)\big|_{[0,\infty)}\in C_0([0,\infty)).
\]
Let $(T_t)_{t\geq 0}$ be the semigroup on $C_0([0,\infty))$ associated with a Feller process $Y$ on $[0,\infty)$,  whose infinitesimal generator is denoted by $(\sL,\cD(\sL))$.  For each $\bar{f} \in C([0,\infty])$, with corresponding projection $f \in C_0([0,\infty))$, define
\[
	\bar{T}_t \bar{f}=T_t f+\bar{f}(\infty)\cdot 1_{[0,\infty]},\quad \forall t\geq 0.
\]
It is clear that $(\bar{T}_t)_{t\geq 0}$ forms a strongly continuous contraction semigroup on $C([0,\infty])$,  whose infinitesimal generator is given by
\begin{equation}\label{eq:C1}
\begin{aligned}
	&\cD(\bar{\sL})=\{\bar{f}\in C([0,\infty]): f\in \cD(\sL)\},\\
	 &\bar{\sL}\bar{f}=\sL f,\quad \bar{f}\in \cD(\bar{\sL}).
\end{aligned}
\end{equation}
Note that $(\bar{T}_t)_{t\geq 0}$ can still be interpreted as the semigroup of the process $Y$ by treating the point $\infty$ as an absorbing state, rather than as the cemetery. (When $\infty$ serves as the cemetery, it is conventional to require that every function $\bar f$ satisfies $\bar f(\infty) = 0$.)
Let $(\bar{G}_\alpha)_{\alpha>0}$ denote the resolvent of $(\bar{T}_t)_{t\geq 0}$. 

\subsection{Proof of necessity and sufficiency}

In this section, we prove the main part of Theorem~\ref{THM32}, excluding uniqueness.  For convenience,  we define
\[
	\sG:=\{f\in C_0([0,\infty))\cap C^2((0,\infty)): f''\in C_0([0,\infty))\}
\]
and
\[
	\bar{\sG}:=\{\bar{f}\in C([0,\infty])\cap C^2((0,\infty)): \bar f''\in C([0,\infty]).  
\]
Note that $\sG=\{f: \bar{f}\in \sG\}$.  The only point requiring clarification is that any $\bar{f}\in \bar{\sG}$ satisfies $\lim_{x\rightarrow \infty}\bar{f}''(x)=0$, which follows from \cite[page 22, Lemma 2]{M68}.  Given a quadruple $(p_1,p_2,p_3,p_4)$ satisfying \eqref{eq:330} and \eqref{eq:38},  define
\[
	\Phi_p(f):=p_1f(0) - p_2f'(0) + \frac{p_3}{2}f''(0) + \int_{(0,\infty)} \left(f(0) - f(x)\right) p_4(\dd x),\quad f\in \sG
\]
and 
\[
\bar\Phi_p(\bar f):=p_1 (\bar f(0)-\bar{f}(\infty)) - p_2\bar f'(0) + \frac{p_3}{2}\bar f''(0) + \int_{(0,\infty)} \left(\bar f(0) - \bar f(x)\right) p_4(\dd x),\quad \bar f\in \bar \sG.  
\]
In addition,  define the subspaces $\sG_p:=\{f\in \sG:\Phi_p(f)=0\}$ and $\bar{\sG}_p:=\{f\in \bar{\sG}: \bar \Phi_p(\bar{f})=0\}$,  so that $\sG_p=\{f: \bar{f}\in \bar{\sG}_p\}$.  

\begin{proof}[Proof of necessity and sufficiency]

We begin by establishing the necessity.  Let $Y$ be a Feller's Brownian motion on $[0,\infty)$ with infinitesimal generator $(\sL,\cD(\sL))$,  and let $(\bar{\sL},\cD(\bar{\sL}))$ be defined as \eqref{eq:C1}.  We claim that
\[
	\cD(\bar{\sL})\subset \bar{\sG},\quad \bar{\sL}\bar f=\frac{1}{2}\bar{f}'',\quad \forall \bar{f}\in \cD(\bar{\sL}).
\]
This is equivalent to stating that $\cD(\sL)\subset \sG$ and $\sL f=\frac{1}{2}f''$ for all $f\in \cD(\sL)$,  which follows directly from \eqref{eq:32}, \eqref{eq:33}, and \eqref{eq:31}.  Since $0$ is a regular boundary,  we may apply \cite[page 39, Theorem~2]{M68} to $\bar{\sL}$.  This yields a collection of parameters $\kappa_0,  s_0,  q_0,  \sigma_0$,  where $\kappa_0,\sigma_0\geq 0$ are constants,  $s_0\in C([0,\infty])$ is non-decreasing and satisfies $s_0(x)=2x$ in a neighbourhood of $0$,  and $q_0$ is a finite positive measure on $[0,\infty]$.  (To avoid ambiguity, we have replaced $p_0$ in \cite[page 39, (40)]{M68} with $s_0$). These parameters satisfy the following condition:
\begin{equation}\label{eq:C2}
\text{either}\quad \sigma_0>0\quad \text{or} \quad \int_{[0,\infty]}\frac{q_0(\dd x)}{s_0(x)}=\infty, 
\end{equation}
and for any $\bar f\in \cD(\bar{\sL})$,  it holds that
\begin{equation}\label{eq:C4}
	\kappa_0 \bar{f}(0)-\frac{q_0(\{0\})}{2}\bar{f}'(0)+\int_{(0,\infty]}(\bar{f}(0)-\bar{f}(x)) \frac{q_0(\dd x)}{s_0(x)}+\frac{\sigma_0}{2}\bar{f}''(0)=0. 
\end{equation}
As $\bar{T}_t1_{[0,\infty]}=1_{[0,\infty]}$,  the parameter $r(t)$ in \cite[page 39, (43)]{M68} equals $0$.  Consequently,  $\kappa_0=0$.  Moreover,  if $q_0(\{0\})>0$,  then $\int_{[0,\infty]}\frac{q_0(\dd x)}{s_0(x)}=\infty$.  Noting that $s_0(\infty)<\infty$,  the condition \eqref{eq:C2} is therefore equivalent to 
\begin{equation}\label{eq:C3}
\int_{(0,\infty)}\frac{q_0(\dd x)}{s_0(x)}=\infty,\quad \text{if } \sigma_0=0, q_0(\{0\})=0.
\end{equation}
We now define
\[
	\hat p_1:=\frac{q_0(\{\infty\})}{s_0(\infty)},\quad \hat p_2:=\frac{q_0(\{0\})}{2},\quad \hat p_3:=\sigma_0,\quad \hat p_4(\dd x):=\frac{q_0(\dd x)}{s_0(x)}\bigg|_{(0,\infty)}.
\]
Since $q_0$ is a finite measure and $s_0(x)=2x$ in a neighbourhood of $0$,  we have $\int_{(0,\infty)}(1\wedge x) p_4(\dd x)<\infty$.  Therefore,  there exists a positive constant $c$ such that 
\[
	p_1:=c\cdot \hat p_1,\quad p_2:=c\cdot \hat p_2,\quad p_3:=c\cdot \hat p_3,\quad p_4:=c\cdot \hat p_4
\]
satisfy the conditions \eqref{eq:330} and \eqref{eq:38}.  Moreover,  by \eqref{eq:C4},  we have $\bar \Phi_p(\bar{f})=0$ for all $\bar f \in \cD(\bar{\sL})$,  so that
\[
	\cD(\bar{\sL})\subset \bar{\sG}_p,\quad \bar{\sL}\bar f=\frac{1}{2}\Delta \bar f:=\frac{1}{2}\bar{f}''.
\]
By \cite[page 41, Theorem~3]{M68},  the operator $\frac{1}{2}\Delta$ with domain $\bar{\sG}_p$ is the infinitesimal operator of a Feller semigroup on $C([0,\infty])$.  Let $(R_\alpha)_{\alpha>0}$ denote its resolvent.  Then,  by the Hille-Yosida theorem (see,  e.g.,  \cite[page 2, Theorem~1]{M68}),  for any $\bar{h}\in C([0,\infty])$,  both $\bar{G}_\alpha \bar{h}$ and $R_\alpha h$ solves the equation
\[
	\alpha \bar{F}-\frac{1}{2}\bar{F}''=\bar h,
\]
subject to the condition $\bar \Phi_p(\bar F)=0$.  According to \cite[page 41, Lemma~4]{M68},  these solutions coincide,  i.e., $\bar{G}_\alpha \bar{h}=R_\alpha \bar h$ for any $\bar h\in C([0,\infty])$.  Therefore,  $\cD(\bar{\sL})=\bar{\sG}_p$,  which immediately implies $\cD(\sL)=\sG_p$.  This completes the proof of necessity.  

For the sufficiency,  we first apply \cite[II\S5, Theorem~3]{M68} to conclude that the operator $\frac{1}{2}\Delta$ with domain $\bar{\sG}_p$ is indeed the infinitesimal generator of a Feller semigroup on $C([0,\infty])$.  Its projection onto $C_0([0,\infty))$, denoted by $(\sL, \cD(\sL))$ with $\cD(\sL)=\sG_p$, then serves as the infinitesimal generator of a Feller process $Y$ on $[0,\infty)$, where $\infty$ is interpreted as the cemetery state.  It remains to demonstrate that the killed process 
\begin{equation}\label{eq:317}
	Y^0_t:=\left\lbrace
	\begin{aligned}
	&Y_t,\quad &t<\tau_0, \\
	&\partial\;(=\infty),\quad &t\geq \tau_0,
	\end{aligned}
	\right.
\end{equation}
where $\tau_0=\inf\{t>0:Y_t=0\}$,  
is identical in law to the {\color{blue}killed} Brownian motion on $(0,\infty)$.  In fact,  the proof of \cite[II\S5, Theorem~3]{M68} shows that the resolvent kernel of $Y^0=(Y^0_t)_{t\geq 0}$ is given by \cite[\S3, \#7]{M68} (with $p(x)=2x$ and $m$ being the Lebesgue measure).  This is precisely the resolvent kernel of {\color{blue}killed} Brownian motion.
\end{proof}

\subsection{Proof of uniqueness}

In this section, we prove that the parameters $(p_1, p_2, p_3, p_4)$, which characterizes the boundary behavior of a given Feller's Brownian motion $Y$, are uniquely determined -- an observation that, to the best of our knowledge, has not been explicitly stated in the existing literature including \cite{IM63} and \cite{M68}.

Let $(\tilde{p}_1,\tilde{p}_2,\tilde{p}_3,\tilde{p}_4)$ be another quadruple satisfying \eqref{eq:330} and \eqref{eq:38}, which also characterize the boundary behaviour of $Y$. As explained following the statement of Theorem~\ref{THM32}, the fact that $f=G_\alpha h\in \cD(\sL)$ with $h\in C_0([0,\infty))$ satisfies the boundary condition \eqref{eq:12} is equivalent to the representation of $G_\alpha h(0)$ given in \eqref{eq:39}.  Applying this observation to  $(\tilde{p}_1,\tilde{p}_2,\tilde{p}_3,\tilde{p}_4)$, we obtain
\begin{equation}\label{eq:A3}
\begin{aligned}
		&\frac{2p_2\int_{(0,\infty)}e^{-\sqrt{2\alpha}x}h(x)\dd x+p_3h(0)+\int_{(0,\infty)}G^0_\alpha h(x)p_4(\dd x)}{p_1+\sqrt{2\alpha}p_2+\alpha p_3+\int_{(0,\infty)}\left(1-e^{-\sqrt{2\alpha}x}\right)p_4(\dd x)} \\
		&\qquad\qquad  =\frac{2\tilde p_2\int_{(0,\infty)}e^{-\sqrt{2\alpha}x}h(x)\dd x+\tilde p_3h(0)+\int_{(0,\infty)}G^0_\alpha h(x)\tilde p_4(\dd x)}{\tilde p_1+\sqrt{2\alpha}\tilde p_2+\alpha \tilde p_3+\int_{(0,\infty)}\left(1-e^{-\sqrt{2\alpha}x}\right)\tilde p_4(\dd x)}
\end{aligned}\end{equation}
for all $h\in C_0([0,\infty))$ and $\alpha>0$. We will prove that $(\tilde{p}_1,\tilde{p}_2,\tilde{p}_3,\tilde{p}_4)=(p_1, p_2, p_3, p_4)$ based on this equality.  Before doing so, we collect some elementary facts that will be used in the argument.

\begin{lemma}\label{LMA1}
The following statements hold:
\begin{itemize}
\item[(1)] Given $h\in C_c((0,\infty))$ with $\text{supp}[h]\subset [l,r]\subset (0,\infty)$, there exist constants $\alpha_0>0$ and $C>0$ such that
\[
	\sup_{\alpha \geq \alpha_0} \left|\alpha G^0_\alpha h(x)\right|\leq C\left(x\wedge \frac{l}{2}\right),\quad \forall x>0. 
\]
\item[(2)] Given a positive function $h\in C_c([0,\infty))$,  define $G^0h(x):=\lim_{\alpha\downarrow 0}G^0_\alpha h(x)$. Then there exists a constant $C>0$ such that 
\begin{equation}\label{eq:B42}
	G^0 h(x)\leq C(x\wedge 1),\quad \forall x>0. 
\end{equation}
\item[(3)] For any $x>0$, the following inequalities hold:
\[
	\sup_{\alpha\geq 2}\frac{1-e^{-\sqrt{2\alpha}x}}{\alpha}\leq 1\wedge x,\quad \sup_{\alpha\geq 1}\frac{1-e^{-\sqrt{2\alpha}x}}{\sqrt{\alpha}}\leq 1\wedge x.
\]
\end{itemize}
\end{lemma}
\begin{proof}
(1) Define $\|h\|_u:=\sup_{x\in (0,\infty)}|h(x)|$. Clearly, $|\alpha G^0_\alpha h(x)|\leq \|h\|_u$ for all $\alpha>0$ and $x\in (0,\infty)$. Therefore, to establish the estimate, it suffices to find constants $\alpha_0>0$ and $C>0$ such that
\[
	\sup_{\alpha \geq \alpha_0} \left|\alpha G^0_\alpha h(x)\right|\leq C x,\quad x\leq l/2. 
\]
Recall that $\alpha G^0_\alpha h(x)=\alpha \int_0^\infty g^0_\alpha (x,y)h(y)\dd y$, where $g^0_\alpha$ is given by \eqref{eq:310}.  For $x\leq l/2$,  a direct computation yields
\[
\begin{aligned}
	\alpha G^0_\alpha h(x)&=\frac{\sqrt{2\alpha}}{2}(1-e^{-2\sqrt{2\alpha}x})e^{-\frac{l}{2}\sqrt{2\alpha}}\int_l^r e^{-\sqrt{2\alpha}(y-x-\frac{l}{2})}h(y)\dd y \\
	&\leq \frac{\|h\|_u(r-l)}{2}\sqrt{2\alpha}\cdot (1-e^{-2\sqrt{2\alpha}x})\cdot e^{-\frac{l}{2}\sqrt{2\alpha}}. 
\end{aligned}\]
Now,  choose $\alpha_0>0$ such that $e^{\frac{l}{4}\sqrt{2\alpha}}\geq 2\sqrt{2\alpha}$ for all $\alpha \geq \alpha_0$. Then,  we have
\[
	1-e^{-2\sqrt{2\alpha}x}\leq e^{\frac{l}{4}\sqrt{2\alpha}}\cdot x,\quad \forall \alpha\geq \alpha_0,x>0. 
\]
It follows that
\[
	\sup_{\alpha \geq \alpha_0} \left|\alpha G^0_\alpha h(x)\right|\leq \frac{\|h\|_u(r-l)}{4}x,\quad x\leq l/2,
\]
which completes the proof of the desired estimate.

(2) To prove the inequality \eqref{eq:B42},  it suffices to observe that as $\alpha \downarrow 0$,
\[
	G^0_\alpha h(x)\uparrow G^0h(x)=2\int_0^\infty (x\wedge y) \cdot h(y)\dd y; 
\]
see \cite[pp 45, (30)]{CZ95}. 

(3) These are elementary evaluations
\end{proof}

We are now ready to prove the uniqueness of $(p_1,p_2,p_3,p_4)$. The proof will be divided into several steps for clarity. 

\begin{proposition}\label{PROA2}
If \eqref{eq:A3} holds for all $h\in C_0([0,\infty))$ and $\alpha>0$, then $(\tilde{p}_1,\tilde{p}_2,\tilde{p}_3,\tilde{p}_4)=(p_1, p_2, p_3, p_4)$.
\end{proposition}
\begin{proof}
\emph{Step 1: $p_1=0$ (resp. $p_3=0$) if and only if $\tilde{p}_1=0$ (resp. $\tilde{p}_3=0$).} 

Take a positive function $h\in C_c([0,\infty))$ such that $h(0)>0$. Applying \eqref{eq:A3} to this function and letting $\alpha \downarrow 0$,  we obtain
\[
	\begin{aligned}
		&\frac{2p_2\int_{(0,\infty)}h(x)\dd x+p_3h(0)+\int_{(0,\infty)}G^0 h(x)p_4(\dd x)}{p_1} \\
		&\qquad\qquad  =\frac{2\tilde p_2\int_{(0,\infty)}h(x)\dd x+\tilde p_3h(0)+\int_{(0,\infty)}G^0 h(x)\tilde p_4(\dd x)}{\tilde p_1}.
		\end{aligned}
\]
According to the constraint condition \eqref{eq:38}, the numerators on both sides are finite and strictly positive (finiteness is ensured by \eqref{eq:B42}).  Hence, $p_1=0$ if and only if $\tilde{p}_1=0$. 

To establish the equivalence between  $p_3=0$ and $\tilde{p}_3=0$, we argue by contradiction.  Suppose $p_3=0, \tilde{p}_3>0$. Consider first the case $p_2=0$.  Then,  by  \eqref{eq:38},  we must have $|p_4|=\infty$. Choosing $h\in C_c((0,\infty))$ such that $\int_0^\infty h(x)p_4(\dd x)>0$, we obtain
\[
		\frac{\int_{(0,\infty)}\alpha G^0_\alpha h(x)p_4(\dd x)}{\frac{p_1}{\alpha}+\int_{(0,\infty)}\frac{1-e^{-\sqrt{2\alpha}x}}{\alpha}p_4(\dd x)}  =\frac{2\tilde p_2\int_{(0,\infty)}\alpha e^{-\sqrt{2\alpha}x}h(x)\dd x+\int_{(0,\infty)}\alpha G^0_\alpha h(x)\tilde p_4(\dd x)}{\frac{\tilde p_1}{\alpha}+\frac{\sqrt{2\alpha}\tilde p_2}{\alpha}+ \tilde p_3+\int_{(0,\infty)}\frac{1-e^{-\sqrt{2\alpha}x}}{\alpha}\tilde p_4(\dd x)}.
\]
As $\alpha \uparrow \infty$,  the left hand side diverges to $\infty$,  while the right hand side converges to 
\[
	\frac{\int_0^\infty h(x)\tilde{p}_4(\dd x)}{\tilde{p}_3}<\infty,
\]
which is a contradiction. 

Now consider the case $p_2>0$. If $|p_4|>0$, the same contradiction follows by the same argument. If instead $|p_4|=0$,  then for all $h\in C_c((0,\infty))$,  
\[
\begin{aligned}
		&\frac{2p_2\int_{(0,\infty)}\alpha^{3/2} e^{-\sqrt{2\alpha}x}h(x)\dd x}{\frac{p_1}{\sqrt{\alpha}}+\sqrt{2}p_2} \\
		&\qquad\qquad  =\frac{2\tilde p_2\int_{(0,\infty)}\alpha e^{-\sqrt{2\alpha}x}h(x)\dd x+\int_{(0,\infty)}\alpha G^0_\alpha h(x)\tilde p_4(\dd x)}{\frac{\tilde p_1}{\alpha}+\frac{\sqrt{2\alpha}\tilde p_2}{\alpha}+ \tilde p_3+\int_{(0,\infty)}\frac{1-e^{-\sqrt{2\alpha}x}}{\alpha}\tilde p_4(\dd x)}.
\end{aligned}
\]
Taking $\alpha\rightarrow \infty$, we deduce $\int_0^\infty h(x)\tilde{p}_4(\dd x)=0$.  Since this holds for all $h\in C_c((0,\infty))$,  it follows that $|\tilde{p}_4|=0$. Substituting $|p_4|=|\tilde{p}_4|=0$ into \eqref{eq:A3},  we obtain
\[
		\frac{2p_2\int_{(0,\infty)} e^{-\sqrt{2\alpha}x}h(x)\dd x}{p_1+\sqrt{2\alpha}p_2}  =\frac{2\tilde p_2\int_{(0,\infty)} e^{-\sqrt{2\alpha}x}h(x)\dd x+\tilde p_3 h(0)}{p_1+\sqrt{2\alpha}\tilde p_2+\alpha \tilde p_3}.
\]
To reach a contradiction, it suffices to find a function $h\in C_0([0,\infty))$ such that 
\[
	\int_0^\infty e^{-\sqrt{2\alpha}x}h(x)\dd x=0,\quad h(0)\neq 0.
\]
Such a function indeed exists. For instance, take $h_1\in C_0([0,\infty))$ with $h_1(0)\neq 0$ and a positive function $h_2\in C_c((0,\infty))$.  If  
\[
\gamma_1:=\int_{(0,\infty)} e^{-\sqrt{2\alpha}x}h_1(x)\dd x=0, 
\]
then $h_1$ is the desired function.  If $\gamma_1\neq 0$,  define
\[
	h:=h_1-\frac{\gamma_1}{\gamma_2}h_2,
\]
where $\gamma_2:=\int_{(0,\infty)} e^{-\sqrt{2\alpha}x}h_2(x)\dd x>0$. 
Then $\int_{(0,\infty)} e^{-\sqrt{2\alpha}x} h(x)\dd x = 0$, and $h(0) = h_1(0) \neq 0$, as desired.

\emph{Step 2: If $p_1=p_2=p_3=\tilde{p}_1=\tilde{p}_2=\tilde{p}_3=0$, then $p_4=\tilde{p}_4$.}

Under this condition,  we have
\[
\begin{aligned}
		\frac{\int_{(0,\infty)}G^0_\alpha h(x)p_4(\dd x)}{\int_{(0,\infty)}\left(1-e^{-\sqrt{2\alpha}x}\right)p_4(\dd x)} =\frac{\int_{(0,\infty)}G^0_\alpha h(x)\tilde p_4(\dd x)}{\int_{(0,\infty)}\left(1-e^{-\sqrt{2\alpha}x}\right)\tilde p_4(\dd x)}
\end{aligned}
\]
for any $\alpha>0$ and $h\in C_0([0,\infty))$. Let $c_1$ and $\tilde{c}_1$ denote the denominators on the left and right hand sides, respectively, when $\alpha = 1$.  Then
\begin{equation}\label{eq:A55}
	\int_0^\infty G^0_1 h(x)p_4(\dd x)=c\int_0^\infty G^0_1 h(x)\tilde p_4(\dd x)
\end{equation}
for all $h\in C_0([0,\infty))$, where $c:=c_1/\tilde{c}_1$. For any $\beta>0$,  let $h=G^0_\beta h_1$ with $h_1\in C_0([0,\infty))$.  By the resolvent equation,
\[
	G^0_1h_1-G^0_\beta h_1=(\beta-1)G^0_1G^0_\beta h_1=(\beta-1)G^0_1 h. 
\]
Since $h, h_1\in C_0([0,\infty))$, it follows from \eqref{eq:A55} that
\[
	\int_0^\infty \beta G^0_\beta h_1(x) p_4(\dd x)=c\int_0^\infty \beta G^0_\beta h_1(x)\tilde p_4(\dd x),\quad \forall \beta>0, h_1\in C_0([0,\infty)). 
\]
Now take any $h_1\in C_c((0,\infty))$ and let $\beta\uparrow \infty$.  By the first statement of Lemma~\ref{LMA1} and the dominated convergence theorem,  we obtain
\[
	\int_0^\infty h_1(x)p_4(\dd x)=c\int_0^\infty h_1(x)\tilde{p}_4(\dd x),\quad \forall h_1\in C_c((0,\infty)). 
\]
Therefore, $p_4=c\tilde{p}_4$. In view of \eqref{eq:330}, we must have $c=1$, so $p_4=\tilde{p}_4$. 


\emph{Step 3: Prove the conclusion for the case $p_3=\tilde{p}_3=0$.}

We first show that in this case, $p_2=0$ is equivalent to $\tilde{p}_2=0$. Suppose, for the sake of contradiction, that $p_2=0$ (which implies $|p_4|=\infty$) but $\tilde{p}_2>0$. Taking $h\in C_c((0,\infty))$ such that $\int_0^\infty h(x)p_4(\dd x)>0$, we have
\[
		\frac{\int_{(0,\infty)}\alpha G^0_\alpha h(x)p_4(\dd x)}{\frac{p_1}{\sqrt{\alpha}}+\int_{(0,\infty)}\frac{1-e^{-\sqrt{2\alpha}x}}{\sqrt{\alpha}}p_4(\dd x)} =\frac{2\tilde p_2\int_{(0,\infty)}\alpha e^{-\sqrt{2\alpha}x}h(x)\dd x+\int_{(0,\infty)}\alpha G^0_\alpha h(x)\tilde p_4(\dd x)}{\frac{\tilde p_1}{\sqrt{\alpha}}+\sqrt{2}\tilde p_2+\int_{(0,\infty)}\frac{1-e^{-\sqrt{2\alpha}x}}{\sqrt{\alpha}}\tilde p_4(\dd x)}.
\]
Letting $\alpha\rightarrow \infty$ leads to a contradiction.

Now assume that $p_2=\tilde{p}_2=0$. The condition \eqref{eq:38} indicates that $|p_4|=|\tilde{p}_4|=\infty$.  If $p_1=\tilde p_1=0$, then the conclusion follows from the second step. If $p_1,\tilde{p}_1>0$, then we have
\[
\begin{aligned}
		\frac{\int_{(0,\infty)}G^0_\alpha h(x)p_4(\dd x)}{p_1+\int_{(0,\infty)}\left(1-e^{-\sqrt{2\alpha}x}\right)p_4(\dd x)} =\frac{\int_{(0,\infty)}G^0_\alpha h(x)\tilde p_4(\dd x)}{\tilde p_1+\int_{(0,\infty)}\left(1-e^{-\sqrt{2\alpha}x}\right)\tilde p_4(\dd x)}. 
\end{aligned}
\]
We can use the same argument as in the second step to conclude that $p_4=c\tilde{p}_4$ for the constant
\[
c=\frac{p_1+\int_{(0,\infty)}\left(1-e^{-\sqrt{2\alpha}x}\right)p_4(\dd x)}{\tilde p_1+\int_{(0,\infty)}\left(1-e^{-\sqrt{2\alpha}x}\right)\tilde p_4(\dd x)}=\frac{p_1+c\int_{(0,\infty)}\left(1-e^{-\sqrt{2\alpha}x}\right)\tilde p_4(\dd x)}{\tilde p_1+\int_{(0,\infty)}\left(1-e^{-\sqrt{2\alpha}x}\right)\tilde p_4(\dd x)}.
\]
Thus $p_1=c\tilde{p}_1$, and because of \eqref{eq:330}, we must have $c=1$. 

It remains to consider the case where $p_2,\tilde{p}_2>0$. We have
\begin{equation}\label{eq:A66}
	\begin{aligned}
		&\frac{2p_2\int_{(0,\infty)}\alpha e^{-\sqrt{2\alpha}x}h(x)\dd x+\int_{(0,\infty)}\alpha G^0_\alpha h(x)p_4(\dd x)}{\frac{p_1}{\sqrt{\alpha}}+\sqrt{2}p_2+\int_{(0,\infty)}\frac{1-e^{-\sqrt{2\alpha}x}}{\sqrt{\alpha}}p_4(\dd x)} \\
		&\qquad\qquad  =\frac{2\tilde p_2\int_{(0,\infty)}\alpha e^{-\sqrt{2\alpha}x}h(x)\dd x+\int_{(0,\infty)}\alpha G^0_\alpha h(x)\tilde p_4(\dd x)}{\frac{\tilde p_1}{\sqrt{\alpha}}+\sqrt{2}\tilde p_2+\int_{(0,\infty)}\frac{1-e^{-\sqrt{2\alpha}x}}{\sqrt{\alpha}}\tilde p_4(\dd x)}.
\end{aligned}
\end{equation}
Taking $h\in C_c((0,\infty))$ and letting $\alpha\uparrow \infty$, we obtain 
\[
\frac{1}{p_2}\int_0^\infty h(x)p_4(\dd x)=\frac{1}{\tilde p_2}\int_0^\infty h(x)\tilde p_4(\dd x),\quad \forall h\in C_c((0,\infty)).
\]
Therefore, $p_4=c\tilde{p}_4$ for $c=p_2/\tilde{p}_2>0$. Substituting these two relations into \eqref{eq:A66}, we further obtain $p_1=c\tilde{p}_1$. Finally, by virtue of \eqref{eq:330}, it follows that $c=1$. 

\emph{Step 4: Prove the conclusion for the case $p_3,\tilde{p}_3>0$.}

Let $h \in C_c((0,\infty))$ in \eqref{eq:A3}. After multiplying the numerators on both sides by $\alpha$ and dividing the denominators by $\alpha$, we let $\alpha \to \infty$. Under this limiting procedure, we conclude that
\[
	\frac{1}{p_3}\int_0^\infty h(x)p_4(\dd x)=\frac{1}{\tilde p_3}\int_0^\infty h(x)\tilde p_4(\dd x),\quad \forall h\in C_c((0,\infty)).
\]
This implies that $p_4=c\tilde{p}_4$ for $c=p_3/\tilde{p}_3>0$. 

Let $q_\alpha:=p_1+\sqrt{2\alpha}p_2$ and $\tilde q_\alpha:=\tilde p_1+\sqrt{2\alpha}\tilde p_2$. If $|p_4|=|\tilde p_4|=0$, then by comparing the coefficients of $h(0)$ and $\int_0^\infty e^{-\sqrt{2\alpha}x}h(x)\dd x$,  we obtain
\[
	\tilde p_3 q_\alpha=c\tilde{p}_3 \tilde{q}_\alpha,\quad \forall \alpha>0,
\]
which immediately implies $p_1=c\tilde{p}_1$ and $p_2=c\tilde{p}_2$.  Hence, $c=1$, and the conclusion follows. 

Finally, consider the case where $|\tilde{p}_4|>0$. Substituting $p_4=c\tilde{p}_4$ and  $p_3=c\tilde{p}_3$ into \eqref{eq:A3}, and taking $h\in C_c((0,\infty))$ such that $\int_0^\infty h(x)\tilde{p}_4(\dd x)\neq 0$, we obtain
\[
\begin{aligned}
&\left(\frac{q_\alpha}{\sqrt{\alpha}}-c\frac{\tilde{q}_\alpha}{\sqrt{\alpha}}\right)\int_0^\infty \alpha G^0_\alpha h(x)\tilde{p}_4(\dd x)=2\left(\int_0^\infty \alpha^{3/2} e^{-\sqrt{2\alpha}x}h(x)\dd x \right)\cdot  \\
&\qquad \qquad \left[(p_2\frac{\tilde q_\alpha}{\alpha}-\tilde{p}_2\frac{q_\alpha}{\alpha})+(p_2-c\tilde{p}_2)\left(\tilde{p}_3+\int_0^\infty \frac{1-e^{-\sqrt{2\alpha}x}}{\alpha}\tilde{p}_4(\dd x)\right)\right].
\end{aligned}\]
Letting $\alpha\rightarrow \infty$,  we deduce that $p_2=c\tilde{p}_2$.  It then follows readily that $p_1=c\tilde{p}_1$. Therefore, the conclusion is established.
\end{proof}
}

\section{Proof of Theorem~\ref{THMA4}}\label{APPA}

\begin{proof}
For convenience, we will consider $\hat{X}$ and $\hat{X}^{(n)}$ instead of $X$ and $X^{(n)}$.  Denote by $\hat{\zeta}^{(n)}$ the lifetime of $\hat{X}^{(n)}$.  According to the definition of the $C(\hat{\eta}^{(n)}_m,\hat{\sigma}^{(n)}_m)$-transformation,   it is straightforward to observe that
\[
	\hat{\zeta}^{(1)}\leq \hat\zeta^{(2)}\leq \cdots\leq \hat{\zeta}^{(n)}\leq \cdots \leq \hat{\zeta}.  
\]
Let $\hat{\zeta}_\infty:=\lim_{n\rightarrow \infty}\hat{\zeta}^{(n)}$ ($\leq \hat{\zeta}$).  

Firstly,  we show that there exists an integer $N$ such that for any $n\geq N$,
\begin{equation}\label{eq:610}
	\bP_{\hat{c}_i}\left(\hat{\sigma}^{(n)}_m<\infty\right)=1,\quad \forall m\geq 1,  i\in \bN.  
\end{equation}
It suffices to prove \eqref{eq:610} for some integer $n=N$,  due to $\hat{\sigma}^{(n+1)}_m\leq \hat{\sigma}^{(n)}_m$ for all $n\in \bN$.  
To do this,  note that we can find some $t\geq 0$ and $N\geq 1$ such that
\begin{equation}\label{eq:611}
	\bP_0(\hat X_t\geq \hat{c}_N \text{ or }\hat{X}_t=\partial)>0,
\end{equation}
because otherwise 
\[
	1=\bP_0\left( \bigcap_{t\in \mathbb Q_+,n\geq 1}\left\{\hat{X}_t<\hat{c}_n \right\} \right)=\bP_0\left( \hat{X}_t=0,\forall t\geq 0\right),
\]
where $\mathbb Q_+$ is the set of all positive rational numbers. 
To prove \eqref{eq:610} for $n=N$ with $N$ in \eqref{eq:611},  consider the case $m=1$. We note that $\hat{\sigma}^{(N)}_1=\hat{\sigma}^{(N)}\circ \theta_{\hat{\eta}}$,  and {\color{blue}since the boundary point is regular, it holds that} $\hat{\eta}<\infty$,  $\bP_{\hat{c}_i}$-a.s.  It follows that
\[
	\bP_{\hat{c}_i}\left(\hat{\sigma}^{(N)}_1<\infty\right)=\bP_0\left(\hat{\sigma}^{(N)}<\infty\right)\geq \bP_0(\hat X_t\geq \hat{c}_N \text{ or }\hat{X}_t=\partial)>0.
\]
Thus, there exists a constant $T>0$ such that
\[
	\bP_{\hat{c}_i}\left(\hat{\sigma}^{(N)}_1\geq T\right)=\bP_{0}\left(\hat{\sigma}^{(N)}\geq T\right)=:\alpha <1.
\]
For $k\geq 2$,  it follows from the strong Markov property that
\[
\begin{aligned}
	\bP_{\hat{c}_i}\left(\hat{\sigma}^{(N)}_1\geq kT\right)&=\bP_{\hat{c}_i}\left(\hat{\sigma}^{(N)}_1\circ \theta_{(k-1)T}\geq T;\hat{\sigma}^{(N)}_1\geq (k-1)T\right) \\
	&=\bP_{\hat{c}_i}\left(\bP_{\hat{X}_{(k-1)T}}\left(\hat{\sigma}^{(N)}_1\geq T\right);\hat{\sigma}^{(N)}_1\geq (k-1)T\right) \\
	&\leq \alpha  \bP_{\hat{c}_i}\left(\hat{\sigma}^{(N)}_1\geq (k-1)T\right).
\end{aligned}\]
By induction,  we get $\bP_{\hat{c}_i}\left(\hat{\sigma}^{(N)}_1\geq kT\right)\leq \alpha^k$ for all $k\geq 0$.  Hence,
\[
	\bE_{\hat{c}_i}\hat{\sigma}^{(N)}_1\leq T\sum_{k\geq 0}\bP_{\hat{c}_i}(\hat{\sigma}^{(N)}_1\geq kT)\leq T\sum_{k\geq 0}\alpha^k<\infty. 
\]
This indicates \eqref{eq:610} for the case $m=1$.  For general $m\geq 2$,  \eqref{eq:610} can also be established by using the strong Markov property.  We provide details for the case $m=2$,  with the other cases handled by induction.  In fact,  we have
\[
\begin{aligned}
	\bP_{\hat{c}_i}\left(\hat{\sigma}^{(N)}_2<\infty\right)&=\bP_{\hat{c}_i}\left(\hat{\sigma}^{(N)}_2<\infty; \hat X_{\hat{\sigma}_1^{(N)}}\in E\right)+\bP_{\hat{c}_i}\left(\hat{\sigma}^{(N)}_2<\infty; \hat X_{\hat{\sigma}_1^{(N)}}=\partial\right) \\
	&=\bP_{\hat{c}_i}\left(\bP_{\hat X_{\hat{\sigma}_1^{(N)}}}\left(\hat{\sigma}^{(N)}_1<\infty\right); \hat X_{\hat{\sigma}_1^{(N)}}\in E\right)+\bP_{\hat{c}_i}\left(\hat X_{\hat{\sigma}_1^{(N)}}=\partial\right) \\
	&=\bP_{\hat{c}_i}\left(\hat X_{\hat{\sigma}_1^{(N)}}\in E\right)+\bP_{\hat{c}_i}\left(\hat X_{\hat{\sigma}_1^{(N)}}=\partial\right)=1. 
\end{aligned}\]

Next,  we prove that for all $i\in \bN$,
\[
	\hat{\zeta}_\infty=\hat{\zeta},\quad \bP_{\hat{c}_i}\text{-a.s.}
\]
 If $\hat{\zeta}(\omega)<\infty$,  then according to the definition of $\hat{X}^{(n)}$,  we have
 \[
 	\hat{\zeta}^{(n)}(\omega)\geq \hat{\zeta}(\omega)-\left|\{0\leq t<\hat{\zeta}(\omega):\hat{X}_t(\omega)<\hat{c}_n\} \right|,
 \]
 where $|\cdot|$ stands for the Lebesgue measure of the time set. Note that $$\left|\{0\leq t<\hat{\zeta}(\omega): \hat{X}_t(\omega)=0\}\right|=0,  \quad\bP_{\hat{c}_i}\text{-a.s.} $$
 It follows that 
 \[
 \begin{aligned}
 \hat{\zeta}_\infty(\omega)&=\lim_{n\rightarrow \infty}\hat{\zeta}^{(n)}(\omega)\geq \hat{\zeta}(\omega)-\lim_{n\rightarrow \infty}\left|\{0\leq t<\hat{\zeta}(\omega):\hat{X}_t(\omega)<\hat{c}_n\} \right| \\
 &=\hat{\zeta}(\omega)-\lim_{n\rightarrow \infty}\left|\{0\leq t<\hat{\zeta}(\omega): \hat{X}_t(\omega)=0\}\right|=\hat{\zeta}(\omega).
 \end{aligned}\]
 Thus, $\hat{\zeta}_\infty(\omega)=\hat{\zeta}(\omega)$ whenever $\hat{\zeta}(\omega)<\infty$.  For the case $\hat{\zeta}(\omega)=\infty$,  we prove that $\hat{\zeta}^{(n)}(\omega)=\infty$ for all $n\geq N$ with $N$ satisfying \eqref{eq:610}.  It suffices to consider the case where $\hat{\lambda}^{(n)}(\{\partial\})>0$,  because otherwise $\hat{X}^{(n)}$ is honest by Theorem~\ref{THMB1}.  If $\hat{\zeta}^{(n)}<\infty$,  then the total duration discarded from $\hat{X}_\cdot(\omega)$ in constructing $\hat{X}^{(n)}_\cdot(\omega)$ must be  infinite.  It follows from \eqref{eq:610} that
 \[
 	\{\hat{\zeta}^{(n)}<\infty,\hat{\zeta}=\infty\}\subset \{\hat{X}_{\hat{\sigma}^{(n)}_m}\in E,\forall m\geq 1\}.
 \]
 By the strong Markov property,  for any integer $M\geq 1$,
 \[
 \bP_{\hat{c}_i}\left(\hat{X}_{\hat{\sigma}^{(n)}_m}\in E,1\leq m\leq M\right)=(1-\hat{\lambda}^{(n)}(\{\partial\}))^M.
 \]
 Therefore, 
 \[
 	\bP_{\hat{c}_i}\left(\hat{\zeta}^{(n)}<\infty,\hat{\zeta}=\infty\right)\leq \lim_{M\rightarrow\infty}(1-\hat{\lambda}^{(n)}(\{\partial\}))^M=0.
 \]
 This implies $\hat{\zeta}_\infty(\omega)=\hat{\zeta}(\omega)$ for $\hat{\zeta}(\omega)=\infty$.

We are now prepared to prove \eqref{eq:612}.   Let us first establish
\begin{equation}\label{eq:A4}
	\lim_{n\rightarrow \infty}\hat{X}^{(n)}_t=\hat{X}_t,\quad \bP_{\hat{c}_i}\text{-a.s.}
\end{equation}
for a fixed $t\geq 0$.  
For $\omega\in \Omega$,  let $\rho^{(n)}_t(\omega)$ represent the total duration discarded from the path of $\hat{X}_\cdot(\omega)$ in constructing $\hat{X}^{(n)}_s(\omega)$ for $0\leq s\leq t$, namely,  $\hat{X}^{(n)}_t(\omega)=\hat{X}_{t+\rho^{(n)}_t(\omega)}(\omega)$. Obviously,  $\rho^{(n)}_t(\omega)$ is decreasing as $n\rightarrow \infty$.  If $t\geq \hat{\zeta}(\omega)=\hat{\zeta}_\infty(\omega)$,  then 
\[
	\hat{X}_t(\omega)=\hat{X}^{(n)}_t(\omega)=\partial,\quad \forall n\in \bN.
\]
If $t<\hat{\zeta}(\omega)=\hat{\zeta}_\infty(\omega)$,  then there exists an integer $n_0$ (which may depend on $\omega$) such that 
\[
	t<\hat{\zeta}^{(n_0)}(\omega)\leq \hat{\zeta}^{(n)}(\omega),\quad \forall n\geq n_0.
\]
This implies
\[
\rho^{(n)}_t(\omega)\leq \rho^{(n_0)}_t(\omega)<\infty. 
\]
We have
\[
\begin{aligned}
	\rho^{(n)}_t(\omega)&\leq \left|\{0\leq s\leq t+\rho^{(n)}_t(\omega): \hat{X}_s(\omega)<\hat{c}_n\}\right| \\
		&\leq \left|\{0\leq s\leq t+\rho^{(n_0)}_t(\omega): \hat{X}_s(\omega)<\hat{c}_n\}\right|.
\end{aligned}\]
Thus 
\[
\begin{aligned}
	\lim_{n\rightarrow \infty} \rho^{(n)}_t(\omega)&\leq \lim_{n\rightarrow \infty}\left|\{0\leq s\leq t+\rho^{(n_0)}_t(\omega): \hat{X}_s(\omega)<\hat{c}_n\}\right|  \\
	&=\left|\{0\leq s\leq t+\rho^{(n_0)}_t(\omega): \hat{X}_s(\omega)=0\}\right|=0.
\end{aligned}\]
By the right continuity of $\hat{X}$,  we obtain
\[
	\lim_{n\rightarrow \infty}\hat{X}^{(n)}_t(\omega)=\lim_{n\rightarrow \infty} \hat{X}_{t+\rho^{(n)}_t(\omega)}(\omega)=\hat{X}_t(\omega),
\]
whenever $t<\hat{\zeta}(\omega)$.  Therefore,  \eqref{eq:A4} is established.  
Finally,  it remains to observe that $\rho^{(n)}_{t'}(\omega)\leq \rho^{(n)}_t(\omega)$ for all $t'\leq t$,  which,  in conjunction with \eqref{eq:A4}, leads to
\begin{equation}\label{eq:A5}
	\bP_{\hat{c}_i}\left(\lim_{n\rightarrow \infty}\hat{X}^{(n)}_{t'}=\hat{X}_{t'},\quad \forall 0\leq t'\leq t \right)=1
\end{equation}
for any fixed $t\geq 0$.  Clearly,  \eqref{eq:A5} implies the desired \eqref{eq:612}.
\end{proof}

\section*{Acknowledgement}

The author wishes to thank Professor Patrick J.  Fitzsimmons from the University of California, San Diego, for his invaluable suggestion, which inspired the exploration of this topic. {\color{blue}The author gratefully acknowledges the referee for a thorough reading of the manuscript and for offering numerous valuable suggestions. These insightful comments have substantially improved the quality of the paper. }

\bibliographystyle{plain} 
\bibliography{PathBD}

\end{document}